\documentclass[12pt,reqno]{amsart}

\usepackage{mathtools}
\usepackage{dsfont}
\usepackage{amsthm}
\usepackage[hidelinks]{hyperref}
\hypersetup{linktoc=all}
\usepackage[T1]{fontenc}
\usepackage[shortlabels]{enumitem} 
\usepackage{amssymb}
\DeclareMathSymbol{\sminus}{\mathbin}{AMSa}{"39}
\usepackage{bm} 
\usepackage{doi}
\usepackage{graphicx} 
\usepackage{algorithm}
\usepackage{algorithmic} 

\newtheorem{thm}[equation]{Theorem}
\newtheorem{cor}[equation]{Corollary}
\newtheorem{lem}[equation]{Lemma}

\theoremstyle{definition}
\newtheorem{rem}[equation]{Remark}
\newtheorem{rems}[equation]{Remarks}
\newtheorem{defn}[equation]{Definition}
\newtheorem{defns}[equation]{Definitions}
\newtheorem{exam}[equation]{Example}
\newtheorem{pgraph}[subsection]{{\!\!}}
\newtheorem{subpgraph}[subsubsection]{{\!\!}}

\newcommand{\nin}{\noindent}
\newcommand{\fn}{\footnote}
\newcommand{\n}{\textbf}
\newcommand{\e}{\emph}
\newcommand{\ov}{\overline}
\newcommand{\seq}{\subseteq}
\newcommand{\ch}{{\rm char}} 

\newcommand{\N}{\mathds{N}}
\newcommand{\Z}{\mathds{Z}}
\newcommand{\Q}{\mathds{Q}}

\newcommand{\Co}{\mathds{C}} 
\newcommand{\Mat}{\text{M}}
\newcommand{\C}{{\mathcal{C}}} 
\newcommand{\sgn}{{\rm sign}}
\newcommand{\Ann}{{\rm Ann}} 
\newcommand{\vp}{\varphi}

\newcommand{\eq}[1]{\begin{align*}#1\end{align*}} 
\newcommand{\eqnum}[1]{\begin{align}#1\end{align}} 

\newcommand{\pma}{\begin{pmatrix}}
\newcommand{\epma}{\end{pmatrix}}
\newcommand{\vma}{\begin{vmatrix}}
\newcommand{\evma}{\end{vmatrix}}
\newcommand{\tr}{{\rm tr}} 

\newcommand{\enum}{\begin{enumerate}[leftmargin=*,itemsep=-0.5ex]}
\newcommand{\enumi}{\begin{enumerate}[{\rm i)},leftmargin=*,itemsep=-0.5ex]}
\newcommand{\enuma}{\begin{enumerate}[{\rm a)},leftmargin=*,itemsep=-0.5ex]}
\newcommand{\eenum}{\end{enumerate}}
\newcommand{\ite}{\begin{itemize}[leftmargin=*,itemsep=-0.5ex]}
\newcommand{\eite}{\end{itemize}}

\begin{document}

\title[Tridiagonal $k$-Toeplitz matrices]{The determinant, spectral properties, and inverse of a tridiagonal $k$-Toeplitz matrix\\over a commutative ring}

\dedicatory{Dedicated to the memory of our friend José Carlos Petronilho}

\author[J. Brox]{Jose Brox$^*$}
\address{University of Coimbra, CMUC, Department of Mathematics, 3004-504 Coimbra, Portugal}
\thanks{This work was partially supported by the Centre for Mathematics of the University of Coimbra - UIDB/00324/2020, funded by the Portuguese Government through FCT/MCTES. The first author was supported by the Portuguese Government through the FCT grant SFRH/BPD/118665/2016 (FCT/Centro 2020/Portugal 2020/ESF)}
\email[*corresponding author]{josebrox@mat.uc.pt}

\author[H. Albuquerque]{Helena Albuquerque}
\address{University of Coimbra, Department of Mathematics, 3004-504 Coimbra, Portugal}
\email{lena@mat.uc.pt}

\keywords{Tridiagonal $k$-Toeplitz matrix, continuant, determinant, characteristic polynomial, eigenvector, inverse, continuant polynomial, generalized Fibonacci polynomial, computational complexity, commutative ring}
\subjclass[2020]{15B05, 47B36, 15A15, 15A18, 15A09, 65F50, 68W30, 68W40, 47B35, 33C45, 42C05, 39A06}

\begin{abstract}
A square matrix is $k$-Toeplitz if its diagonals are periodic sequences of period $k$. We find universal formulas for the determinant, the characteristic polynomial, some eigenvectors, and the entries of the inverse of any tridiagonal $k$-Toeplitz matrix (in particular, of any tridiagonal matrix) over any commutative unital ring, expressed in terms of the elementary operations of the ring. The results are proven using combinatorial identities and elementary linear algebra. We conduct a complexity analysis of algorithms based on our formulas, showing that they are efficient, and we compare our results favourably with those found in the literature. Concretely, the determinant, the characteristic polynomial, and any entry of the inverse of a tridiagonal $k$-Toeplitz matrix of size $n$ can each be found with $O(\displaystyle\log\frac nk+k)$ operations, while an eigenvector can be determined with $O(n+k)$ operations.
\end{abstract}

\maketitle

\newpage

\bigskip
\bigskip
\setcounter{tocdepth}{3}
\large\tableofcontents\normalsize

\newpage
\section{Introduction}

\subsection{Interest of tridiagonal $k$-Toeplitz matrices}
\mbox{}

Given a commutative unital ring $K$, a tridiagonal matrix $T\in\Mat_n(K)$ is $k$-Toeplitz if the entries along the main diagonal of $T$ and its adjacent diagonals are periodic sequences of period $k$, so that it has the form
\[T=\pma
\bm{a_1} & \bm{b_1} & & & & & & & \\
\bm{c_1} & \ddots & \ddots &&&&&&  \\
&\ddots& \bm{a_k} & \bm{b_k} &&&&& \\
 & \ddots & \bm{c_k} & \bm{a_1} & \bm{b_1} &&&& \\
&&& \bm{c_1} & \ddots &\ddots &&& \\
&&&&\ddots& \bm{a_k} & \bm  {b_k} && \\
&&&&& \bm{c_k} & \bm{a_1} & \ddots \\
&&&&&& \ddots & \ddots &
\epma_{n\times n}.\]
$T$ is called reducible if some element $b_i$ or $c_i$ in the adjacent diagonals is $0$, irreducible otherwise. Observe that general tridiagonal matrices can be treated as tridiagonal $k$-Toeplitz matrices by considering $n\leq k$.

Tridiagonal matrices appear frequently in many areas of pure and applied mathematics (see \cite{Meurant1992}). From the mathematical point of view, tridiagonal matrices appear in problems related to linear recurrence equations of second order, while from the perspective of Physics they appear related to problems comprising a system in series in which each subsystem is affected by (and only by) its immediate left and right subsystems. Consequently, tridiagonal $k$-Toeplitz matrices arise in those contexts when some periodicity of the studied problem or physical system is assumed. For example, tridiagonal matrices appear in trigonometric polynomials problems (\cite{EgervarySzasz1928}), number-theoretical problems involving second-order difference equations (\cite{AndelicFonseca2020}), the discretization of elliptic or parabolic partial differential equations by finite difference methods (\cite{FischerUsmani1969,MattheijSmooke1986,Yamamoto2001}), classical mechanics (\cite{Wittenburg1998}), chain models of quantum physics (\cite{AlvarezNodarsePetronilhoQuintero}), sound propagation theory (\cite{Chandler-WildeGover1989,Chandler-WildeGoverHothersah1986}), telecommunication system analysis (\cite{LuSun}), circuit models of wireless power transfer arrays (\cite{albertofast,MATRIX,AlbertoBrox}), etc. Thus, the determinant, the eigenvalues, or the inverse of the associated matrix may be invoked to solve these problems; and at times, only some specific entries of the inverse are needed (see e.g. \cite{AlbertoBrox}). Moreover, any complex square matrix is similar to a tridiagonal one (as shown for example by Lanczos' algorithm, \cite{Lanczos1950}), and therefore its determinant, spectral properties, and inverse may be computed from those of its associated tridiagonal matrix.

It is therefore natural that both general and $k$-Toeplitz tridiagonal matrices (and variations) have been studied many times independently in the past, in many different research areas, producing a vast literature in which formulas for the determinant, spectral properties and inverse abound; but most often treating specific cases, providing solutions with varying degree of explicitness, presenting unwieldy or inefficient formulas, or using intricate methods in their proofs. In particular, most results resort to division (at the least) and are therefore only suitable over fields. In addition, formulas are invariably offered without any accompanying complexity analysis (of algorithms based on them), which hinders the assessment of their efficiency and the comparison of formulas coming from different works. In contrast, in this paper we treat the problems in full generality: over commutative unital rings, for all pairs $(n,k)$, and for all possible entries; we provide explicit solutions (written only in terms of elementary operations of the ring) through accessible and efficient universal formulas, proved by elementary combinatorial and linear algebraic techniques; and we produce concrete algorithms based on our formulas, together with their complexity analysis.

\subsection{Previous literature}
\mbox{}

Hereunder we briefly review some of the most significant previous works. In his Treatise on the Theory of Determinants of 1882 (\cite{Muir1882}), Muir already studies the determinants of tridiagonal matrices, which he calls \e{continuants} due to their relation to continued fractions (\cite[Example 1 in p.157]{Muir1882}), establishes several interesting identities, and devises a recursive procedure to write a continuant in \e{non-determinant form} (in his own words), that is, as a polynomial expressed in the canonical basis of the underlying polynomial ring. In the second and much enlarged edition of Muir's Treatise (\cite{MuirMetzler1933}), prepared by Metzler up to 1928, the non-determinant form of a continuant is given by a non-recursive rule for writing each monomial (\cite[Item 545]{MuirMetzler1933}). Metzler also provides a polynomial formula for the elements of the adjugate matrix of a tridiagonal matrix in terms of smaller continuants (\cite[Item 555]{MuirMetzler1933}), from which a formula for the elements of the inverse readily follows. Their proofs are quite short and simple, and although they lack some rigor for nowadays standards, their combinatorial nature makes the formulas valid over any commutative unital ring. Mallik in \cite{Mallik2001} in essence rediscovered the non-determinant form of a continuant (while expressing it more formally) and Metzler's formula for the elements of the inverse of a tridiagonal matrix, with a more convoluted elaboration, valid only for irreducible tridiagonal matrices over fields (due to divisions). The same restrictions apply to Lewis' formula from 1982 for the elements of the inverse (\cite{Lewis1982}), based on computing two recurrence equations with many divisions, which is nevertheless different from (and slightly less efficient than) Metzler's.

Presumably, the first work studying tridiagonal $k$-Toeplitz matrices is Egerváry and Szász's \cite{EgervarySzasz1928} of 1928, in which they determined, in terms of graph theory, the characteristic polynomial and inverse powers of some symmetric Toeplitz (i.e., $1$-Toeplitz) complex matrices, particular examples which have been rediscovered with less elegant methods many times since (see \cite{FonsecaKowalenkoLosonczi2020} for more information). Rózsa in 1969 (\cite{Rozsa1969}) published a formula for the determinant of any irreducible and symmetric tridiagonal $k$-Toeplitz matrix over a field, which he extended to non-symmetric complex matrices through similarity (involving the use of square roots). Rózsa's formula is very interesting in that it is expressed by means of Chebyshev polynomials of the second kind. Later in 2005, using tools from the theory of orthogonal polynomials and apparently unaware of Rózsa's work, da Fonseca and Petronilho generalized Rózsa's formula to any complex irreducible tridiagonal $k$-Toeplitz matrix (symmetric or not) and produced another formula for the elements of its inverse, which is in essence Metzler's but with the corresponding continuants expressed by means of Chebyshev polynomials evaluated on some determinants of smaller matrices (the small cases $k=2,3$ having been considered previously with similar methods by the authors in \cite{FonsecaPetronilho2001}). Previously in 1998, Wittenburg in \cite[Section 4]{Wittenburg1998} already produced formulas for the elements of the inverse of any complex tridiagonal $k$-Toeplitz matrix, although in terms of smaller continuants\footnote{Interestingly, Wittenburg is the only author cited in this section that refers to Muir and Metzler's book in his referenced work.} which he did not compute explicitly, and through the somewhat convoluted study of a related recurrence equation.

With regard to spectral properties: Inspired by Rózsa's work, Elsner and Redheffer (\cite{ElsnerRedheffer1967}) studied in 1967 the characteristic polynomial and eigenvectors of complex tridiagonal $k$-Toeplitz matrices in the special cases $n\equiv 0\!\pmod k$ and $n\equiv k-1\!\pmod k$ and showed that, in the latter case, the characteristic polynomial factorizes through a Chebyshev polynomial of the second kind\footnote{This result easily explains the fact that if $F_k$ denotes the $k$th Fibonacci number, starting with $F_0:=0$, and $n$ is divisible by $k$, then $F_n$ is divisible by $F_k$ (\cite[Remark 8]{ElsnerRedheffer1967}).} (another result noted many times in the literature in more or less generality, starting with Egerváry and Szász). The eigenproblem of a complex tridiagonal $k$-Toeplitz matrix for any $n$ was studied in 1994-2000 for $k=2$ by Gover, and for $k=2,3$ by Marcellán and Petronilho (\cite{Gover1994,MarcellanPetronilho1997,MarcellanPetronilho2000}), and a solution for general $k$ was given for irreducible matrices by da Fonseca and Petronilho in the previously cited paper of 2005 (\cite{FonsecaPetronilho2005}).

On the other hand, the first author devised in \cite{AlbertoBrox} an elementary linear algebra algorithm to compute the entries of the inverse of a complex symmetric $k$-Toeplitz matrix with constant upper and lower diagonals, for a fixed $k$; in said algorithm, some determinants of smaller tridiagonal $k$-Toeplitz matrices needed to be computed, what was achieved through the diagonalization of an associated $2\times 2$ matrix. The results in this paper are inspired by this algorithm.

\subsection{Structure of the paper}
\numberwithin{equation}{subsubsection}

\begin{subpgraph}\n{Outline of our general method.} In the first place, in Section \ref{sectiondeterminant} we obtain formulas for the determinant of any tridiagonal $k$-Toeplitz matrix, expressed only in terms of sums and products of the underlying commutative unital ring. Then we write the characteristic polynomial and some eigenvectors (Section \ref{sectionspectralproperties}), and any element of the inverse (Section \ref{sectioninverse}), as functions of determinants of certain tridiagonal $k$-Toeplitz matrices, allowing us to apply our previous formulas and devise specific algorithms, whose complexity we analyze afterwards (Section \ref{complexitysection}). Throughout the paper, we compare our results with those found in the literature, when suited.

For the sake of completeness, and although our formulas work equally for both irreducible and reducible tridiagonal matrices, we study the reducible case separately (with block-triangular matrices techniques) whenever some extra knowledge is to be gained, as this study is not completely trivial (since the matrices are $k$-Toeplitz and we work over a commutative ring). Moreover, we already need to resort to block-triangular matrices when working with the inverse. Elementary notions about these concepts can be found in Section \ref{intromatrices}.

Informally speaking, the theorems and algorithms we develop produce the universal tridiagonal $k$-Toeplitz example, which is free in the underlying ring, the size, the period, and the elements of the diagonals of the matrix. In Section \ref{sectionexamples} we construct two examples which are more concrete: one with fixed size and period, and a completely specific one over $\Z/60\Z$.
\end{subpgraph}

\begin{subpgraph}\n{Determinant.} In rough terms, if the Euclidean division of $n$ by $k$ gives quotient $m$ and remainder $r$, then inside a tridiagonal $k$-Toeplitz matrix of order $n$ we can see $m$ complete ``tridiagonal blocks'' of period $k$ and one last incomplete ``tridiagonal block'' of size $r$ ---a ``tail''. In the proof of the main theorem of this paper, Theorem \ref{determinant}, we manage to decouple the effect of both parts on the determinant of the matrix, with the effect of the complete periods being encoded by generalized Fibonacci polynomials of order $m$ and the effect of the tail being encoded by some polynomials in $2k$ variables that we call continuant polynomials. This decoupling we achieve by solving a second-order linear difference equation with periodic coefficients (which appears from applying Laplace's expansion to the determinant twice) by writing it as a $2\times 2$ matrix difference equation and employing recursion, and then applying induction to show some combinatorial identities that prove the correctness of the associated polynomial formulas. Theorem \ref{determinant} unveils two related formulas for the determinant of a tridiagonal $k$-Toeplitz matrix; we show that one of them generalizes da Fonseca and Petronilho's formula based on Chebyshev polynomials of the second kind (Remark \ref{FonsecaPetronilho}).
\end{subpgraph}

\begin{subpgraph}\n{Generalized Fibonacci polynomials.} One of our main contributions to the subject under study is the application of generalized Fibonacci polynomials, which happen to fill the role, over arbitrary commutative unital rings, that Chebyshev polynomials of the second kind played over the complex numbers in other papers; specifically, the addition of a second variable allows us to remove the divisions and square roots needed together with Chebyshev polynomials, showing that the generalized Fibonacci polynomials are indeed the natural object to consider in this context. We study the properties of generalized Fibonacci polynomials in Section \ref{sectionFibonacci}, notably the following elementary but powerful fact, that we consider of independent interest and which we have not been able to locate in the literature: By the Cayley-Hamilton theorem, the powers of a $2\times2$ matrix are linear combinations of the matrix itself and the identity matrix; those linear combinations are parametrized by generalized Fibonacci polynomials evaluated on its trace and determinant (Lemma \ref{powers}(2)).
\end{subpgraph}

\begin{subpgraph}\n{Continuant polynomials.} We introduce continuant polynomials and their properties in Section \ref{sectioncontinuants}. Continuant polynomials are multivariate polynomials related to the Leibniz expansion of a periodic continuant. They satisfy abstract combinatorial definitions that allow to prove interesting identities between them by induction (Lemma \ref{recurrences}), including some essential recurrence relations. The definitions of continuant polynomials involve some intermediary divisions, which are carried out in rings of rational functions; this feature allows to avoid altogether the divisions appearing in other approaches that resort to recurrence relations.
\end{subpgraph}

\begin{subpgraph}\n{Spectral properties.} The characteristic polynomial of a tridiagonal $k$-Toeplitz matrix over $K$ arises as the determinant of a tridiagonal $k$-Toeplitz matrix over $K[X]$, which we are already able to determine. We show that the factorization property of the characteristic polynomial for the case $n\equiv k-1\!\pmod k$ generalizes to any commutative unital ring (Remark \ref{charpolfactorizesspecialcase}). We also procure a procedure to construct an eigenvector associated to a given eigenvalue in Theorem \ref{eigenvectors}, which significantly generalizes the previously known constructions and sufficient conditions for its existence (see Remark \ref{eigenvectorsgeneralizespreviousresults}).
\end{subpgraph}

\begin{subpgraph}\n{Inverse.} The entries of the inverse are computed from the adjugate matrix in Theorem \ref{inverse}. The associated submatrices of the first minors of a tridiagonal $k$-Toeplitz matrix are block triangular with three diagonal blocks: the middle one is triangular and the two in the extremes are tridiagonal $k$-Toeplitz matrices, allowing to write the cofactor from the product of their determinants, which we are already able to compute.
\end{subpgraph}

\begin{subpgraph}\n{Algorithms and complexity analysis.} After introducing the necessary notions in Section \ref{introcomplexity}, in Section \ref{complexitysection} we conduct a worst-case algebraic complexity analysis of algorithms based on the previously introduced formulas. We compare four main algorithms for computing the determinant, all arising from different ideas used in the proof of Theorem \ref{determinant}. We study algorithms for the case of a general (non-periodic) tridiagonal matrix. We compare the efficiency of our algorithms with others based on Lewis' and da Fonseca and Petronilho's formulas.
We find the following complexities for our algorithms\footnote{For functions $f,g:\N^2\rightarrow\N$, we have $f(n,k)=O(g(n,k))$ when there exist constants $M\in\N$ and $c>0$ such that $|f(n,k)|\leq c|g(n,k)|$ for all $n,k\geq M$.}:

\begin{center}
\begin{tabular}{|c|c|}
  \hline
  \n{Object} & \n{Complexity} \\\hline\hline
  Determinant & $O(21\log_2(n/k)+7k)$ \\\hline
  Char. poly. & $O(21\log_2(n/k)+7k)$ \\\hline
  Eigenvector & $O(6n+k)$ \\\hline
  Inverse entry & $O(68\log_2(n/k)+14k)$ \\\hline
  Full inverse & $O(\frac52n^2+2kn)$ \\
  \hline
\end{tabular}
\end{center}
\end{subpgraph}

\section{Preliminaries and notation}\label{preliminaries}

\numberwithin{equation}{subsection}

\begin{pgraph}\n{Commutative rings.} Throughout this paper let $K$ be any commutative unital ring.  If $K$ is a field, by $\ov K$ we denote an algebraic closure of $K$. An element $a\in K$ is a \e{zero divisor} if there is $0\neq b\in K$ such that $ab=0$; an element which is not a zero divisor is called \e{regular}. Note that $0$ is a zero divisor, and that if $a_1a_2$ is a zero divisor with $a_1,a_2\in K$ then $a_1$ or $a_2$ is a zero divisor. The \e{annihilator} of $a\in K$ is $\Ann(a):=\{b\in K \ | \ ab=0\}$; we have $\Ann(a)=0$ if and only if $a$ is a regular element. The \e{total quotient ring} $Q(K)$ of $K$ is the localization $S^{-1}K$ with $S$ the set of regular elements; it is an injective extension of $K$ in which every regular element is a unit.
\end{pgraph}

\begin{pgraph}\n{Combinatorial objects.} In this paper $\N$ denotes the natural numbers with $0\in\N$, $\N^*$ stands for $\N\setminus\{0\}$, an empty summation yields $0$, and an empty product yields the identity of the ambient ring. Notation $\lfloor\,\cdot\,\rfloor$ stands for the floor function from $\Q$ to $\Z$, and $\binom ij$ with $i,j\in\N$ stands for the image of the corresponding binomial coefficient under the canonical homomorphism from $\Z$ to $K$, understanding $\binom ij=0$ when $i<j$.
For $k\in\N^*$, $\ov a:=(a_1,\ldots,a_k)$ describes a vector of $K^k$, if $x\in K$ then $\ov x:=(x,\ldots,x)\in K^k$, and $\lambda\ov a+\mu\ov b$ with $\lambda,\mu\in K$, $\ov a,\ov b\in K^k$ is the usual linear combination of vectors. Given $\ov a:=(a_1,\ldots,a_k)\in K^k$, we extend it periodically by defining $a_{i+k}:=a_i$ for $i\in\N^*$.
By $S_k$ we denote the symmetric group on $k$ elements acting on $K^k$ and by $\sigma_s\in S_k$, $s\in\N$, the $j$th cyclic permutation to the left, so that $\sigma_0(x_1,\ldots,x_k)=(x_1,\ldots,x_k)$, $\sigma_1(x_1,\ldots,x_k)=(x_2,\ldots,x_k,x_1)$, $\sigma_k=\sigma_0$, etc.
\end{pgraph}

\begin{pgraph}\n{Matrices.}\label{intromatrices} For $n\in\N^*$, $\Mat_n(K)$ denotes the ring of square matrices of order $n$ over $K$, $I_n\in\Mat_n(K)$ denotes the identity matrix, and $\tr(A),\det(A),A^T$ respectively denote the trace, the determinant, and the transpose of matrix $A\in\Mat_n(K)$, which is invertible over $K$ if and only if $\det(A)$ is a unit of $K$ (\cite[Corollary 2.21]{Brown1993}).

Consider the matrix $A\in\Mat_n(K)$, $A=(a_{ij})_{i,j=1}^n$. $A$ is \e{tridiagonal} if $a_{ij}=0$ for all $i,j$ such that $|i-j|\geq2$. The \e{upper main diagonal} and \e{lower main diagonal} of $A$ are, respectively, the vectors $(a_{12},\ldots,a_{n-1,n})$ and $(a_{21},\ldots,a_{n,n-1})$. A tridiagonal matrix is \e{irreducible} if its upper and lower main diagonals have no zeros, \e{reducible} otherwise.
Given $k\in\N^*$, the matrix $A$ is \e{$k$-Toeplitz} if
\[a_{i+k,j+k}=a_{ij}\text{ for all }1\leq i,j\leq n-k.\]
Given $k,n\in\N^*$ and $\ov a:=(a_1,\ldots,a_k),\ov b:=(b_1,\ldots,b_k),\ov c:=(c_1,\ldots,c_k)\in K^k$, by $\bm{T^k_n(\ov a,\ov b,\ov c)}\in\Mat_n(K)$, $T^k_n(\ov a,\ov b,\ov c)=(t_{ij})_{i,j=1}^n$, we denote the tridiagonal $k$-Toeplitz matrix such that $t_{ii}:=a_i$ for $1\leq i\leq\min(n,k)$ and $t_{i,i+1}:=b_i$, $t_{i+1,i}:=c_i$ for $1\leq i\leq\min(n-1,k)$, i.e., the vectors $\ov a,\ov b,\ov c$ periodically generate, respectively, the main, upper main and lower main diagonals of $T^k_n(\ov a,\ov b,\ov c)$.

A matrix $A\in\Mat_n(K)$ is \e{block lower triangular} if there exists a nontrivial partition of $A$ into blocks, the \e{associated partition} $A=(A_{ij})_{i,j=1}^q$ of \e{size} $q$ ($q>1$), in which the diagonal blocks $A_{ii}$ are square matrices for $1\leq i\leq q$ and $A_{ij}=0$ if $1\leq i<j\leq q$; $A$ is \e{block upper triangular} when $A^T$ is block lower triangular (if $A^T=(B_{ij})_{i,j=1}^q$ then an associated partition of $A$ is $(B_{ji}^T)_{i,j=1}^q$); and $A$ is \e{block triangular} if it is block lower triangular or block upper triangular. Given two partitions showing $A$ as block triangular, the \e{finer} one is that with greater size. The following result is well known over fields.

\begin{thm}[\n{Determinant of a block triangular matrix}]\label{blockdeterminant}
\mbox{}

\nin If $A\in\Mat_n(K)$ is block triangular with associated partition $A=(A_{ij})_{i,j=1}^q$ then
\[\det(A)=\prod_{i=1}^q \det(A_{ii}).\]
\end{thm}

\begin{proof}
We proceed by induction on the size of the partition. Suppose first $A=(a_{ij})_{i,j=1}^n$ is block lower triangular with associated partition of size $2$, $A=\pma B & 0\\ C & D\epma$ with $B\in\Mat_m(K)$ ($m<n$), $B=(b_{ij})_{i,j=1}^m$, $D\in\Mat_{n-m}(K)$, $D=(d_{ij})_{i,j=1}^{n-m}$. For permutation $\sigma\in S_n$ denote $a_\sigma:=\sgn(\sigma)\prod_{i=1}^n a_{i,\sigma(i)}$. Leibniz's formula states $\det(A)=\sum_{\sigma\in S_n} a_\sigma$. Since $A$ is block lower triangular, $a_\sigma=0$ except when $\sigma$ permutes the first $m$ elements among themselves, the subgroup of $S_n$ of such permutations being isomorphic to $S_m\times S_{n-m}$, with $\sigma\mapsto(\tau,\rho)$ implying $\sgn(\sigma)=\sgn(\tau)\sgn(\rho)$. Therefore
\[\det(A)=\sum_{\sigma\in S_n} a_\sigma = \!\!\!\!\sum_{(\tau,\rho)\in S_m\times S_{n-m}}\!\!\!\!\!\!\!\!\!\!\!\! b_\tau d_\rho \, = \sum_{\tau\in S_m} b_\tau \!\!\sum_{\rho\in S_{n-m}}\!\!\!\! d_\rho = \det(B)\det(D).\]
Now suppose the induction hypothesis true; if $A=(A_{ij})_{i,j=1}^{q+1}$ is an associated partition of size $q+1>2$ then an associated partition of $A$ of size $q\geq 2$ is $A=(A'_{ij})_{i,j=1}^q$ with $A'_{ij}=A_{ij}$ if $i,j<q$, $A'_{i,q}=0$ for $i<q$, $A'_{q,j}=\pma A_{q,j}\\A_{q+1,j}\epma$ for $j<q$, and $A'_{qq}=\pma A_{qq} & 0\\ A_{q,q+1} & A_{q+1,q+1}\epma$; by the induction hypothesis, $\det(A)=\prod_{i=1}^q\det(A'_{ii}) = \prod_{i=1}^{q-1}\det(A_{ii})\cdot \det(A_{qq})\det(A_{q+1,q+1})$, as we wanted to prove. Finally, if $A$ is block upper triangular, then $\det(A)=\det(A^T$) with $A^T$ block lower triangular.
\end{proof}

Tridiagonal matrices are not block triangular in general, but they are close in some senses. For example, the associated submatrices of the first minors of a tridiagonal matrix are block triangular (see the proof of Theorem \ref{inverse}). In addition, reducible tridiagonal matrices are block triangular.
\begin{rem}[\n{Reducible tridiagonal matrices are block triangular}]\label{tridiagonalblocktriangular}
\mbox{}

If
$T:=\pma
a_1 & b_1 & & \\
c_1 & a_2 & b_2 &\\
& \ddots & \ddots & \ddots\\
& & c_{n-1} & a_n
\epma$ has $b_i=0$ then
\[T=\left(\begin{array}{cccc|ccc}
a_1 & b_1 & & &&&\\
c_1 & \ddots & \ddots &&&&\\
& \ddots & a_{i-1} & b_{i-1}&&&\\
& & c_{i-1} & a_i & \bm0 &&\\\hline
& & & c_i & a_{i+1} & b_{i+1} & \\
& & & & \ddots & \ddots & \ddots\\
& & & & & c_{n-1} & a_n
\end{array}\right)\]
is an associated partition of size $2$ showing that $T$ is block lower triangular, with diagonal blocks which are tridiagonal matrices. Analogously, if $c_i=0$ then $T$ is block upper triangular with tridiagonal diagonal blocks. If in addition $T$ is $k$-Toeplitz, then so are its diagonal blocks; concretely, if $T=T^k_n(\ov a,\ov b,\ov c)$ then its first diagonal block is $T^k_i(\ov a,\ov b,\ov c)$, the second one $T^k_{n-i}(\sigma_i(\ov a),\sigma_i(\ov b),\sigma_i(\ov c))$ (which contains $b_i=0$ again in the upper main diagonal, and so is reducible and can be decomposed further by blocks, if $n>k+i$).
\end{rem}

Given $A\in\Mat_n(K)$, its \e{characteristic polynomial} is $p_A(x):=\det(xI_n-A)\in K[x]$. An element $\lambda\in K$ is an \e{eigenvalue} of $A$ if there is a nonzero column\footnote{For ease of reading, we also enumerate column vectors component by component, as in $v=(v_1,\ldots,v_n)$.} vector $v\in K^n$ such that $Av=\lambda v$; we say that $v$ is an \e{eigenvector} of $A$ \e{associated to $\lambda$}. An element $\lambda\in K$ is an eigenvalue of $A$ if and only if $p_A(\lambda)$ is a zero divisor of $K$ (\cite[Lemma 17.2]{Brown1993}).

\begin{lem}[\n{Eigenvalues of a block triangular matrix}]\label{blockeigenvalues}
\mbox{}

\nin If $A\in\Mat_n(K)$ is block triangular with associated partition $A=(A_{ij})_{i,j=1}^q$ and $\lambda$ is an eigenvalue of $A$ then $\lambda$ is an eigenvalue of $A_{ii}$ for some $1\leq i\leq n$.
\end{lem}

\begin{proof}
Since $\lambda$ is an eigenvalue of $A$, $p_A(\lambda)=\det(\lambda I_n-A)$ is a zero divisor of $K$, and since $A$ is block triangular, $\lambda I_n-A$ is a block triangular matrix with associated partition $(A'_{ij})_{i,j=1}^q$, $A'_{ij}:=-A_{ij}$ if $i\neq j$, $A'_{ii}:=\lambda I_{n_i}-A_{ii}$ with $n_i$ the order of $A_{ii}$. By Theorem \ref{blockdeterminant}, $\det(\lambda I_n-A)=\prod_{i=1}^q\det(\lambda I_{n_i}-A_{ii})$, so $p_{A_{ii}}(\lambda)=\det(\lambda I_{n_i}-A_{ii})$ is a zero divisor of $K$ for some $1\leq i\leq n$, i.e., $\lambda$ is an eigenvalue of $A_{ii}$ for some $1\leq i\leq n$.
\end{proof}
\end{pgraph}

\begin{pgraph}\n{Polynomials.} Given a polynomial $f\in \Z[x_1,\ldots,x_k,y_1,\ldots,y_k]$ and vectors $\ov a:=(a_1,\ldots,a_k),$ $\ov b:=(b_1,\ldots,b_k)\in K^k$, we define the \e{evaluation} $f^{\ov a,\ov b}$ as the image of the evaluation of $f$ into $K^k$ mapping $x_i\mapsto a_i$ and $y_i\mapsto b_i$ for $1\leq i\leq k$. Given $s\in\N$ we define the (cyclic) \e{shift of $f$ by $s$} as
\[f_s:=f(x_{s+1},\ldots,x_{s+k},y_{s+1},\ldots,y_{s+k})\]
(recall that, by periodic extension, we denote $x_{i+k}:=x_i, y_{i+k}:=y_i$ for all $i\in\N^*$). In other words, $f_s=f^{\sigma_s(\ov x),\sigma_s(\ov y)}$ with $\ov x:=(x_1,\ldots,x_k), \ov y:=(y_1,\ldots,y_k)$. Observe that shifts are automorphisms of $\Z[x_1,\ldots,x_k,y_1,\ldots,y_k]$ and that \\ $(f_{s_1})_{s_2}=f_{s_1+s_2}$ (we say that \e{shifts are additive}) for $s_1,s_2\in\N$.
Note that when a shift and an evaluation have both to be applied, the shift must be applied first, giving
\[f^{\ov a,\ov b}_s = f^{\sigma_s(\ov a),\sigma_s(\ov b)} = f(a_{s+1},\ldots,a_{s+k},b_{s+1},\ldots,b_{s+k}).\]
\end{pgraph}

\begin{pgraph}\n{Complexity.}\label{introcomplexity} After we arrive to formulas for the determinant, spectral properties, and elements of the inverse of a tridiagonal $k$-Toeplitz matrix, we will study in Section \ref{complexitysection} the complexity of different algorithms arising from the formulas and related procedures, in order to compare them in terms of efficiency. We describe now the relevant ideas and make definitions in a somewhat informal manner; a completely rigorous treatment is out of the scope of this paper.
We make use of algebraic complexity, in the tradition of Ostrowski and Winograd (\cite{Ostrowski1954,Winograd1971}); given a commutative ring $K$, we consider a model of computation in which:
\enuma
\item The elementary operations of $K$ (addition, substraction, multiplication, division if $K$ is a field, or Euclidean division if $K=\Z$) have unit cost each.
\item  Every intermediate result computed by an algorithm is available for its subsequent steps.
\eenum
For our purposes it is enough to consider a Turing machine $M_K$ provided with oracles for the elementary operations of $K$ (for a more specific model of computation, see \cite{BurgisserClausenShokrollahi1997})\footnote{Observe that the bit complexity as defined for a usual Turing machine is the algebraic complexity in $M_{\Z_2}$, while the arithmetic complexity is that in $M_{\Z}$.}.
Then, given a mathematical entity $X$ (square root, inverse of a matrix, determinant, etc.) defined over some commutative unital ring $K$, the \e{algebraic complexity} (or \e{cost}) of an algorithm $F$ that computes $X$ in $M_K$, denoted by $\C_F(X)=\C^K_F(X)$, is the number of elementary operations needed in $K$ to arrive at the result in the worst case of the algorithm\fn{A complexity of $0$ or less means that no additional computations are needed from the original data.}. When necessary, we write the number $x$ of operations as $(x)_K$ to emphasize the underlying ring. If an algorithm $F$ uses several different rings $K_1,\ldots,K_p$ in the computation of $X$ (for example, $K$ for managing elementary operations and $\Z$ for managing exponents)\footnote{To achieve this we can work in the Turing machine $M_{K_1\oplus\cdots\oplus K_p}$.} then we define its complexity as the external sum $\C_F(X):=\sum_{i=1}^p\C^{K_i}_F(X)\in\N^p$ (e.g. we may have $\C_F(X)=(100)_K+(10)_\Z$). As an example, if $F$ is the algorithm multiplying $m$ matrices $A_1,\ldots,A_m\in\Mat_2(K)$ by naive matrix multiplication (rows times columns) then $\C_F(A_1\cdots A_m)=12(m-1)$, since (in worst case) each entry of the product of two matrices is computed through $2$ products and $1$ sum, there are $4$ nonzero entries in a $2\times2$ matrix, and $m-1$ products of two matrices are realized.

We will need to compare algorithms which compute families of entities depending on several parameters ($n$, $k$, etc.). For this we define our notion of efficiency below (not completely operational, but enough for our purposes). Informally:
\enuma
\item We will apply first a coarse, limit comparison ($100\log_2 n$ operations are more efficient than $n$ operations), then and only if needed, a finer comparison ($n$ operations are more efficient than $2n$ operations).
\item We will give priority to some parameters over the others; mainly, to $n$ over $k$, by considering first $n$ variable and $k$ constant ($\log_2 n+k$ operations are more efficient than $n+\log_2 k$ operations), then $k$ variable only if needed.
\eenum
When algorithms $F$ and $G$ both compute a specific entity $X$ in $M_K$, we say that $F$ is \e{more efficient} than $G$ when computing $X$ if $\C_F^K(X)<\C_G^K(X)$ in $\N$. Now suppose $F$ and $G$ both compute in $M_K$ a family of entities $\{X_m\}$ indexed by one parameter $m\in\N$, and define their \e{complexity functions} $f(m):=\C_F^K(X_m)$, $g(m):=\C_G^K(X_m)$; we say that $f(m)\equiv g(m)$ when there exists some $m_0\in\N$ such that $f(m)=g(m)$ for $m>m_0$, and that $f(m)\prec g(m)$ if either $f(m)=o(g(m))$\footnote{$f(m)=o(g(m))$ when $\displaystyle\lim_{m\rightarrow\infty}\frac{f(m)}{g(m)}=0$. For example we have $\log m=o(m)$.} or $g(m)=cf(m)$ with $c\in\Q$, $c>1$ (these two cases are not complementary, but are the only ones needed in this paper). Then we say that $F$ is \e{more efficient} than $G$ if $f(m)\prec g(m)$. Now suppose that $F$ and $G$ both compute in $M_K$ a family of entities $\{X_{m_1,\ldots,m_p}\}$ indexed by $p>1$ ordered parameters $m_1,\ldots,m_p\in\N$, and define their complexity functions $f(m_1,\ldots,m_p):=\C_F^K(X_{m_1,\ldots,m_p})$, $g(m_1,\ldots,m_p):=\C_G^K(X_{m_1,\ldots,m_p})$. We give priority to smaller indices in the comparison of complexities, hence we say that $f(m_1,\ldots,m_p)\prec g(m_1,\ldots,m_p)$ ($F$ is \e{more efficient} than $G$) if, when considering $m_2,\ldots,m_p$ fixed as constants, either $f(m_1)\prec g(m_1)$ or $f(m_1)\equiv g(m_1)$ and $f(m_2,\ldots,m_p)\prec g(m_2,\ldots,m_p)$ (as functions in $p-1$ variables)\footnote{Our $\prec$ relation is not connected (e.g. $(m_2+2m_3)m_1$ and $(2m_2+m_3)m_1)$ are incomparable), but all the relevant pairs of functions appearing in this paper are comparable.}. Finally, if $F$ and $G$ both compute a family of entities using $q$ different rings, then in order to compare their efficiencies we consider the costs in all rings to be equivalent, and so if $\C_F(X_{m_1,\ldots,m_p})=(f_1(m_1,\ldots,m_p)_{K_1},\ldots,f_q(m_1,\ldots,m_p)_{K_q})$ then we use $f=f_1+\cdots+f_q$ as the complexity function of $F$.

\smallskip

Herein, let us justify that this resorting to algebraic complexity is not superfluous. A ring $K$ is \e{computably presentable} (in the usual sense) if it has an isomorphic presentation (a \e{computable presentation}) in which the elementary operations are computable and the equality relation is decidable, in an ordinary Turing machine $M_{\Z_2}$. When $K$ is computably presentable, the \e{bit complexity} of an algorithm $F$ over $K$ is defined as the algebraic complexity of $F$ (in $M_{\Z_2}$) for a previously fixed computable presentation of $K$. Thus, the usual bit complexity is not well defined for arbitrary commutative rings, as there are many commutative rings which are not computably presentable: to begin with, no uncountable commutative ring is computably presentable; and there even exist countable noncomputably presentable commutative rings, such as $\Q[\sqrt{p_i} \ | \ T_i\text{ does not halt}]$, where $p_i$ is the $i$th prime number and $T_i$ is the $i$th Turing machine. Moreover, in an ordinary Turing machine, different computable presentations of the same ring may produce different bit complexities for the same algorithm. In contrast, the use of algebraic complexity allows to abstractly compare the efficiency of different algorithms over the same ring, for an arbitrary ring. In addition, if $K$ is given a computable presentation, then the bit complexity of an algorithm in this presentation can be found from its algebraic complexity analysis by inserting the bit complexity of each elementary operation in the corresponding places.
\end{pgraph}

\section{Generalized Fibonacci polynomials}\label{sectionFibonacci}

\numberwithin{equation}{section}

\begin{defn}[\n{Generalized Fibonacci polynomials}]\label{Fibonacci}
\mbox{}

\nin Given $m\in\N$, we define the (bivariate) \e{generalized Fibonacci polynomial of order $m$} over $K$ as
\[U_m(x,y):=\sum_{i=0}^{\mathclap{\lfloor(m-1)/2\rfloor}}\; (-1)^i\binom{m-1-i}i x^{m-1-2i}y^i.\]
Note that $U_0(x,y)=0$, $U_1(x,y)=1$.
Since these polynomials were introduced by Lucas in \cite{Lucas}, their sequence is also called the \e{Lucas polynomial sequence of the first kind} (\cite[p. 2]{Ribenboim}).
\end{defn}

\smallskip

By the Cayley-Hamilton theorem, the powers of a $2\times2$ matrix are linear combinations of the matrix itself and the identity matrix; as it turns out, those linear combinations are parametrized by generalized Fibonacci polynomials evaluated on its trace and determinant.

\begin{lem}[\n{Properties of generalized Fibonacci polynomials}]\label{powers}
\mbox{}

\enum
\item For all $m\in\N^*$,
\[U_{m+1}(x,y)=xU_m(x,y) - yU_{m-1}(x,y).\]
\item If $A\in\Mat_2(K)$ then, for all $m\in\N^*$,
\[A^m=U_m(\tr(A),\det(A))A - \det(A)U_{m-1}(\tr(A),\det(A))I_2.\]
\eenum
\end{lem}

\newpage
\begin{proof}
\mbox{}

\enum
\item $xU_m(x,y) - yU_{m-1}(x,y) =$
\eq{
&= \sum_{i=0}^{\mathclap{\lfloor(m-1)/2\rfloor}}\; (-1)^i\binom{m-1-i}i x^{m-2i}y^i +\sum_{i=0}^{\mathclap{\lfloor(m-2)/2\rfloor}}\; (-1)^{i+1}\binom{m-2-i}i x^{m-2-2i}y^{i+1} = \\
&= x^m+\sum_{i=1}^{\mathclap{\lfloor(m-1)/2\rfloor}}\; (-1)^i\binom{m-1-i}i x^{m-2i}y^i +\sum_{i=1}^{\mathclap{\lfloor m/2\rfloor}}\; (-1)^i\binom{m-1-i}{i-1} x^{m-2i}y^{i} = \\
&= x^m+\sum_{i=1}^{\mathclap{\lfloor m/2\rfloor}}\; (-1)^i\left(\binom{m-1-i}i+\binom{m-1-i}{i-1}\right) x^{m-2i}y^i =\\
& = \sum_{i=0}^{\mathclap{\lfloor m/2\rfloor}}\; (-1)^i\binom{m-i}i x^{m-2i}y^i = U_{m+1}(x,y),
}
since $\lfloor(m-1)/2\rfloor =\lfloor m/2\rfloor$ when $m$ is odd and $\binom{m-1-\lfloor m/2\rfloor}{\lfloor m/2\rfloor}=0$ when $m$ is even.
\item We proceed by induction. Denote $t:=\tr(A)$, $d:=\det(A)$ and $U_m:=U_m(t,d)$. The base case $A=A$ is true since $U_1=1$, $U_0=0$. By the Cayley-Hamilton theorem $A^2=tA-dI_2$, so if $A^m=U_mA - dU_{m-1}I_2$ then
\[A^{m+1}= AA^m=U_mA^2 - dU_{m-1} A = (tU_m - dU_{m-1})A-dU_mI_2 = U_{m+1}A-dU_mI_2\]
by the previous item.\qedhere
\eenum
\end{proof}

In some rings we can write generalized Fibonacci polynomials in terms of Chebyshev polynomials of the second kind, divisions, and square roots.
\begin{rem}[\n{\small Generalized Fibonacci polynomials as Chebyshev polynomials}]\label{Chebyshev}

Given $m\in\N$, we define the $m$th (univariate) \e{Chebyshev polynomial of the second kind} over $K$ as
\eqnum{U_m(x):=\sum_{i=0}^{\mathclap{\lfloor m/2\rfloor}}\; (-1)^i\binom{m-i}i (2x)^{m-2i}.\label{Chebyshevdef}}
Through \eqref{Chebyshevdef} we also define $U_{-1}(x):=0$. Chebyshev polynomials of the second kind satisfy the recurrence relation $U_{m+1}(x)=2xU_{m}(x)-U_{m-1}(x)$.\\
Let $K$ be a commutative unital ring in which every element is a square. If $t,d\in K$ and $d$ is a regular element, then $\sqrt d$ is regular and for all $m\in\N$ we have, from Definition \ref{Fibonacci},
\[U_m(t,d)=(\sqrt d)^{m-1} U_m(t/\sqrt d,1),\]
with the computation done in $Q(K)$ but the result lying in $K$. If in addition $K$ is free of $2$-torsion (i.e., if $2$ is a regular element of $K$) then by \eqref{Chebyshevdef} we get
\eqnum{U_m(t,d)=(\sqrt d)^{m-1}U_{m-1}(t/(2\sqrt d)), \label{(B)}}
which writes the generalized Fibonacci polynomial of order $m$ in terms of the $(m-1)$th Chebyshev polynomial of the second kind.
\end{rem}

\medskip

Over fields we can write the coefficients of the linear combination of Lemma \ref{powers}(2) in terms of the eigenvalues of the matrix.

\begin{rems}[\n{Powers through the eigenvalues}]\label{eigenvalues}
\mbox{}

In these remarks let $K$ be a field.
\enum
\item Given $A\in\Mat_2(K)$ we can also express $U_m(\tr(A),\det(A))$ in terms of its eigenvalues $\lambda_1,\lambda_2$ (possibly equal) in an algebraic closure $\ov K$ of $K$. By induction it is easily shown that, for $m\in\N^*$,
\eqnum{U_m(\tr(A),\det(A))=\sum_{i=0}^{m-1}\lambda_1^i\lambda_2^{m-i-1}.\label{(A)}}
Thus by Lemma \ref{powers}(2) (taking into account that $\det(A)=\lambda_1\lambda_2$) we can express $A^m$ in terms of the eigenvalues. This choice makes the formula dependent on the characteristic $\ch(K)$ of the field: if $\ch(K)\neq2$, the eigenvalues of $A\in\Mat_2(K)$ can be found from the characteristic polynomial by the quadratic formula, but when $\ch(K)=2$ the roots of $x^2+ax+b\in K[x]$ cannot be expressed by radicals when the polynomial is irreducible over $K$ and $a\neq0$, and a different approach is taken (see e.g. \cite[Exercise 2.4.6]{Cox}): its roots are $x_1=aR(b/a^2)$ and $x_2=x_1+a$, where $R(y)$ denotes a root of $x^2+x+y$.
\item Formula \eqref{(A)} can be simplified as follows: if $A$ is nondefective ($\lambda_1\neq\lambda_2$) then
\eqnum{U_m(\tr(A),\det(A))=\frac{\lambda_2^m-\lambda_1^m}{\lambda_2-\lambda_1},\label{(A1)}}
while if $A$ is defective ($\lambda_1=\lambda_2=:\lambda$) then
\eqnum{U_m(\tr(A),\det(A))=m\lambda^{m-1}.\label{(A2)}}
\n{Remark:} Defectiveness is easy to detect: if $\ch(K)=2$, the matrix $A\in\Mat_2(K)$ is defective if and only if $\tr(A)=0$, i.e., if and only if the characteristic polynomial is of the form $x^2+\det(A)$ (with single eigenvalue $\sqrt{\det(A)}\in\ov K$). If $\ch(K)\neq2$, by the quadratic formula the matrix is defective if and only if $\tr(A)^2-4\det(A)=0$ (with single eigenvalue $\tr(A)/2$). In any case, the matrix $A$ is defective if and only if $\tr(A)^2-4\det(A)=0$.
\eenum
In case $K$ is not a field, Formula \eqref{(A)} still holds if $\lambda_1,\lambda_2$ are two eigenvalues in some overring $\ov K$ such that the characteristic polynomial of $A$ equals $(x-\lambda_1)(x-\lambda_2)$ in $\ov K[x]$, the simplification in the defective case can always be done, and the simplification in the nondefective case can be done when $\lambda_2-\lambda_1$ is a unit of $\ov K$.
\end{rems}

\section{Continuant polynomials}\label{sectioncontinuants}

In what follows we define and study the multivariate polynomials which encode the periodicity in the formula for the determinant. We call them \e{continuant polynomials} since they are closely related to continuants: Continuant monomials of type $p$ are related to the Leibniz expansion of a continuant. Continuant polynomials of type $\alpha$ generalize the non-determinant form of a continuant given by Muir and Metzler in \cite[Item 545]{MuirMetzler1933}. Continuant polynomials of type $\beta$ and $\pi$ are generalizations included here to fit the $k$-Toeplitz case. We give rigorous, combinatorial definitions which allow us to prove several elementary but essential properties of continuant polynomials (see Lemma \ref{recurrences}).

\smallskip

Roughly speaking, given variables $x_1,\ldots,x_k$ and $y_1,\ldots,y_k$, to build the monomial $p_{r,k}(i_1,\ldots,i_m)$ we start with the product $x_1\cdots x_r$ and then for each index $i_j$ we substitute two consecutive $x$ variables in the product, $x_{i_j}$ and $x_{i_j+1}$, with the corresponding $y$ variable $y_{i_j}$ (so a $y$ variable ``weights'' like two $x$ variables), even cyclically: $x_r$ and $x_1$ can be substituted together, but with the caveat that they are not substituted by $y_r$, but by $y_k$.\footnote{This phenomenon reflects the fact that, in the periodic extension of the vector $(y_1,\ldots,y_k)$, the element ``preceding'' $y_1$ is $y_k$.} The indices are taken so that the consecutive substitutions they imply are indeed possible. Then the polynomial $\pi(r,k)$ is the sum of all $p_{r,k}$ polynomials for all possible indices, the polynomial $\alpha(r,k)$ is the sum of those $p_{r,k}$ which do not have the variable $y_k$, and the polynomial $\beta(r,k)$ is the sum of those $p_{r,k}$ which do have the variable $y_k$.

\begin{defns}[\n{Continuant polynomials}]\label{polynomials}
\mbox{}

\nin Given $r\in\Z$ we denote $[r]:=\{1,\ldots,r\}$ if $r\geq1$, $[r]:=\emptyset$ otherwise. For a finite set $S\seq\N^*$, by $\binom S{m}_2$ with $m\in\N^*$ we denote the set of all $m$-combinations of the set $S$ satisfying $|s-t|\geq2$ for all $s,t\in S$, and  by $\binom S{m}_{2c}$ the subset which applies this rule also cyclically, i.e., the subset of $\binom S{m}_2$ which excludes those combinations including both $\min(S)$ and $\max(S)$. For example \[\binom{[7]}{3}_{2c}=\{(1,3,5),(1,3,6),(1,4,6),(2,4,6),(2,4,7),(2,5,7),(3,5,7)\}.\]
We also denote $\binom S{0}_{2}:=\{0\}$ and $\binom S{0}_{2c}:=\{0\}$ (even if $S$ is empty). Given $k,r\in\N^*$ with $r\leq k+1$ and considering the ring $R:=\Z[x_1,\ldots,x_k,x_{k+1},y_1,\ldots,y_k]$, we denote
\eq{
&x'_i:=x_i\text{ for }1\leq i\leq r, \,\, x'_{r+1}:=x_1,\\
&y'_i:=y_i\text{ for }1\leq i<r, \,\,  y'_r:=y_k,}
and define the \e{continuant monomial of type $p$} of $R$ (computed inside the ring $\Z(x_1,\ldots,y_k)$)
\eqnum{p_{r,k}(i_1,\ldots,i_m):=x_1\cdots x_r\cdot\frac{y'_{i_1}}{x'_{i_1}x'_{i_1+1}}\cdots\frac{y'_{i_m}}{ x'_{i_m}x'_{i_m+1}}\label{monomials}}
for $(i_1,\ldots,i_m)\in \binom{[r]}{m}_{2c}$ with $1\leq m\leq\lfloor r/2\rfloor$. With the same Formula \eqref{monomials} and defining
\[x'_0:=1, y'_0:=x_1\]
we also extend the definition of $p_{r,k}(i_1,\ldots,i_m)$ to the case $i_1=0$, $(i_2,\ldots,i_m)\in \binom{[r]}{m-1}_{2c}$ (the second condition holding when $m>1$). So we have
\[p_{r,k}(0)=x_1\cdots x_r, \,\,\,\, p_{r,k}(0,i)=p_{r,k}(i)\,\text{ for }i\in \binom{[r]}{m}_{2c}, m\geq1.\]
In addition we define $p_{0,k}(0):=1$.\\
For example,
\eq{
&p_{6,8}(3)=x_1x_2y_3x_5x_6,\,\,
p_{6,8}(1,5)=y_1x_3x_4y_5,\\
&p_{6,8}(6)=x_2x_3x_4x_5y_8,\,\,
p_{6,6}(3,6)=x_2y_3x_5y_6,\,\,
p_{7,6}(3,7)=x_2y_3x_5x_6y_6,\\
&p_{3,4}(0)=x_1x_2x_3,\,\,
p_{3,4}(0,3)=p_{3,4}(3)=x_2y_4.
}
Now, for fixed $0\leq r\leq k$ we denote in $\Z[x_1,\ldots,x_k,y_1,\ldots,y_k]$ the sum of all the monomials $p_{r,k}$ by $\pi(r,k)$,
\eqnum{
\pi(r,k):=\sum_{m=0}^{\lfloor r/2\rfloor} \sum_{i\in\binom{[r]}{m}_{2c}} p_{r,k}(i),\label{pi}
}
the sum of those $p_{r,k}$ having degree $0$ in $y_k$ by $\alpha(r,k)$,
\eqnum{
\alpha(r,k):=\sum_{m=0}^{\lfloor r/2\rfloor} \sum_{i\in\binom{[r-1]}{m}_2} p_{r,k}(i),\label{alpha}
}
and the sum of those $p_{r,k}$ having degree $1$ in $y_k$ by $\beta(r,k)$,
\eqnum{
\beta(r,k):=\sum_{m=0}^{\mathclap{\lfloor(r-2)/2\rfloor}}\; \sum_{i\in\binom{[r-2]-\{1\}}{m}_2} p_{r,k}(i,r).\label{beta}
}
We extend the definitions to $\alpha(-1,k)=0$ through Formula \eqref{alpha} and to \linebreak $\beta(k+1,k)$ through Formula \eqref{beta}.\\
Note that $\pi(0,k)=1=\alpha(0,k)$, $\beta(0,k)=0=\beta(1,k)$ and that, for $0\leq r\leq k$,
\eqnum{
\pi(r,k)=\alpha(r,k)+\beta(r,k).\label{alphabetasumpi}
}
For example we have
\eq{\pi(4,6)=&x_1x_2x_3x_4+y_1x_3x_4+x_1y_2x_4+x_1x_2y_3+x_2x_3y_6+y_1y_3+y_2y_6,\\
\alpha(4,6)=&x_1x_2x_3x_4+y_1x_3x_4+x_1y_2x_4+x_1x_2y_3+y_1y_3, \,\, \beta(4,6)=x_2x_3y_6+y_2y_6,\\
\beta(6,5)=&x_2x_3x_4x_5y_5+y_2x_4x_5y_5+x_2y_3x_5y_5+x_2x_3y_4y_5+y_2y_4y_5.
}
We call the \e{continuant polynomials of type $\alpha$, type $\beta$, and type $\pi$} respectively to the sets of polynomials $\{\alpha(r,k)\}$, $\{\beta(r,k)\}$, $\{\pi(r,k)\}$ for all valid pairs $(r,k)$ in each case. The shift by $s$ of continuant polynomials of type $\alpha$ we write as $\alpha_s(r,k):=(\alpha(r,k))_s$.
\end{defns}

\begin{lem}[\n{Identities with continuant polynomials}]\label{recurrences}
\mbox{}

\nin Given $k\in\N^*$, in $\Z[x_1,\ldots,x_{k},y_1,\ldots,y_k]$ we have:
\enum
\item For $1\leq r\leq k+1$,
\[\beta(r,k)=y_k\alpha_1(r-2,k).\]
\item For $0\leq r\leq k-1$,
\[\alpha(r+1,k)=x_{r+1}\alpha(r,k)+y_r\alpha(r-1,k).\]
\item For $2\leq r\leq k$,
\[\beta(r+1,k)=x_{r}\beta(r,k)+y_{r-1}\beta(r-1,k).\]
\item For $1\leq r \leq k-1$ and $1\leq s\leq k-r$,
\[\alpha_{s-1}(r,k)= x_s\alpha_{s}(r-1,k)+y_s\alpha_{s+1}(r-2,k).\]
\item For $0\leq r\leq k-1$,
\[\alpha(k-1,k)\beta(r+1,k)-\alpha(r,k)\beta(k,k)=(-1)^{r+1}y_ky_1\cdots y_r\alpha_{r+1}(k-r-2,k).\]
\eenum
\end{lem}

\begin{proof}
\mbox{}

\enum
\item Recall that for $1\leq r\leq k+1$ we have, by definition,
\[\alpha(r-2,k)=\sum_{m=0}^{\mathclap{\lfloor (r-2)/2\rfloor}}\; \sum_{i\in\binom{[r-3]}{m}_2} p_{r-2,k}(i).\]
Fix some $p:=p_{r-2,k}(i_1,\ldots,i_m)$ appearing as a term in the above expression of $\alpha(r-2,k)$, with indices rearranged so that $i_1<\ldots<i_m$. Then, working in $\Z(x_1,\ldots,x_{k},y_1,\ldots,y_k)$, the shift of $p$ by $1$ satisfies \[(p_{r-2,k}(0))_1=p_{r-1,k}(0)/x_1\]
if $m=0$ and
\[(p_{r-2,k}(i_1,\ldots,i_m))_1 = p_{r-1,k}(i_1+1,\ldots,i_m+1)/x_1\]
with $\left\{(i_1+1,\ldots,i_m+1) \ | \ (i_1,\ldots,i_m)\in\binom{[r-3]}{m}_2\right\}=\binom{[r-2]-\{1\}}{m}_2$ if $m\geq1$; whence
\eq{
&y_k\alpha_1(r-2,k)=\sum_{m=0}^{\mathclap{\lfloor (r-2)/2\rfloor}}\; \sum_{i\in\binom{[r-2]-\{1\}}{m}_2} p_{r-1,k}(i)y_k/x_1 =\\
&= \sum_{m=0}^{\mathclap{\lfloor (r-2)/2\rfloor}}\; \sum_{i\in\binom{[r-2]-\{1\}}{m}_2} p_{r,k}(i,k) = \beta(r,k)
}
by definition.
\item For $r=1$ we have $\alpha(2,k)=x_1x_2+y_1$, $\alpha(1,k)=x_1$, $\alpha(0,k)=1$, so indeed $\alpha(2,k)=x_2\alpha(1,k)+y_1\alpha(0,k)$. For $2\leq r\leq k-1$ consider
\[\alpha(r+1,k)=\sum_{m=0}^{\lfloor(r+1)/2\rfloor}\!\!\!\! \sum_{i\in\binom{[r]}{m}_2}\!\! p_{r+1,k}(i)\]
and fix some $p:=p_{r+1,k}(i_1,\ldots,i_m)$ appearing as a term in the above expression of $\alpha(r+1,k)$, with indices rearranged so that $i_1<\ldots<i_m$. Since $i_m<r+1$ (so $y_k$ is not a factor of $p$), we have that either
\ite
\item[--] $i_m<r$, $x_{r+1}$ is a factor of $p$ and $y_r$ is not, whence $p=x_{r+1}p_{r,k}(i_1,\ldots,i_m)$, or
\item[--] $i_m=r$, $y_r$ is a factor of $p$ and $x_{r+1}$ is not, whence $p=y_rp_{r-1,k}(i_1,\ldots,i_{m-1})$ if $m\geq 2$ and $p=y_rp_{r-1,k}(0)$ if $m=1$.
\eite
Denote
\begin{align*}
&S_1(m):=\left\{(i_1,\ldots,i_m)\in \binom{[r]}{m}_2 \ | \ i_1,\ldots,i_m<r\right\}, \\
&S_2(m):=\left\{(i_1,\ldots,i_{m-1},r)\in \binom{[r]}{m}_2\right\}
\end{align*}
and observe that $\binom{[r]}{m}_2$ is the disjoint union of $S_1$ and $S_2$ for $0\leq m\leq\lfloor(r+1)/2\rfloor$. We have
$S_1(m)=\binom{[r-1]}{m}_2$ for $0\leq m\leq\lfloor(r+1)/2\rfloor$ and $S_2(m)= \{(i,r) \ | \ i\in \binom{[r-2]}{m-1}_2\}$ for $2\leq m\leq\lfloor(r+1)/2\rfloor$; in addition, since $p_{r+1,k}(0,r)=p_{r+1,k}(r)=y_rp_{r-1,k}(0)$, we can substitute $S_2(1)$ with $\{(i,r) \ | \ i\in \binom{[r-2]}{0}_2\}$. Therefore
\eq{
\alpha(r+1,k) &\overset{(a)}{=}\sum_{m=0}^{\lfloor r/2\rfloor}\sum_{i\in\binom{[r-1]}{m}_2}x_{r+1}p_{r,k}(i) +
\!\!\!\!\sum_{m=1}^{\lfloor(r+1)/2\rfloor}\!\!\!\!\sum_{i\in\binom{[r-2]}{m-1}_2}y_rp_{r-1,k}(i) = \\
&\overset{(b)}{=} x_{r+1}\sum_{m=0}^{\lfloor r/2\rfloor}\sum_{i\in\binom{[r-1]}{m}_2}\!\!\!\! p_{r,k}(i) \, + \, y_r\!\!\!\!\!\!\!\!\sum_{m=0}^{\lfloor (r-1)/2\rfloor}\!\!\!\!\sum_{i\in\binom{[r-2]}{m}_2}\!\!\!\! p_{r-1,k}(i) =\\
&= x_{r+1}\alpha(r,k)+y_r\alpha(r-1,k),}
where to rewrite the bounds of summations we have applied:
\ite
\item In the first term at the RHS of (a), that $\lfloor(r+1)/2\rfloor = \lfloor r/2\rfloor$ when $r$ is even, while we cannot simultaneously have $m=\lfloor(r+1)/2\rfloor$ and $i_m<r$ when $r$ is odd.
\item In the second term at the RHS of (a), that we cannot simultaneously have $m=0$ and $i_m=r$.
\item In the second term at the RHS of (b), that $\lfloor(r+1)/2\rfloor-1= \lfloor(r-1)\rfloor/2$.
\eite
\item Using item (1) and the fact that the shift by $1$ is an automorphism of the polynomial ring we get, for $2\leq r\leq k$,
\eq{
&\beta(r+1,k)=y_k\alpha_1(r-1,k)=y_k(x_{r-1}\alpha(r-2,k)+y_{r-2}\alpha(r-3,k))_1 = \\
&=(x_{r-1})_1y_k\alpha_1(r-2,k)+(y_{r-2})_1y_k\alpha_1(r-3,k) = x_r\beta(r,k)+y_{r-1}\beta(r-1,k).
}
\item With arguments similar to those of item (2), applied now to the variables of lowest index, we are going to show that
\[\alpha(r,k)=x_1\alpha_1(r-1,k)+y_1\alpha_2(r-2,k).\]
From this, since the shift by $s-1$ is an automorphism of the polynomial ring, and shifts are additive, we get the desired consequence:
\eq{\alpha_{s-1}(r,k)&=(x_1)_{s-1}(\alpha_1(r-1,k))_{s-1}+(y_1)_{s-1}(\alpha_2(r-2,k))_{s-1} = \\ &=x_s\alpha_s(r-1,k)+y_s\alpha_{s+1}(r-2,k).}
We expose the arguments with less detail than in item (2). Recall that
\[\alpha(r,k)=\sum_{m=0}^{\lfloor r/2\rfloor}\sum_{i\in\binom{[r-1]}{m}_2}\!\! p_{r,k}(i)\]
and fix some $p:=p_{r,k}(i_1,\ldots,i_m)$ with $i_1<\cdots<i_m$. Then either
\ite
\item[--] $i_1\neq1$, $x_{1}$ is a factor of $p$ (since $r\geq1$) and $y_1$ is not, whence \\$p=x_1(p_{r-1,k}(i_1-1,\ldots,i_m-1))_1$ if $i_1>0$ and \\$p=x_1(p_{r-1,k}(0,i_2-1,\ldots,i_m-1))_1$ if $i_1=0$, or
\item[--] $i_1=1$, $y_1$ is a factor of $p$ and $x_{1}$ is not, whence \\$p=y_1(p_{r-2,k}(i_2-2,\ldots,i_m-2))_2$ if $m\geq 2$ and $p=y_1(p_{r-2,k}(0))_2$ if $m=1$.
\eite
For $i=(i_1,\ldots,i_m)\in\binom{S}m_2$ with $i_1<\cdots<i_m$ and $z\in\N$, $i_1\geq z$ or $i_1=0$ and $i_2\geq z$ (or $m=1$), we define $i-z:=(i_1-z,\ldots,i_m-z)$ when $i_1\geq z$ and $i-z:=(0,i_2-z,\ldots,i_m-z)$ when $i_1=0$ and $i_2\geq z$.
Then
\eq{
&\alpha(r,k)=\sum_{m=0}^{\lfloor r/2\rfloor}\sum_{i\in\binom{[r-1]-\{1\}}{m}_2}\!\!\!\!\!\!\!\!x_1(p_{r-1,k}(i-1))_1 +
\sum_{m=1}^{\lfloor r/2\rfloor}\sum_{i\in\binom{[r-1]-\{1,2\}}{m-1}_2}\!\!\!\!\!\!\!\!y_1(p_{r-2,k}(i-2))_2 = \\
&= x_1\left(\sum_{m=0}^{\lfloor (r-1)/2\rfloor}\sum_{i\in\binom{[r-2]}{m}_2}\!\!\!\! p_{r-1,k}(i)\right)_1 \, + \, y_1\!\left(\sum_{m=0}^{\lfloor (r-2)/2\rfloor}\!\!\!\!\sum_{i\in\binom{[r-3]}{m}_2}\!\!\!\! p_{r-2,k}(i)\right)_2 = \\
& = x_1\alpha_1(r-1,k)+y_1\alpha_2(r-2,k).}

\item Fixed $k$ we proceed by strong induction on $r$. We need two base cases. First, for $r=0$ we have $\beta(k,k)=y_k\alpha_1(k-2,k)$ by item (1). Therefore, since $\beta(1,k)=0$ and $\alpha(0,k)=1$, this shows
\[\alpha(k-1,k)\beta(1,k)-\alpha(0,k)\beta(k,k)=(-1)^1y_k\alpha_1(k-2,k).\]
Second, for $r=1$ we have $k-1\geq1$ and then items (1,4) and the fact that $\beta(2,k)=y_k$ allow us to show that
\eq{
&\alpha(k-1,k)\beta(2,k)-\alpha(1,k)\beta(k,k) = \alpha(k-1,k)y_k-x_1y_k\alpha_1(k-2,k) = \\
=&\, y_k(\alpha(k-1,k)-x_1\alpha_1(k-2,k)) = y_ky_1\alpha_2(k-3,k).
}
Now pick $2\leq r\leq k-1$ and suppose that
\[\alpha(k-1,k)\beta(r',k)-\alpha(r'-1,k)\beta(k,k)=(-1)^{r'}y_ky_1\cdots y_{r'-1}\alpha_{r'}(k-r'-1,k)\]
for $r'\in \{r,r-1\}$. Then by items (2,3), the induction hypothesis, and item (4),
\eq{
&\alpha(k-1,k)\beta(r+1,k)-\alpha(r,k)\beta(k,k)=\\
=&\alpha(k-1,k)(x_r\beta(r,k)+y_{r-1}\beta(r-1,k))-(x_r\alpha(r-1,k)+y_{r-1}\alpha(r-2,k))\beta(k,k) =\\ =&x_r(\alpha(k-1,k)\beta(r,k)-\alpha(r-1,k)\beta(k,k))+y_{r-1}(\alpha(k-1,k)\beta(r-1,k)-\alpha(r-2,k)\beta(k,k)) =\\
=&x_r(-1)^ry_ky_1\cdots y_{r-1}\alpha_r(k-r-1,k) + y_{r-1}(-1)^{r-1}y_ky_1\cdots y_{r-2}\alpha_{r-1}(k-r,k) =\\
=&(-1)^{r+1}y_ky_1\cdots y_{r-1}(\alpha_{r-1}(k-r,k) - x_r\alpha_r(k-r-1,k)) = \\
=&(-1)^{r+1}y_ky_1\cdots y_{r-1}y_r\alpha_{r+1}(k-r-2,k).\qedhere
}
\eenum
\end{proof}

\begin{defn}[\n{Continuant polynomials of type $\alpha$}]
\mbox{}

\nin Inside $\Z[x_1,\ldots,x_k,y_1,\ldots,y_k]$, we extend the definition of polynomials of type $\alpha$ to $\alpha(n,k)$ for $n>k$ through the recurrence relation of Lemma \ref{recurrences}(2), denoting $x_{k+i}:=x_i$ for all $i\in\N^*$. For example, we have
\[\alpha(4,3)=x_1\alpha(3,3)+y_3\alpha(2,3) = x_1^2x_2x_3 + x_1y_1x_3 + x_1^2y_2 + x_1x_2y_3 + y_1y_3.\]
\end{defn}

\section{Determinant}\label{sectiondeterminant}

\begin{defn}[\n{Universal determinants}]
\mbox{}

\nin Given $k,n\in\N^*$ we call \e{the universal determinant of tridiagonal $k$-Toeplitz matrices of size $n$} to the polynomial $D(n,k)\in \Z[x_1,\ldots,x_k,y_1,\ldots,y_k]$ defined by
\[D(n,k):=\det(T^k_n(x_1,\ldots,x_k,1,\ldots,1,y_1,\ldots,y_k)).\]
We also define $D(0,k):=1, D(-1,k):=0$.\\
The shift by $s$ of universal determinants we write as $D_s(n,k):=(D(n,k))_s$.
\end{defn}

\begin{lem}[\n{Identities of universal determinants}]\label{detrecurrences}
\mbox{}

\nin Consider $\Z[x_1,\ldots,x_k,y_1,\ldots,y_k]$. Then:
\enum
\item For $n\in \N^*$,
\[D(n,k)=x_nD(n-1,k)-y_{n-1}D(n-2,k).\]
\item For $n\in \N^*$ and $s\in\N$,
\[D_s(n,k)=x_{s+1}D_{s+1}(n-1,k)-y_{s+1}D_{s+2}(n-2,k).\]
\eenum
\end{lem}

\newpage
\begin{proof}
\mbox{}

\enum
\item For $n\in\{1,2\}$ the result is true. For $n\geq3$, applying Laplace's expansion to $D(n,k)$ along the last column,
\[\vma
x_1 & 1 & &\\
y_1 & \ddots & \ddots \\
&\ddots& x_{n-1} & \bm{1}\\
 & & b_{n-1} & \bm{x_n}
\evma = x_n\vma
x_1 & 1 & &\\
y_1 & \ddots & \ddots \\
&\ddots& x_{n-2} & 1\\
 & & y_{n-2} & x_{n-1}
 \evma -\vma
x_1 & 1 & &\\
y_1 & \ddots & \ddots \\
&\ddots& x_{n-2} & 1\\
 & & 0 & y_{n-1}
\evma,\]
we get $D(n,k)= x_nD(n-1,k)-D'(n,k)$, with the auxiliary determinant $D'$ satisfying $D'(n,k)=y_{n-1}D(n-2,k)$ by Laplace's expansion along its last row:
\[\vma
x_1 & 1 & &\\
y_1 & \ddots & \ddots \\
&\ddots& x_{n-2} & 1\\
 & & \bm0 & \bm{y_{n-1}}
\evma = y_{n-1}\vma
x_1 & 1 & &\\
y_1 & \ddots & \ddots \\
&\ddots& x_{n-3} & 1\\
 & & y_{n-3} & x_{n-2}
\evma.\]
\item An argument similar to the previous one, applying Laplace's expansion to $D(n,k)$ along the first column, then Laplace's expansion to the resulting auxiliary determinant along the first row, gives
\[D(n,k)=x_1D_1(n-1,k)-y_1D_2(n-2,k).\]
Now, since the shift by $s$ is an automorphism of the polynomial ring, and shifts are additive, we get
$D_{s}(n,k)=(x_1)_{s}(D_1(n-1,k))_{s}-(y_1)_{s}(D_2(n-2,k))_{s} = x_{s+1}D_{s+1}(n-1,k)-y_{s+1}D_{s+2}(n-2,k).$\qedhere
\eenum
\end{proof}

The determinant of any tridiagonal $k$-Toeplitz matrix is an evaluation of a universal determinant.

\begin{lem}\label{reducevariables}
Let $K$ be a commutative unital ring. Given $k,n\in\N^*$ and $\ov a, \ov b:=(b_1,\ldots,b_k), \ov c:=(c_1,\ldots,c_k)\in K^k$, put $d_i:=b_ic_i$ for $1\leq i\leq k$ and $\ov d:=(d_1,\ldots,d_k)$. Then $\det(T^k_n(\ov a,\ov b,\ov c))=D^{\ov a,\ov d}(n,k)$.
\end{lem}

\begin{proof}
Consider the ring $R:=\Z(x_1,\ldots,x_k,y_1,\ldots,y_k,z_1,\ldots,z_k)$ and the matrix $T^k_n(x_1,\ldots,x_k,y_1,\ldots,y_k,z_1,\ldots,z_k)$. Let $\vp_i$ denote the operation on $A\in\Mat_n(R)$ consisting on first dividing the $i$th column of $A$ by $Y_i:=\prod_{j=1}^{i-1} y_i$, then multiplying the $i$th row of the resulting matrix by $Y_i$. We have $\det(A)=\det(\vp_i(A))$ for all $1\leq i\leq n$.  Then $\det(A)=\det(\vp_n(\vp_{n-1}(\cdots(\vp_2(A))\cdots)))$ implies
\[\det(T^k_n(x_1,\ldots,x_k,y_1,\ldots,y_k,z_1,\ldots,z_k)) = D^{\ov x,\ov w}(n,k)\]
with $\ov x:=(x_1,\ldots,x_k)$, $\ov w:=(y_1z_1,\ldots, y_kz_k)$. Therefore, by evaluation,\\
$\det(T^k_n(\ov a,\ov b,\ov c))=D^{\ov a,\ov d}(n,k)$.
\end{proof}

The solution over $K$ of the $n$th term of the $k$-periodic linear recurrence equation of second order $z_n=a_nz_{n-1}-b_{n-1}z_{n-2}$ with $a_n,b_n\in K$, $n\in\N^*$, $a_{k+i}=a_i, b_{k+i}=b_i$ for $i\in\N^*$, and initial conditions $z_{-1}:=0, z_0:=1$ is unique; since this recurrence relation is satisfied by both the evaluation $D^{\ov a,\ov b}(n,k)$ of the universal determinant of $k$-Toeplitz matrices of size $n$ (Lemma \ref{detrecurrences}(1)) and the evaluation $\alpha^{\ov a,\sminus\ov b}(n,k)$ of the polynomial of type $\alpha$ (Lemma \ref{recurrences}(2)), we get
\eqnum{D^{\ov a,\ov b}(n,k)=\alpha^{\ov a,\sminus\ov b}(n,k)\label{generaldeterminant}}
for all $n,k\in\N$. When $n\leq k$, this identity (through Lemma \ref{reducevariables}) gives the determinant of a general tridiagonal matrix. The next theorem shows that when $n>k$ the determinant of a tridiagonal $k$-Toeplitz matrix can, in addition, be written as a linear combination of two of the first $2k$ continuant polynomials of type $\alpha$ (the evaluations of $\alpha(0,k),\ldots,\alpha(2k-1,k)$), or equivalently, of the first $k$ continuant polynomials and first $k$ shifted continuant polynomials of type $\alpha$; a linear combination whose coefficients are, in essence, evaluations of generalized Fibonacci polynomials.

\begin{thm}[\n{Determinant of a tridiagonal $k$-Toeplitz matrix}]\label{determinant}
\mbox{}

\nin Let $K$ be a commutative unital ring. Given $k<n\in\N^*$ and $\ov a, \ov b:=(b_1,\ldots,b_k)$, $\ov c:=(c_1,\ldots,c_k)\in K^k$, put $d_i:=b_ic_i$ for $1\leq i\leq k$, $\ov d:=(d_1,\ldots,d_k)$ and $d:=d_1\cdots d_k$. Denote $U(i):=U_i(\pi^{\ov a,\sminus\ov d}(k,k),d)$ for $i\in\N$ and write $n=mk+r$ by Euclidean division. Then
\eqnum{
\det(T^k_n(\ov a,\ov b,\ov c)) = & \, U(m)\alpha^{\ov a,\sminus\ov d}(k+r,k) -dU(m-1)\alpha^{\ov a,\sminus\ov d}(r,k) \label{det1}\\
= & \, U(m+1)\alpha^{\ov a,\sminus\ov d}(r,k)+d_kd_1\cdots d_rU(m)\alpha_{r+1}^{\ov a,\sminus\ov d}(k-r-2,k) \label{det2}.
}
In particular, if $r=k-1$ then $\det(T^k_n(\ov a,\ov b,\ov c))$ is a multiple of $\det(T^k_{k-1}(\ov a,\ov b,\ov c))$:
\eqnum{\det(T^k_{mk+k-1}(\ov a,\ov b,\ov c))=U(m+1)\alpha^{\ov a,\sminus\ov d}(k-1,k) = U(m+1)\det(T^k_{k-1}(\ov a,\ov b,\ov c)).\label{det3}}
\end{thm}

\begin{proof}
By Lemma \ref{reducevariables} we have $\det(T^k_n(\ov a,\ov b,\ov c))=D^{\ov a,\ov d}(n,k)$. Put $a_0:=a_k, d_0:=d_k, d_{-1}:=d_{k-1}$. From the recurrence relation of universal determinants (Lemma \ref{detrecurrences}(1)) we get the $k$-periodic second-order linear recurrence equation
\eqnum{D^{\ov a,\ov d}(n,k)=a_rD^{\ov a,\ov d}(n-1,k)-d_{r-1}D^{\ov a,\ov d}(n-2,k),\label{(proof0)}}
which can be written in matrix form as
\eqnum{\pma D^{\ov a,\ov d}(n,k) \\ D^{\ov a,\ov d}(n-1,k)\epma = \pma a_r & -d_{r-1}\\ 1 & 0\epma\pma D^{\ov a,\ov d}(n-1,k) \\ D^{\ov a,\ov d}(n-2,k)\epma. \label{(proof1)}}
Define $A_s:=\pma a_s & -d_{s-1}\\ 1 & 0\epma$ for $0\leq s\leq k$ (note $A_0=A_k$), which has $\det(A_s)=d_{s-1}$, and $A:=A_kA_{k-1}\cdots A_1$.
Observe that we also have
\[\pma D^{\ov a,\ov d}(1,k) \\ D^{\ov a,\ov d}(0,k)\epma = \pma a_1 \\ 1\epma = \pma a_1 & -d_k\\ 1 & 0\epma\pma 1 \\ 0\epma = A_1\pma 1 \\ 0\epma.\]
Therefore, by recursion on \eqref{(proof1)}, for $n=mk+r$ we have
\eqnum{\pma D^{\ov a,\ov d}(n,k) \\ D^{\ov a,\ov d}(n-1,k)\epma = A_rA_{r-1}\cdots A_1 \cdot A^m\pma 1 \\ 0\epma, \label{(proof2)}}
understanding $A_r\cdots A_1=I_2$ when $r=0$.
Denote $A(s):=A_s\cdots A_1$ for $1\leq s\leq k$ (note $A(k)=A$). Let us show by induction on $s$ that
\eqnum{A(s)= \pma \alpha^{\ov a,\sminus\ov d}(s,k) & \beta^{\ov a,\sminus\ov d}(s+1,k)\\ \alpha^{\ov a,\sminus\ov d}(s-1,k) & \beta^{\ov a,\sminus\ov d}(s,k)\epma. \label{(proof3)}}
From Definitions \ref{polynomials}, in the ring $\Z[x_1,\ldots,x_{k+1},y_1,\ldots,y_k]$ we compute $\beta(2,k)=y_k$, $\pi(1,k)=x_1$, $\beta(1,k)=0$, $\alpha(1,k)=x_1$, $\alpha(0,k)=1$, so
\[\pma \alpha^{\ov a,\sminus\ov d}(1,k) & \beta^{\ov a,\sminus\ov d}(2,k)\\ \alpha^{\ov a,\sminus\ov d}(0,k) & \beta^{\ov a,\sminus\ov d}(1,k)\epma = \pma a_1 & -d_k\\ 1 & 0\epma = A_1 = A(1).\]
This shows the base case. Now assume that \eqref{(proof3)} is true for a fixed $1\leq s\leq k-1$. Then
\eq{
A(s+1)=&A_{s+1}A(s)= \pma a_{s+1} & -d_s\\ 1 & 0\epma
\pma \alpha^{\ov a,\sminus\ov d}(s,k) & \beta^{\ov a,\sminus\ov d}(s+1,k)\\ \alpha^{\ov a,\sminus\ov d}(s-1,k) & \beta^{\ov a,\sminus\ov d}(s,k)\epma = \\
=&\pma a_{s+1}\alpha^{\ov a,\sminus\ov d}(s,k) - d_s\alpha^{\ov a,\sminus\ov d}(s-1,k) & a_{s+1}\beta^{\ov a,\sminus\ov d}(s+1,k)-d_s\beta^{\ov a,\sminus\ov d}(s,k)\\ \alpha^{\ov a,\sminus\ov d}(s,k) & \beta^{\ov a,\sminus\ov d}(s+1,k)\epma = \\
=& \pma \alpha^{\ov a,\sminus\ov d}(s+1,k) & \beta^{\ov a,\sminus\ov d}(s+2,k)\\ \alpha^{\ov a,\sminus\ov d}(s,k) & \beta^{\ov a,\sminus\ov d}(s+1,k)\epma
}
by Lemma \ref{recurrences}(2,3), as we needed to show.\\
Now, by \eqref{(proof2)} and setting $A(0):=I_2$, we get
\eqnum{\pma  D^{\ov a,\ov d}(n,k) \\  D^{\ov a,\ov d}(n-1,k)\epma = A(r)A^m\pma 1 \\ 0\epma,\label{(proof4)}}
so $ D^{\ov a,\ov d}(n,k)$ equals the $(1,1)$ entry of $A(r)A^m$.
Since by \eqref{(proof3)}
\[A=A(k)= \pma \alpha^{\ov a,\sminus\ov d}(k,k) & \beta^{\ov a,\sminus\ov d}(k+1,k)\\ \alpha^{\ov a,\sminus\ov d}(k-1,k) & \beta^{\ov a,\sminus\ov d}(k,k)\epma,\]
we get $\tr(A)= \alpha^{\ov a,\sminus\ov d}(k,k)+\beta^{\ov a,\sminus\ov d}(k,k)=\pi^{\ov a,\sminus\ov d}(k,k)$. In addition we have $\det(A)=\det(A(k))=\prod_{s=1}^k\det(A_s)=d_1\cdots d_k = d$. Thus by Lemma \ref{powers}(2), if $U(i):=U_i(\pi^{\ov a,\sminus\ov d}(k,k),d)$ then $A^m=U(m)A-dU(m-1)I_2$, so the first column of $A^m$ is
\eqnum{\pma U(m)\alpha^{\ov a,\sminus\ov d}(k,k)-dU(m-1)\\ U(m)\alpha^{\ov a,\sminus\ov d}(k-1,k)\epma. \label{(proof5)}}
By \eqref{(proof3)}, if $r>0$ then the first row of $A(r)$ is
\eqnum{\pma \alpha^{\ov a,\sminus\ov d}(r,k) & \beta^{\ov a,\sminus\ov d}(r+1,k)\epma, \label{(proof6)}}
and the same holds for $A(0)=I_2$, since $\alpha^{\ov a,\sminus\ov d}(0,k)=1$, $\beta^{\ov a,\sminus\ov d}(1,k)=0$.
Using \eqref{(proof5)} and \eqref{(proof6)} in \eqref{(proof4)} we get
\eqnum{
&D^{\ov a,\ov d}(n,k) =\nonumber\\
&U(m)(\alpha^{\ov a,\sminus\ov d}(k,k)\alpha^{\ov a,\sminus\ov d}(r,k)+\alpha^{\ov a,\sminus\ov d}(k-1,k)\beta^{\ov a,\sminus\ov d}(r+1,k)) -dU(m-1)\alpha^{\ov a,\sminus\ov d}(r,k). \label{(proof7)}}
Observe that if $k\leq n<2k$ then $m=1$, so $U(1)=1, U(0)=0$ and we get
\eqnum{\alpha^{\ov a,\sminus\ov d}(k,k)\alpha^{\ov a,\sminus\ov d}(r,k)+\alpha^{\ov a,\sminus\ov d}(k-1,k)\beta^{\ov a,\sminus\ov d}(r+1,k)=D^{\ov a,\ov d}(k+r,k)=\alpha^{\ov a,\sminus\ov d}(k+r,k)\label{(proof7b)}}
because of \eqref{generaldeterminant}, giving Formula \eqref{det1}
\[D^{\ov a,\ov d}(n,k) =U(m)\alpha^{\ov a,\sminus\ov d}(k+r,k) -dU(m-1)\alpha^{\ov a,\sminus\ov d}(r,k).\]
In addition, we can rewrite \eqref{(proof7)} as
\eqnum{
&D^{\ov a,\ov d}(n,k) = \nonumber\\
&\overset{(1)}{=} \alpha^{\ov a,\sminus\ov d}(r,k)(U(m)(\pi^{\ov a,\sminus\ov d}(k,k)-\beta^{\ov a,\sminus\ov d}(k,k))-dU(m-1))+U(m)\alpha^{\ov a,\sminus\ov d}(k-1,k)\beta^{\ov a,\sminus\ov d}(r+1,k)) = \nonumber\\
&\overset{(2)}{=}\alpha^{\ov a,\sminus\ov d}(r,k)(U(m+1)-U(m)\beta^{\ov a,\sminus\ov d}(k,k)) +U(m)\alpha^{\ov a,\sminus\ov d}(k-1,k)\beta^{\ov a,\sminus\ov d}(r+1,k)) = \nonumber\\
& = U(m+1)\alpha^{\ov a,\sminus\ov d}(r,k) + U(m)(\alpha^{\ov a,\sminus\ov d}(k-1,k)\beta^{\ov a,\sminus\ov d}(r+1,k)- \alpha^{\ov a,\sminus\ov d}(r,k)\beta^{\ov a,\sminus\ov d}(k,k)) = \nonumber\\
&\overset{(3)}{=} U(m+1)\alpha^{\ov a,\sminus\ov d}(r,k) + U(m)d_kd_1\cdots d   _r\alpha_{r+1}^{\ov a,\sminus\ov d}(k-r-2,k),\label{(proof8)}
}
using (1) $\pi^{\ov a,\sminus\ov d}(k,k)=\alpha^{\ov a,\sminus\ov d}(k,k)+\beta^{\ov a,\sminus\ov d}(k,k)$, (2) the definition of $U(i)$ and the recurrence relation for the generalized Fibonacci polynomials (Lemma \ref{powers}(1)), and (3) Lemma \ref{recurrences}(5).
Finally, if $r=k-1$ then $\alpha^{\ov a,\sminus\ov d}_{r+1}(k-r-2,k)=\alpha^{\ov a,\sminus\ov d}_{r+1}(-1,k)=0$ and $D^{\ov a,\ov d}(n,k)=U(m+1)\alpha^{\ov a,\sminus\ov d}(r,k)$ by \eqref{(proof8)}.
\end{proof}

As seen from the proof of Theorem \ref{determinant}, the polynomials of type $\alpha$ in the second period can also be written as a function of the polynomials of types $\alpha,\beta$ in the first period.
\begin{cor}\label{secondperiod} For $0\leq r<k$,
\[\alpha(k+r,k)=\alpha(k,k)\alpha(r,k)+\alpha(k-1,k)\beta(r+1,k).\]
\end{cor}

\begin{proof}
Apply \eqref{(proof7b)} of the proof of Theorem \ref{determinant} to $K:=\Z[x_1,\ldots,x_k,y_1,\ldots,y_k]$.
\end{proof}

Probably there are no other cases apart from $r=k-1$, depending only on $k$ and $r$, for which the determinant of a general tridiagonal $k$-Toeplitz matrix factors\footnote{We have checked that the universal determinants of tridiagonal $k$-Toeplitz matrices of size $n$ are irreducible for $k,n\leq 10$ except when $r=k-1$.}. On the other hand, if some entry of the upper or lower main diagonal is zero, then the determinant factors.

\begin{cor}[\n{\small Determinant of a reducible tridiagonal $k$-Toeplitz matrix}]\label{reducibledet}
\mbox{}

\noindent Let $K$ be a commutative unital ring. Given $k,n\in\N^*$ and $\ov a, \ov b:=(b_1,\ldots,b_k), \ov c:=(c_1,\ldots,c_k)\in K^k$, write $n=mk+r$ by Euclidean division, put $d_i:=b_ic_i$ for $1\leq i\leq k$ and $\ov d:=(d_1,\ldots,d_k)$.
Suppose $b_{i_1}=\ldots=b_{i_q}=0$ ($q\geq1, n>i_1, i_1<i_2<\cdots<i_q\leq k$) or $c_{i_1}=\ldots=c_{i_q}=0$ and put $i_{q+1}:=i_1+k$, $i_0:=i_q$. Let $r':=r+k$, $p:=0$ if $r\leq i_1$ and $r':=r$ and $p\in\N^*$ be such that $i_p<r\leq i_{p+1}$ otherwise. Then
\eq{
&\det(T^k_n(\ov a,\ov b,\ov c))=
\alpha^{\ov a,\sminus\ov d}(i_1,k)\prod_{j=1}^q(\alpha_{i_j}^{\ov a,\sminus\ov d}(i_{j+1}-i_j,k))^{m+m_j}\alpha_{i_p}^{\ov a,\sminus\ov d}(r'-i_p,k)
}
with $m_j:=\left\{\begin{array}{cc} 1 & \text{if }j<p\\0  & \text{if }j\geq p\end{array}\right.$ for $1\leq j<q$ and $m_q:=\left\{\begin{array}{cc} 0 & \text{if }p>0\text{ or }m=0\\-1  & \text{if }p=0\end{array}\right.$.
\end{cor}

\begin{proof}
If $c_{i_1}=\ldots=c_{i_q}=0$ then $T^k_n(\ov a,\ov b,\ov c)^T$ is a tridiagonal $k$-Toeplitz matrix with zeros in the upper main diagonal satisfying $\det(T^k_n(\ov a,\ov b,\ov c)^T)=\det(T^k_n(\ov a,\ov b,\ov c))$, so without loss of generality we study only the case with $b_{i_1}=\ldots=b_{i_q}=0$. By Remark \ref{tridiagonalblocktriangular}, $T^k_n(\ov a,\ov b,\ov c)$ is a block lower triangular matrix with tridiagonal $k$-Toeplitz matrices as diagonal blocks, which repeat periodically. Concretely, the first diagonal block (since $n>i_1$) is $T^k_{i_1}(\ov a,\ov b,\ov c)$; then the array of diagonal blocks
\eq{
&T^k_{i_2-i_1}(\sigma_{i_1}(\ov a),\sigma_{i_1}(\ov b),\sigma_{i_1}(\ov c))), \ldots, T^k_{i_q-i_{q-1}}(\sigma_{i_{q-1}}(\ov a),\sigma_{i_{q-1}}(\ov b),\sigma_{i_{q-1}}(\ov c))),\\
&T^k_{k+i_1-i_q}(\sigma_{i_q}(\ov a),\sigma_{i_q}(\ov b),\sigma_{i_q}(\ov c)))
}
repeats $m-1$ times (if $m>1$, otherwise it does not appear); after that, the array
\[T^k_{i_2-i_1}(\sigma_{i_1}(\ov a),\sigma_{i_1}(\ov b),\sigma_{i_1}(\ov c))), \ldots, T^k_{i_q-i_{q-1}}(\sigma_{i_{q-1}}(\ov a),\sigma_{i_{q-1}}(\ov b),\sigma_{i_{q-1}}(\ov c)))\]
appears again if $m>0$ (it does not appear otherwise); and the last diagonal blocks depend on the relative position of $r$ with respect to the indices $i_1,\ldots,i_q$: if $r\leq i_1$ (what implies $m>0$) then only one block more appears, namely
\[T^k_{k+r-i_q}(\sigma_{i_q}(\ov a),\sigma_{i_q}(\ov b),\sigma_{i_q}(\ov c)),\]
while if on the contrary there is $p\in\N^*$ such that $i_p<r\leq i_{p+1}$ then the ending array of blocks is
\eq{
&T^k_{k+i_1-i_q}(\sigma_{i_q}(\ov a),\sigma_{i_q}(\ov b),\sigma_{i_q}(\ov c))), T^k_{i_2-i_1}(\sigma_{i_1}(\ov a),\sigma_{i_1}(\ov b),\sigma_{i_1}(\ov c))), \ldots\\
&\ldots,T^k_{i_p-i_{p-1}}(\sigma_{i_{p-1}}(\ov a),\sigma_{i_{p-1}}(\ov b),\sigma_{i_{p-1}}(\ov c))), T^k_{r-i_{p}}(\sigma_{i_{p}}(\ov a),\sigma_{i_{p}}(\ov b),\sigma_{i_{p}}(\ov c))).
}
By Theorem \ref{blockdeterminant} the determinant of a block triangular matrix is the product of the determinants of its diagonal blocks, and by \eqref{generaldeterminant}, for $s,t,k\in\N$,
\[\det(T^k_t(\sigma_s(\ov a),\sigma_s(\ov b),\sigma_s(\ov c)))=\alpha^{\ov a,\sminus\ov d}_s(t,k).\]
The result follows.
\end{proof}

\begin{rem}[\n{Da Fonseca and Petronilho's formula}]\label{FonsecaPetronilho}
\mbox{}

In \cite[Section 4]{FonsecaPetronilho2005}, Da Fonseca and Petronilho give a formula for the entries of the inverse of a complex irreducible tridiagonal $k$-Toeplitz matrix, based on a sequence of polynomials $\{Q_{nk+j}\}$ described in terms of determinants and Chebyshev polynomials. In this remark we show that their formula is a particular case of ours, by showing that $Q_{nk+j}$ is in fact the determinant of the tridiagonal $k$-Toeplitz matrix of order $nk+j$, and that the definitory formula of $Q_{nk+j}$, once the determinants are explicitly expressed as polynomials and the Chebyshev polynomials are related to generalized Fibonacci polynomials as in Remark \ref{Chebyshev}, coincides with Formula \eqref{det2} for the determinant given in Theorem \ref{determinant}.

Since the notation used in \cite{FonsecaPetronilho2005} diverges from ours, for ease of comprehension we include a table with the correspondence of notations:
\begin{center}\small\begin{tabular}{|c|c|c||c|c|c|}
  \hline
  Parameter & da Fonseca\&Petronilho & Brox\&Albuquerque & Parameter & d.F.\&P. & B.\&A.\\\hline\hline
  {Diagonal var.} & $z_i$ & $x_i$ or $a_i$  & Matrix size & $N$ & $n$ \\\hline
  {Nondiag. var.} & $w_i$ & $y_i$ or $d_i=b_ic_i$ & Period & $k$ & $k$\\\hline
  {Nondiag. product} & $w^2$ & $d$ & Quotient  & $n$ & $m$ \\\hline
  {Tridiag. det.} & $\Delta$ & $\alpha$ & Remainder & $j$ & $r$ \\\hline
\end{tabular}\end{center}
\newpage

In order to show the equivalence of formulas, we now translate the relevant objects from \cite{FonsecaPetronilho2005} to our setting. In the following any definition is as given in \cite{FonsecaPetronilho2005}. Fixed $k\in\N^*$, consider $z_1,\ldots,z_k\in\Co$, $w_1,\ldots,w_k\in\Co^*$ and $w^2:=\prod_{i=1}^k w_i$. For any $0\leq j<k$ they define
\[\Delta\left(\begin{array}{c}z_1,\ldots,z_j\\ w_1,\ldots,w_{j-1} \end{array};0\right):= \vma
z_1 & 1 & 0 & \cdots & \cdots & 0\\
w_1 & z_2 & 1 & \ddots  & \ddots  & \vdots\\
0 & w_2 & z_3 & \ddots & \ddots & 0\\
\vdots & \ddots & \ddots & \ddots  & \ddots & \vdots\\
\vdots & \ddots & \ddots & \ddots & z_{j-1} & 1\\
0 & \cdots & 0 & \cdots & w_{j-1} & z_j
\evma.\]
We put $\ov z:=(z_1,\ldots,z_k), \ov{w}:=(w_1,\ldots,w_k)\in \Co^k$.
We have, by \eqref{generaldeterminant},
\eqnum{\Delta\left(\begin{array}{c}z_1,\ldots,z_j\\ w_1,\ldots,w_{j-1} \end{array};0\right) = \alpha^{\ov z,\sminus\ov{w}}(j,k). \label{(C1)}}
Substituting $k-j-2$ in place of $j$ and shifting the determinant by $j+1$, in addition we get
\eqnum{\Delta\left(\begin{array}{c}z_{j+2},\ldots,z_{k-1}\\ w_{j+2},\ldots,w_{k-2} \end{array};0\right)= \alpha^{\ov z,\sminus\ov{w}}_{j+1}(k-j-2,k). \label{(C2)}}
They also define
\[\pi_k\left(\begin{array}{c}z_1,\ldots,z_k\\ w_1,\ldots,w_k \end{array};0\right):=
\vma
z_1 & 1 & 0 & \cdots & 0 & 0 & \bm 1 \\
w_1 & z_2 & 1 & \ddots  & 0  & 0 & 0\\
0 & w_2 & z_3 & \ddots & 0 & 0 & 0\\
\vdots & \vdots & \ddots & \ddots  & \ddots & \vdots & \vdots\\
0 & 0 & 0 & \ddots & z_{k-2} & 1 & 0\\
0 & 0 & 0 & \ddots & w_{k-2} & z_{k-1} & 1\\
\bm{w_k} & 0 & 0 & \cdots & 0 & w_{k-1} & z_k
\evma,\]
which can be translated to our setting through the following expansion:
\eq{
&\pi_k\left(\begin{array}{c}z_1,\ldots,z_k\\ w_1,\ldots,w_k \end{array};0\right) =\\
=&\vma
z_1 & 1 & 0 & \cdots & \bm1 \\
w_1 & z_2 & \ddots  & \ddots & \vdots\\
0 & \ddots & \ddots  & \ddots & 0\\
\vdots & \ddots & \ddots & z_{k-1} & 1\\
0 & \cdots & 0 & w_{k-1} & z_k
\evma
+ (-1)^{k+1}w_k\vma
1 & 0 & \cdots & 0 & \bm1 \\
z_2 & 1 & \ddots  & 0 & 0\\
\vdots & \ddots & \ddots  & \ddots & \vdots\\
0 & 0 & \ddots & 1 & 0\\
0 & 0 & \ddots & z_{k-1} & 1
\evma= \\
=&\vma
z_1 & 1 & 0 & \cdots & 0 \\
w_1 & z_2 & \ddots  & \ddots & \vdots\\
0 & \ddots & \ddots  & \ddots & 0\\
\vdots & \ddots & \ddots & z_{k-1} & 1\\
0 & \cdots & 0 & w_{k-1} & z_k
\evma
+(-1)^{k+1}\vma
w_1 & z_2 & \cdots  & 0\\
0 & \ddots & \ddots  & \vdots\\
\vdots & \ddots & \ddots & z_{k-1}\\
0 & \cdots & 0 & w_{k-1}
\evma+\\
&+(-1)^{k+1}w_k\left(\vma
1 & 0 & \cdots & 0 \\
z_2 & \ddots  & \ddots & \vdots\\
\vdots & \ddots & \ddots & 0\\
0 & \cdots & z_{k-1} & 1
\evma+(-1)^k
\vma
z_2 & 1 & 0 & \cdots & 0 \\
w_2 & \ddots & \ddots  & \ddots & \vdots\\
0 & \ddots & \ddots  & \ddots & 0\\
\vdots & \ddots & \ddots & \ddots & 1\\
0 & \cdots & 0 & w_{k-1} & z_{k-1}
\evma\right)=\\
=&\alpha^{\ov z,\sminus\ov{w}}(k,k) +(-1)^{k+1}w^2/w_k +(-1)^{k+1}w_k -w_k\alpha_1^{\ov z,\sminus\ov{w}}(k-2,k).
}
Thus, since $-w_k\alpha_1^{\ov z,\sminus\ov{w}}(k-2,k) = \beta^{\ov z,\sminus\ov{w}}(k,k)$ by Lemma \ref{recurrences}(5) (with $r=0$), and $\alpha^{\ov z,\sminus\ov{w}}(k,k)+\beta^{\ov z,\sminus\ov{w}}(k,k)=\pi^{\ov z,\sminus\ov{w}}(k,k)$, we get
\eqnum{\pi_k\left(\begin{array}{c}z_1,\ldots,z_k\\ w_1,\ldots,w_k \end{array};0\right) +(-1)^k(w_k+w^2/w_k) = \pi^{\ov z,\sminus\ov{w}}(k,k). \label{(C3)}}
For $n\in\N$ they define
\[U_{n,k}\left(\begin{array}{c}z_1,\ldots,z_k\\ w_1,\ldots,w_k \end{array};0\right):=w^n U_n\left(\frac1{2w}\left(\pi_k\left(\begin{array}{c}z_1,\ldots,z_k\\ w_1,\ldots,w_k \end{array};0\right)+(-1)^k(w_k+w^2/w_k)\right)\right)\]
where $U_m(x)$ denotes as before the $m$th Chebyshev polynomial of the second kind. Therefore, by Formulas \eqref{(C3)} and \eqref{(B)} we see that
\eqnum{U_{n,k}\left(\begin{array}{c}z_1,\ldots,z_k\\ w_1,\ldots,w_k \end{array};0\right) = U_{n+1}(\pi^{\ov z,\sminus\ov{w}}(k,k),w^2) \label{(C4)}}
with $U_m(x,y)$ the generalized Fibonacci polynomial of order $m$, as before.
Lastly, for $n\in\N^*$ and $0\leq j<k$ they define
\eq{
Q_{nk+j}\left(\begin{array}{c}z_1,\ldots,z_k\\ w_1,\ldots,w_k \end{array};0\right):=&
\Delta\left(\begin{array}{c}z_1,\ldots,z_j\\ w_1,\ldots,w_{j-1} \end{array};0\right)U_{n,k}\left(\begin{array}{c}z_1,\ldots,z_k\\ w_1,\ldots,w_k \end{array};0\right) +\\
 +w_kw_1\cdots w_j&\Delta\left(\begin{array}{c}z_{j+2},\ldots,z_{k-1}\\ w_{j+2},\ldots,w_{k-2} \end{array};0\right)U_{n-1,k}\left(\begin{array}{c}z_1,\ldots,z_k\\ w_1,\ldots,w_k \end{array};0\right).
}
Putting $U(i):=U_i(\pi^{\ov z,\sminus\ov{w}}(k,k),w^2)$, by \eqref{(C1)}, \eqref{(C2)} and \eqref{(C4)} we arrive at
\[Q_{nk+j}\left(\begin{array}{c}z_1,\ldots,z_k\\ w_1,\ldots,w_k \end{array};0\right)=
\alpha^{\ov z,\sminus\ov{w}}(j,k)U(n+1) +w_kw_1\cdots w_j\alpha^{\ov z,\sminus\ov{w}}_{j+1}(k-j-2,k)U(n)\]
which equals $\det(T^k_{nk+j}(\ov z,\ov b,\ov c))$ by Formula \eqref{det2} of Theorem \ref{determinant}, for any $\ov b:=(b_1,\ldots,b_k),\ov c:=(c_1,\ldots,c_k)\in\Co^k$ such that $b_ic_i=w_i$ for $1\leq i\leq k$.
\end{rem}

\smallskip

In the same vein, it can be easily shown that Rózsa's formula for the determinant of a symmetric irreducible tridiagonal $k$-Toeplitz matrix over $\Co$ (\cite[Theorem]{Rozsa1969}) is a particular case of Formula \eqref{det2}.

\section{Spectral properties}\label{sectionspectralproperties}

We show how to compute the characteristic polynomial and some eigenvectors of a tridiagonal $k$-Toeplitz matrix through determinants of other related tridiagonal $k$-Toeplitz matrices.

\begin{cor}[\n{\small Characteristic polynomial of a tridiagonal $k$-Toeplitz matrix}]\label{charpolcor}
\mbox{}

\nin Let $K$ be a commutative unital ring. Given $k,n\in\N^*$ and $\ov a, \ov b:=(b_1,\ldots,b_k)$, $\ov c:=(c_1,\ldots,c_k)\in K^k$, put $d_i:=b_ic_i$ for $1\leq i\leq k$ and $\ov d:=(d_1,\ldots,d_k)$.\\
If $n\leq k$ (i.e, if we are considering a general tridiagonal matrix) then
\[p_{T^k_n(\ov a,\ov b,\ov c)}(x) = \alpha^{\ov x-\ov a,\sminus\ov d}(n,k).\]
If $n>k$ write $n=mk+r$ by Euclidean division, put $d:=d_1\cdots d_k$, and denote $U^x(i):=U_i(\pi^{\ov x-\ov a,\sminus\ov d}(k,k),d)$ for $i\in\N$;
then
\eqnum{
p_{T^k_n(\ov a,\ov b,\ov c)}(x) =&U^x(m)\alpha^{\ov x-\ov a,\sminus\ov d}(k+r,k) -dU^x(m-1)\alpha^{\ov x-\ov a,\sminus\ov d}(r,k) \label{charpol1}\\
=&U^x(m+1)\alpha^{\ov x-\ov a,\sminus\ov d}(r,k)+d_kd_1\cdots d_rU^x(m)\alpha_{r+1}^{\ov x-\ov a,\sminus\ov d}(k-r-2,k). \label{charpol2}
}
In particular, if $r=k-1$ then $p_{T^k_n(\ov a,\ov b,\ov c)}$ factors as
\[p_{T^k_n(\ov a,\ov b,\ov c)}(x) = U^x(m+1)\alpha^{\ov x-\ov a,\sminus\ov d}(k-1,k).\]
Moreover, if $T^k_n(\ov a,\ov b,\ov c)$ is reducible with $b_{i_1}=\ldots=b_{i_q}=0$ ($q\geq1, n>i_1, i_1<i_2<\cdots<i_q$) or $c_{i_1}=\ldots=c_{i_q}=0$, putting $i_{q+1}:=i_1+k$, $i_0:=i_q$, $r':=r+k$ and $p:=0$ if $r\leq i_1$ or $r':=r$ and $p\in\N^*$ such that $i_p<r\leq i_{p+1}$ otherwise, then $p_{T^k_n(\ov a,\ov b,\ov c)}$ factors as
\eq{
&p_{T^k_n(\ov a,\ov b,\ov c)}=
\alpha^{\ov x-\ov a,\sminus\ov d}(i_1,k)\prod_{j=1}^q(\alpha_{i_j}^{\ov x-\ov a,\sminus\ov d}(i_{j+1}-i_j,k))^{m+m_j}\alpha_{i_p}^{\ov x-\ov a,\sminus\ov d}(r'-i_p,k)
}
with $m_j:=\left\{\begin{array}{cc} 1 & \text{if }j<p\\0  & \text{if }j\geq p\end{array}\right.$ for $1\leq j<q$ and $m_q:=\left\{\begin{array}{cc} 0 & \text{if }p>0\text{ or }m=0\\-1  & \text{if }p=0\end{array}\right.$.
\end{cor}

\begin{proof}
Since \eqref{generaldeterminant} and Theorem \ref{determinant} give formulas for the determinant of any tridiagonal $k$-Toeplitz matrix over any commutative unital ring, they can be used to compute $p_{T^k_n(\ov a,\ov b,\ov c)}=\det(xI_n-T^k_n(\ov a,\ov b,\ov c))=\det(T^k_n(\ov x-\ov a,-\ov b,-\ov c))$ over $K[x]$. Moreover, if $T^k_n(\ov a,\ov b,\ov c)$ is reducible with $b_{i_1}=\ldots=b_{i_q}=0$ or $c_{i_1}=\ldots=c_{i_q}=0$ then the same happens to $T^k_n(\ov x-\ov a,-\ov b,-\ov c)$, and Corollary \ref{reducibledet} applies.
\end{proof}

\begin{rem}[\n{Finding the eigenvalues in the case $r=k-1$}]\label{charpolfactorizesspecialcase}
\mbox{}

The factorization of the characteristic polynomial found in Corollary \ref{charpolcor} in the special case $r=k-1$,
\eqnum{p_{T^k_n(\ov a,\ov b,\ov c)}(x) = U_{m+1}(\pi^{\ov x-\ov a,\sminus\ov d}(k,k),d)\alpha^{\ov x-\ov a,\sminus\ov d}(k-1,k),\label{factored}}
generalizes the same phenomenon shown in the literature for the complex numbers many times before (e.g. \cite[(7)]{EgervarySzasz1928},\cite[Corollary]{Rozsa1969},\cite[Remark 5]{ElsnerRedheffer1967},\cite[Theorem 5.1]{FonsecaPetronilho2005},\cite[(2.6),(2.12)]{AlvarezNodarsePetronilhoQuintero}), where it is used to ease the computation of the eigenvalues of the matrix. Specifically, if $d\neq0$, writing the factorization in terms of a Chebyshev polynomial of the second kind (see Remark \ref{Chebyshev}),
\eqnum{p_{T^k_n(\ov a,\ov b,\ov c)}(x) = (\sqrt{d})^mU_{m}(\pi^{\ov x-\ov a,\sminus\ov d}(k,k)/2\sqrt d)\alpha^{\ov x-\ov a,\sminus\ov d}(k-1,k),\label{factored2}}
we find that the eigenvalues are either roots of $\alpha^{\ov x-\ov a,\sminus\ov d}(k-1,k)$ (a polynomial of degree $k-1$) or roots of $U_{m}(\rho(x))$ (a polynomial of degree $n-k+1$) with $\rho(x):=\pi^{\ov x-\ov a,\sminus\ov d}(k,k)/(2\sqrt d)$; in particular the roots of the Chebyshev polynomial can be found by solving the $m=\lfloor n/k\rfloor$ polynomial equations of degree $k$
\[\rho(x)=\cos\frac{i\pi}{m+1}\text{ for }1\leq i \leq m\]
 (by the well-known trigonometric properties of Chebyshev polynomials over $\Co$).\\
 More in general, if $K$ is an integral domain with $\ch(K)\neq2$ then the factorization in \eqref{factored2} can be used to subdivide the problem in the same way (with solutions in an algebraic closure $\ov{Q(K)}$ of its field of fractions). Over an arbitrary commutative unital ring $K$ we can still use \eqref{factored} to subdivide the problem, since if $p_{T^k_n(\ov a,\ov b,\ov c)}(\lambda)$ is a zero divisor of $K$ for $\lambda\in K$ then $U_{m+1}(\pi^{\ov \lambda-\ov a,\sminus\ov d}(k,k),d)$ or $\alpha^{\ov \lambda-\ov a,\sminus\ov d}(k-1,k)$ must be a zero divisor of $K$.
\end{rem}

\begin{thm}[\n{Eigenvectors of a tridiagonal $k$-Toeplitz matrix}]\label{eigenvectors}
\mbox{}

\nin Let $\lambda$ be an eigenvalue of $T^k_n(\ov a,\ov b,\ov c)$ and pick $0\neq z\in\Ann(p_{T^k_n(\ov a,\ov b,\ov c)}(\lambda))$. Put $d_i:=b_ic_i$, $\ov d:=(d_1,\ldots,d_k)$.
\enuma
\item For $1\leq i\leq n$ define
\[v^z_i(\lambda):=z\prod_{j=i}^{n-1}b_j \cdot D^{\ov\lambda-\ov a,\ov d}(i-1,k).\]
If $v^z_i(\lambda)\neq0$ for some $1\leq i\leq n$ then $(v^z_1(\lambda),\ldots,v^z_n(\lambda))$ is an eigenvector of $T^k_n(\ov a,\ov b,\ov c)$ associated to $\lambda$.
\item For $1\leq i\leq n$ define
\[w^z_i(\lambda):=z\prod_{j=1}^{i-1}c_j \cdot D_i^{\ov\lambda-\ov a,\ov d}(n-i,k).\]
If $w^z_i(\lambda)\neq0$ for some $1\leq i\leq n$ then $(w^z_1(\lambda),\ldots,w^z_n(\lambda))$ is an eigenvector of $T^k_n(\ov a,\ov b,\ov c)$ associated to $\lambda$.
\eenum
\end{thm}

\begin{proof}
The equation $T^k_n(\ov a,\ov b,\ov c)v=\lambda v$ for $v=(v_1,\ldots,v_n)\in K^k$ is equivalent to the system of equations \eqnum{b_iv_{i+1}=(\lambda-a_i)v_i-c_{i-1}v_{i-1}, \,\, 1\leq i\leq n\label{eigenproof}}
with $c_0:=0, b_n:=0$. Denote $v^z_i:=v^z_i(\lambda)$. Since by Lemma \ref {detrecurrences}(1)
\[D^{\ov\lambda-\ov a,\ov d}(i,k)=(\lambda-a_i)D^{\ov\lambda-\ov a,\ov d}(i-1,k)-b_{i-1}c_{i-1}D^{\ov\lambda-\ov a,\ov d}(i-2,k)\]
for all $i\in\N^*$, we have
\eq{
b_iv^z_{i+1} &= zb_i\prod_{j=i+1}^{n-1}b_j \cdot D^{\ov\lambda-\ov a,\ov d}(i,k) = \\
& = z\prod_{j=i}^{n-1}b_j \cdot\left((\lambda-a_{i})D^{\ov\lambda-\ov a,\ov d}(i-1,k)-b_{i-1}c_{i-1}D^{\ov\lambda-\ov a,\ov d}(i-2,k)\right) =\\
& = (\lambda-a_i)z\prod_{j=i}^{n-1}b_j \cdot D^{\ov\lambda-\ov a,\ov d}(i-1,k) -c_{i-1}z\prod_{j=i-1}^{n-1}b_j \cdot D^{\ov\lambda-\ov a,\ov d}(i-2,k) = \\ &=(\lambda-a_i)v^z_i - c_{i-1}v^z_{i-1}
}
for $1\leq i\leq n-1$. Moreover, since $\lambda$ is an eigenvalue of $T_n^k(\ov a,\ov b,\ov c)$ and $z\in\Ann(p_{T_n^k(\ov a,\ov b,\ov c)}(\lambda))$, for $i=n$ we get
\eq{
&(\lambda-a_n)v^z_n - c_{n-1}v^z_{n-1} = z((\lambda-a_n)D^{\ov\lambda-\ov a,\ov d}(n-1,k)-b_{n-1}c_{i-1}D^{\ov\lambda-\ov a,\ov d}(n-2,k)) = \\
& = zD^{\ov\lambda-\ov a,\ov d}(n,k) = zp_{T_n^k(\ov a,\ov b,\ov c)}(\lambda)=0.
}
Therefore $v_1^z,\ldots,v_n^z$ comprise a solution of Equations \eqref{eigenproof}, and generate an eigenvector associated to $\lambda$ if the solution is nontrivial.\\
Now denote $w^z_i:=w^z_i(\lambda)$. Similarly, since by Lemma \ref {detrecurrences}(2)
\[D_{i-1}^{\ov\lambda-\ov a,\ov d}(n-i+1,k)=(\lambda-a_{i})D_{i}^{\ov\lambda-\ov a,\ov d}(n-i,k)-b_{i}c_{i}D_{i+1}^{\ov\lambda-\ov a,\ov d}(n-i-1,k)\]
for $1\leq i\leq n$, we have
\eq{
c_{i-1}w^z_{i-1} &= zc_{i-1}\prod_{j=1}^{i-2}c_j \cdot D_{i-1}^{\ov\lambda-\ov a,\ov d}(n-i+1,k) = \\
& = z\prod_{j=1}^{i-1}c_j \cdot\left((\lambda-a_{i})D_i^{\ov\lambda-\ov a,\ov d}(n-i,k)-b_{i}c_{i}D_{i+1}^{\ov\lambda-\ov a,\ov d}(n-i-1,k)\right) =\\
& = (\lambda-a_i)z\prod_{j=1}^{i-1}c_j \cdot D_i^{\ov\lambda-\ov a,\ov d}(n-i,k) -b_{i}z\prod_{j=1}^{i}c_j \cdot D_{i+1}^{\ov\lambda-\ov a,\ov d}(n-i-1,k) = \\ &=(\lambda-a_i)w^z_i - b_{i}w^z_{i+1}
}
for $2\leq i\leq n$. Moreover, since $\lambda$ is an eigenvalue of $T_n^k(\ov a,\ov b,\ov c)$ and $z\in\Ann(p_{T_n^k(\ov a,\ov b,\ov c)}(\lambda))$, for $i=1$ we get
\eq{
&(\lambda-a_1)w^z_1 - b_1w^z_2 = z((\lambda-a_1)D_1^{\ov\lambda-\ov a,\ov d}(n-1,k)-b_1c_1D_2^{\ov\lambda-\ov a,\ov d}(n-2,k)) = \\
& = zD_0^{\ov\lambda-\ov a,\ov d}(n,k) = zD^{\ov\lambda-\ov a,\ov d}(n,k) = zp_{T_n^k(\ov a,\ov b,\ov c)}(\lambda)=0.
}
Therefore $w_1^z,\ldots,w_n^z$ comprise a solution of Equations \eqref{eigenproof}, and generate an eigenvector associated to $\lambda$ if the solution is nontrivial.
\end{proof}

\begin{rem}[\n{Theorem \ref{eigenvectors} generalizes previous results}]\label{eigenvectorsgeneralizespreviousresults}
\mbox{}

The formulas given in Theorem \ref{eigenvectors} for the eigenvectors of a tridiagonal $k$-Toeplitz matrix of order $n$ over any commutative unital ring generalize those for irreducible complex matrices found in \cite[Theorems 2.1 and 2.2]{AlvarezNodarsePetronilhoQuintero} for $n=2,3$, in which the condition $b_i,c_i>0$ for all $1\leq i\leq k$ is imposed. In fact, to get a (nonzero) eigenvector in Theorem \ref{eigenvectors}a) (resp. b)) over $\Co$, or more in general over any integral domain, it is enough that $b_i\neq0$ (resp. $c_i\neq0$) for all $1\leq i\leq k$, for in that case $v_1^1(\lambda)=\prod_{j=1}^{n-1} b_i$ (resp. $w_n^1(\lambda)=\prod_{j=1}^{n-1} c_i$) is nonzero. Over an arbitrary commutative unital ring, to get $v_1^z(\lambda)\neq0$ (resp. $w_n^z(\lambda)\neq0$) it is enough that $b_i$ (resp. $c_i$) is a regular element for all $1\leq i\leq k$.
\end{rem}

If the tridiagonal $k$-Toeplitz matrix is reducible, Theorem \ref{eigenvectors} still applies, but the computations can be simplified.

\newpage
\begin{rem}[\small\n{Eigenvectors of a reducible tridiagonal $k$-Toeplitz matrix}]
\mbox{}

Suppose $b_i=0$ (the reasonings are analogous for block upper triangular matrices). Then $T_n^k(\ov a,\ov b,\ov c)$ is block triangular with associated partition of size $2$ and diagonal matrices, say, $T_1$ of order $i$ and $T_2$ of order $n-i$, which are again tridiagonal $k$-Toeplitz matrices (see Remark \ref{tridiagonalblocktriangular}), and bottom-left block $C:=\pma 0 & \cdots & c_i \\ \vdots & \ddots & \vdots \\ 0 & \cdots & 0\epma$. By Lemma \ref{blockeigenvalues}, an eigenvalue $\lambda$ of $T_n^k(\ov a,\ov b,\ov c)$ is an eigenvalue of $T_1$ or of $T_2$. The equation by blocks for an eigenvector $(v_1,v_2)$, $v_1\in K^i, v_2\in K^{n-i}$, of $T_n^k(\ov a,\ov b,\ov c)$ associated to $\lambda$ is
\eqnum{\pma T_1 & 0\\ C & T_2\epma\pma v_1\\v_2\epma = \lambda\pma v_1\\v_2\epma \Leftrightarrow \left\{\begin{array}{c}T_1v_1 = \lambda v_1 \\ Cv_1 + T_2v_2 = \lambda v_2\end{array}\right\}.\label{eigenequation}}
\enuma
\item Suppose $\lambda$ is an eigenvalue of $T_2$. We see from \eqref{eigenequation} that if $v_2$ is an eigenvector of $T_2$ associated to $\lambda$, then $(0,v_2)$ is an eigenvector of $T_n^k(\ov a,\ov b,\ov c)$ associated to $\lambda$.
\item Suppose now that $\lambda$ is an eigenvalue of $T_1$ which is not an eigenvalue of $T_2$.\footnote{If $\lambda$ is an eigenvalue of both $T_1$ and $T_2$ then item a) applies.} This implies that $\det(\lambda I_{n-i}-T_2)$ is a regular element of $K$, so it is a unit of $Q(K)$ and $\lambda I_{n-i}-T_2$ is invertible over $Q(K)$. Then if $v_1$ is an eigenvector of $T_1$ associated to $\lambda$ and $v_2=(\lambda I_{n-i}-T_2)^{-1}Cv_1$, we see from \eqref{eigenequation} that $(v_1,v_2)$ is an eigenvector of $T_n^k(\ov a,\ov b,\ov c)$ associated to $\lambda$ over $Q(K)$, which generates an eigenvector over $K$ after multiplication by the denominators of all entries of $(\lambda I_{n-i}-T_2)^{-1}$.
\item  Moreover, if $\lambda$ is an eigenvalue of $T_1$ and $c_i$ is a zero divisor, then $(zv_1,0)$ with $v_1$ an eigenvector of $T_1$ associated to $\lambda$ and $0\neq z\in\Ann(c_i)$ is an eigenvector of $T_n^k(\ov a,\ov b,\ov c)$ if $zv_1\neq0$.
\eenum
Associated eigenvectors of $T_1$ or $T_2$ may then be computed as in Theorem \ref{eigenvectors}. Now let $(A_1,\ldots,A_q)$ be the array of diagonal blocks of the finest partition showing $T_n^k(\ov a,\ov b,\ov c)$ as block lower triangular; then these blocks repeat in a specific pattern (see the proof of Theorem \ref{eigenproof}). Different choices of $T_1$ and $T_2$ to cover this array will produce different possibilities with different matrix sizes to which apply Theorem \ref{eigenvectors}. In addition, if $\lambda$ is an eigenvalue of $T_n^k(\ov a,\ov b,\ov c)$ then it is an eigenvalue of some $A_i$ (Lemma \ref{blockeigenvalues}) and of all of its copies; so, in order to apply item b) for $\lambda$, it is necessary that all copies of $A_i$ are considered inside $T_1$. Suppose that $\lambda$ is an eigenvalue of $A_i$ but not of any other $A_j$ except for the copies of $A_i$; then it is possible to compute an eigenvector of $T_n^k(\ov a,\ov b,\ov c)$ associated to $\lambda$ by applying Theorem \ref{eigenvectors} only to $A_i$: Let $A_p$ be the last copy of $A_i$ in $(A_1,\ldots,A_q)$, and consider the partition of $T_n^k(\ov a,\ov b,\ov c)$ of size $2$ with $T_1$ covering $(A_1,\ldots,A_{p-1})$ and $T_2$ covering $(A_p,\ldots,A_q)$; we apply item a) to dispose of $T_1$, whence we need to compute an eigenvector of $T_2$ associated to $\lambda$, for which we apply item b) to $T_2$ with partition $T'_1=A_p$ and $T'_2$ covering $(A_{p+1},\ldots,A_q)$, getting an eigenvector of $A_p$ associated to $\lambda$ by an application of Theorem \ref{eigenvectors}.
\end{rem}

\section{Inverse}\label{sectioninverse}

The formula of the next theorem generalizes those found in \cite{Wittenburg1998} and \cite{FonsecaPetronilho2005} for irreducible matrices over the complex numbers, and coincides with \cite[Item 555]{MuirMetzler1933} generalized to $k$-Toeplitz matrices. To prove it, we write the entries of the inverse of a nonsingular tridiagonal $k$-Toeplitz matrix in terms of determinants of smaller tridiagonal $k$-Toeplitz matrices. This can be done because the submatrices giving rise to its cofactors are block triangular, with diagonal blocks which are either triangular or tridiagonal $k$-Toeplitz matrices.

\begin{thm}[\n{Inverse of a tridiagonal $k$-Toeplitz matrix}]\label{inverse}
\mbox{}

\nin Let $K$ be a commutative unital ring. Fixed $n\in\N$, $k\in\N^*$, $\ov a, \ov b:=(b_1,\ldots,b_k)$, $\ov c:=(c_1,\ldots,c_k)\in K^k$, put $d_i:=b_ic_i$ for $1\leq i\leq k$ and $\ov d:=(d_1,\ldots,d_k)$. If $D^{\ov a,\ov d}(n,k)$ is a unit of $K$ then $T^k_n(\ov a,\ov b,\ov c)$ is invertible and the $(i,j)$ entry of its inverse is given by
\[(-1)^{i+j}\prod_{p=i}^{j-1}b_p\prod_{p=j}^{i-1}c_p\frac{\displaystyle D^{\ov a,\ov d}(\min(i,j)-1,k)D_{\max(i,j)}^{\ov a,\ov d}(n-\max(i,j),k)}{\displaystyle D^{\ov a,\ov d}(n,k)}.\]
\end{thm}

\begin{proof}
Denote $T^k_n:=T^k_n(\ov a,\ov b,\ov c)$ and suppose $\det(T^k_n)=D^{\ov a,\ov d}(n,k)$ is a unit of $K$. We compute its inverse through its adjugate matrix, so that the $(r,s)$ entry of $(T^k_n)^{-1}$ is
\eqnum{\frac{C_{sr}}{\det(T^k_n)},\label{adjugateinv}}
where $C_{sr}:=(-1)^{r+s}\det(A_{sr})$ is the cofactor obtained from the submatrix $A_{sr}$ of $T^k_n$ formed by removing the $s$th row and the $r$th column. Thus if $T^k_n=(t_{ij})_{i,j=1}^n$ and $A_{rs}=(a_{ij})_{i,j=1}^{n-1}$ we have
\[a_{ij}=\left\{\begin{array}{cl}
t_{ij}, &\text{if }i<r, j<s\\
t_{i+1,j}, &\text{if }i\geq r, j<s\\
t_{i,j+1}, &\text{if }i<r, j\geq s\\
t_{i+1,j+1}, &\text{if }i\geq r, j\geq s\\
\end{array}\right..\]
Since $t_{ij}=0$ if $|j-i|\geq2$, this implies that $A_{rs}$ is a block upper triangular matrix when $r\leq s$ and a block lower triangular matrix when $r\geq s$, with three diagonal blocks. When $r\leq s$ we have
\[A_{rs}=\left(\begin{array}{c|ccc|c}
\bm{A_1} & b_{r-1} & & & \,\,\,_{_{r\text{th row}}}\!\!\\\hline
& \bm{c_r} & a_{r+1} & &\\
& & \bm{c_{r+1}} & \ddots & \\
& & & \bm{\ddots} & b_s \\\hline
& & & _{_{s\text{th col}}}\!\! & \bm{A_2}
\end{array}\right)
\]
where $A_1, A_2$ are again tridiagonal $k$-Toeplitz matrices, concretely
$A_1=T^k_{r-1}(\ov a,\ov b,\ov c)$ and $A_2=T^k_{n-s}(\sigma_s(\ov a),\sigma_s(\ov b),\sigma_s(\ov c))$.
Since the middle diagonal block $A_c$ is upper triangular with main diagonal composed from the lower main diagonal of $T^k_n$, its determinant is
\[\det(A_c)=c_rc_{r+1}\cdots c_{s-1}.\]
Analogously, when $r\geq s$ we get $T^k_s(\ov a,\ov b,\ov c)$, $A_b$, and $T^k_{n-r}(\sigma_r(\ov a),\sigma_r(\ov b),\sigma_r(\ov c))$ as the diagonal blocks of the partition of $A_{rs}$, with $A_b$ lower triangular with main diagonal composed from the upper main diagonal of $T^k_n$ and determinant $\det(A_b)=b_s\cdots b_{r-1}$. Thus, since the determinant of a block triangular matrix is the product of the determinants of its diagonal blocks (Theorem \ref{blockdeterminant}), we get
\[\det(A_{rs})=\left\{\begin{array}{cl}
c_r\cdots c_{s-1} \det(T^k_{r-1}(\ov a,\ov b,\ov c)) \det(T^k_{n-s}(\sigma_s(\ov a),\sigma_s(\ov b),\sigma_s(\ov c)), &\text{if }r\leq s\\
b_s\cdots b_{r-1} \det(T^k_{s-1}(\ov a,\ov b,\ov c)) \det(T^k_{n-r}(\sigma_r(\ov a),\sigma_r(\ov b),\sigma_r(\ov c))) &\text{if }s\leq r\\
\end{array}\right..\]
The result now follows from \eqref{adjugateinv}.
\end{proof}

Clearly, if $\det(T_n^k(\ov a,\ov b,\ov c))$ is not a unit of $K$ but is a regular element, then $T_n^k(\ov a,\ov b,\ov c)$ is invertible over $Q(K)$ and Theorem \ref{inverse} still applies.

\section{Algorithms and complexity analysis}\label{complexitysection}
\numberwithin{equation}{subsection}
\renewcommand\labelenumi{\arabic{enumi}.}
\renewcommand\theenumi{\thesubsection.\arabic{enumi}}

Our previous results provide formulas for the determinant, spectral properties, and elements of the inverse of a tridiagonal $k$-Toeplitz matrix. We now conduct a worst-case algebraic complexity analysis on algorithms based on those formulas, showing that they are not only of theoretical value, but can also be used to produce efficient computations. We also show the complexity for general (non-periodic) tridiagonal matrices, and compare the efficiency of our formulas against those of Da Fonseca-Petronilho and of Lewis.

In the following, the complexity functions of the algorithms depend in general on the parameters $n,m,k,r$, where $n=mk+r$ by Euclidean division. In order to compare efficiencies, we consider these parameters ordered as they are listed above (see Section \ref{introcomplexity}), i.e., when comparing efficiencies we give priority to $n$ and $m$ over $k$, considering a priori $k$ as a fixed constant, and to $k$ over $r$.

\subsection{Computation of generalized Fibonacci polynomials}\label{GeneralizedFibonaccicomplexity}
\mbox{}

We study the complexity of computing $U_m(x,y)$ and $U_{m-1}(x,y)$ together, as needed in some of the algorithms. We can compute $U_m(x,y)$ with $x,y\in K$ by iterating the recurrence relation (Lemma \ref{powers}(1))
\[U_m(x,y)=xU_{m-1}(x,y) - yU_{m-2}(x,y)\]
starting from $U_0=0$, $U_1=1$. This computation requires $3(m-1)$ operations in $K$ and includes the value of $U_{m-1}(x,y)$ as a subproduct, so its complexity is
\[\C_{\text{recurrence}}(U_m(x,y),U_{m-1}(x,y))=3(m-1).\]
This computation is more efficient than the evaluation over the definition in \ref{Fibonacci}.\footnote{Let $N:=\lfloor(m-1)/2\rfloor+1$. The evaluation of $U_m(x,y)$ from $U_m(x,y)=\sum_{i=0}^{N-1} (-1)^i\binom{m-1-i}i x^{m-1-2i}y^i$ requires the computation of $N$ powers of $x$, $N$ powers of $y$, $N-1$ binomial coefficients (which can be recurrently computed in $4$ operations each), $2N$ products of them, and $N-1$ sums, a total of $9N-5>3(m-1)$ operations; to this we would need to add the cost of computing $U_{m-1}(x,y)$.}
We can do better if $m$ is big enough: In \cite[Figure 1]{JoyeQuisquater}, the authors present a faster divide-and-conquer algorithm for the computation of the Lucas sequence of the first kind\footnote{The presentation of $U_m(x,y)$ given by this algorithm is not canonical over $K[x,y]$; on the contrary, we get a compact presentation which explains its efficiency. For example, $U_7(x,y)$ is given as $x(x^2 - 2y)(x(x^2 - 2y) - xy) - y^3$.}, based on the binary expansion of $m$ and the well-known index formulas of Lucas sequences (e.g. $U_{2m-1}(x,y)=U_m^2(x,y)-yU_{m-1}^2(x,y)$). This algorithm actually works over any commutative unital ring. We include it here for the reader convenience\footnote{There is a typo in the original code of the algorithm: the line $V_h=V_h-2*Q_h$ should read $V_h=V_h*V_h-2*Q_h$ instead, as at the end of line 5 of our code.}:
\begin{algorithm}[H] 
\caption{Divide-and-conquer computation of $U_m(x,y)$ (\cite[Figure 1]{JoyeQuisquater})}\label{divideandconquer}
\hspace{-165pt}\textbf{Inputs} $m=2^s\sum_{i=s}^{\lfloor\log_2 m\rfloor} m_i2^{i-s}$ ($m_s=1$); $x,y\in K$ \\
\hspace{-130pt}\textbf{Output} $U_m(x,y)$
\begin{algorithmic}[1]
\STATE $U_h=1; V_l=2; V_h=x; Q_l=1; Q_h=1$
\FOR{$i$ \n{from} $\lfloor\log_2 m\rfloor$ \n{to} $s+1$ \n{by} -1}
\STATE $Q_l=Q_l\cdot Q_h$
\IF{$m_i=1$}
\STATE $Q_h=Q_l\cdot y$; $U_h=U_h\cdot V_h$; $V_l=V_h\cdot V_l-x\cdot Q_l$; $V_h=V_h\cdot V_h-2\cdot Q_h$
\ELSE
\STATE $Q_h=Q_l$; $U_h=U_h\cdot V_l-Q_l$; $V_h=V_h\cdot V_l-x\cdot Q_l$; $V_l=V_l\cdot V_l-2\cdot Q_l$
\ENDIF
\ENDFOR
\STATE $Q_l=Q_l\cdot Q_h$;\! $Q_h=Q_l\cdot y$;\! $U_h=U_h\cdot V_l-Q_l$;\! $V_l=V_h\cdot V_l-x\cdot Q_l$;\! $Q_l=Q_l\cdot Q_h$
\FOR{$i$ \n{from} $1$ \n{to} $s$}
\STATE $U_h=U_h\cdot V_l$; $V_l=V_l\cdot V_l-2\cdot Q_l$; $Q_l=Q_l\cdot Q_l$
\ENDFOR
\RETURN $U_h$
\end{algorithmic}
\end{algorithm}
In general, with this strategy we do not get the value of $U_{m-1}(x,y)$ as a subproduct of the computation of $U_m(x,y)$, so we have to run the algorithm twice to compute both. This algorithm contains two different loops: the second one for the ending zeros of the binary expansion of $m$, which contains $5$ elementary operations per cycle, and the first for the rest of digits, with 10 elementary operations per cycle; there is also a small number of operations between cycles. Therefore the worst case is produced when either $m$ is an odd number and $m-1$ is only divisible by $2$ once (i.e., when $m$ is congruent to $3\!\!\mod 4$) or when their roles are reversed ($m$ congruent to $2\!\!\mod 4)$. Both cases are equivalent, so we analyze the first. We need $\lfloor\log_2 m\rfloor$ divisions in $\Z$ to produce the binary expansion of $m$, and then we get the expansion of $m-1$ in one operation in $\Z$. In worst case, $U_m(x,y)$ requires $(9)_K+(1)_\Z$ elementary operations (sums, multiplications, and a parity check) for each of the $\lfloor\log_2 m\rfloor$ steps of the first loop, and $3$ ending operations in $K$. Note that since $m$ is not a power of $2$, $\lfloor \log_2(m-1)\rfloor = \lfloor \log_2 m\rfloor$. The computation of $U_{m-1}(x,y)$ requires $(9)_K+(1)_\Z$ operations in each of the $\lfloor\log_2 m\rfloor-1$ steps of the first loop, $8$ intermediate operations, and $5$ operations in $K$ in a single step of the second loop. Thus the total number of operations is \[\C_{\text{divide-and-conquer}}(U_m(x,y),U_{m-1}(x,y))=(18\lfloor\log_2 m\rfloor+7)_K + (3\lfloor\log_2 m\rfloor)_\Z.\]
The divide-and-conquer algorithm is more efficient than the recurrence one, as its complexity is logarithmic with $m$ instead of linear.

\setcounter{equation}{1}
\begin{rem}[\n{Computing $U_m(x,y)$ through a Chebyshev polynomial}]\label{Chebyshevcomplexity}
\mbox{}

For rings with enough square roots and free of $2$-torsion, the computation of $U_m(x,y)$ can be produced via a Chebyshev polynomial when $y$ is a regular element (see Remark \ref{Chebyshev}). From Formula \eqref{(B)} we get
\[U_m(x,y)= (\sqrt y)^{m-1}U_{m-1}(x/(2\sqrt y))\]
where $U_m(x)$ is as before the $m$th Chebyshev polynomial of the second kind and the computation is done in $Q(K)$. Since the known computing strategies for Chebyshev polynomials are analogous to those of Lucas sequences, the most efficient one being a divide-and-conquer algorithm equivalent to Algorithm \ref{divideandconquer} (see \cite{Koepf}), there seems to be no real gain in switching to Chebyshev polynomials. In addition, with this strategy we also need to compute $(\sqrt d)^{m-1}$, which requires the computation of a square root when $m$ is even.
\end{rem}

\subsection{Computation of continuant polynomials}\label{continuantcomplexity}
\mbox{}

We show that their definitions are not very efficient when computing continuant polynomials of types $\alpha,\beta,\pi$, the recurrence equations of Lemma \ref{recurrences} providing much better results. We work only with polynomials of type $\alpha$; polynomials of type $\beta$ are similar, while $\pi(r,k)=\alpha(r,k)+\beta(r,k)$.
\enum
\item\n{Polynomials of type $\alpha$ from their definition.} From Formula \refeq{alpha},
\[\alpha(r,k)=\sum_{m=0}^{\lfloor r/2\rfloor} \sum_{i\in\binom{[r-1]}{m}_2} p_{r,k}(i).\]
We can show that $\left|\binom{[r-1]}{m}_2\right|=\binom{r-1-m}m$, which is true for $m=0$ by definition, by establishing the bijection $f:\binom{[r-1-m]}m\rightarrow\binom{[r-1]}{m}_2$ given by $f(i_1,\ldots,i_m):=(i_1,i_1+1,\ldots,i_m,i_m+1)$ for $m\in\N^*$. Then, since each monomial $p_{r,k}(i)$ with $i\in\binom{[r-1]}{m}_2$ is the product of $r-m$ variables, we get
    \[\C_{\text{definition}}(\alpha(r,k))=\sum_{m=0}^{\lfloor r/2\rfloor}\binom{r-m-1}m(r-m-1) -1\]
    (although some products could be reused to reduce the complexity\footnote{E.g., in the computation of $\alpha(5,7)$, the product $x_1x_2$ found when computing $x_1\cdots x_5$ could be reused in $x_1x_2y_3x_5$, then $x_1x_2x_3$ reused in $x_1x_2x_3y_4$, etc.}), which is quite large:  $\C_{\text{definition}}(\alpha(r,k))\geq (r-2)^2$ already for $r\geq2$.
\item\n{Polynomials of type $\alpha$ from their recurrence equation.} From Lemma \ref{recurrences}(2),
\[\alpha(r+1,k)=x_{r+1}\alpha(r,k)+y_r\alpha(r-1,k)\]
with $\alpha(0,k)=1$. Hence we can compute $\alpha(r,k)$ recursively from $\alpha(0,k),\ldots$, $\alpha(r-1,k)$, getting
\[\C_{\text{recurrence}}(\alpha(r,k))=3(r-1).\]
\eenum

\subsection{Computation of the determinant}\label{detcomplexity}
\mbox{}

We develop four algorithms that compute the determinant $D^{\ov a,\ov d}(n,k)$ of $T_n^k(\ov a,\ov b,\ov c)$, all arising from Theorem \ref{determinant} and its proof; thus we follow the notation and ideas established in them. In particular we denote $d_i:=b_ic_i$, $A_i:=\pma a_i & -d_{i-1}\\ 1 & 0\epma$ ($d_0:=d_k$) for $1\leq i\leq k$ and $A:=A_k\cdots A_1$, and we write $n=mk+r$ by Euclidean division, assuming $m>0$. After that we compare the algorithms in terms of efficiency and study some variants that are important in the analysis of the subsequent algorithms.

\enum
\item\n{Algorithm D1.}\label{algorithmD1}
Using the recurrence of Formula \eqref{(proof0)}
\[D^{\ov a,\ov d}(n,k)=a_rD^{\ov a,\ov d}(n-1,k)-d_{r-1}D^{\ov a,\ov d}(n-2,k)\]
with $D^{\ov a,\ov d}(0,k)=1$, $D^{\ov a,\ov d}(1,k)=a_1$ given in the proof of Theorem \ref{determinant}, we can compute $D^{\ov a,\ov d}(n,k)$ in at most $3(n-1)+k$ operations, with $k$ operations coming from the multiplications needed to get $d_i=b_ic_i$ for $1\leq i\leq k$. This algorithm, which we call D1, mostly ignores the periodicity of the matrix. D1 has complexity
\[\C_{\text{D1}}(D^{\ov a,\ov d}(n,k))=3n+k-3.\]

\item\n{Algorithm D2.}\label{algorithmD2}
Alternatively we can compute, as in Formula \eqref{(proof2)} from the proof of Theorem \ref{determinant},
\[\pma D^{\ov a,\ov d}(n,k) \\ D^{\ov a,\ov d}(n-1,k)\epma = A_r\cdots A_1 \cdot A^m\pma 1 \\ 0\epma.\]
The $d_i$ require $k$ operations. Due to the special nature of the second row of the $A_i$ we can compute $A$ in $6(k-1)$ operations, getting $A_r\cdots A_1$ as a subproduct.
We could get $A^m$ in $12(m-1)$ operations by naive matrix multiplication, but it is more efficient to use iterative squaring: writing $m=\sum_{i=0}^{L} m_i2^i$ in base $2$ with $L:=\lfloor\log_2 m\rfloor$ (which needs $L$ divisions in $\Z$) we get $A^m$ as $A_L$ in the recurrence equation $A_{i+1}:=(A_i)^2A^{m_{L-i+1}}$ with $A_0:=A$, a computation that requires $24L$ operations in worst case (when $m=2^p-1$ for some $p$). Lastly, we get $(A_r\cdots A_1)\cdot(A^m)$ in $12$ operations. To this process we call algorithm D2. Therefore
\[\C_{\text{D2}}(D^{\ov a,\ov d}(n,k))=(24\lfloor\log_2 m\rfloor+7k+6)_K + (\lfloor\log_2 m\rfloor)_\Z.\]
D2 does not take advantage of Lemma \ref{powers}(2) in the determination of $A^m$.

\item\n{Algorithm D3.}\label{algorithmD3}
Formula \eqref{det1} for the determinant is
\[D^{\ov a,\ov d}(n,k)=U(m)\alpha^{\ov a,\sminus\ov d}(k+r,k) -dU(m-1)\alpha^{\ov a,\sminus\ov d}(r,k)\]
with $U(i)=U_i(\pi^{\ov a,\sminus\ov d}(k,k),d)$. We study how many operations are enough to compute $D^{\ov a,\ov d}(n,k)$ from this formula.
\ite
\item The parameters $d_i$ require $k$ operations.
\item The best way we know to compute $\pi^{\ov a,\sminus\ov d}(k,k)$ is as $\alpha^{\ov a,\sminus\ov d}(k,k)+\beta^{\ov a,\sminus\ov d}(k,k)$, producing the evaluations with the recurrence formulas of Lemma \ref{recurrences}(2,3),
\[\alpha(i+1,k)=x_{i+1}\alpha(i,k)+y_i\alpha(i-1,k), \beta(i+1,k)=x_i\beta(i,k)+y_{i-1}\beta(i-1,k)\]
with $\alpha(0,k)=1$, $\beta(0,k)=0=\beta(1,k)$, $\beta(2,k)=y_k$ (see Section \ref{continuantcomplexity}). This computation needs $3(k-1)+3(k-2)+1$ operations, which include $\alpha^{\ov a,\sminus\ov d}(r,k)$, $\alpha^{\ov a,\sminus\ov d}(k-1,k)$ and $\beta^{\ov a,\sminus\ov d}(r+1,k)$ as subproducts.
\item The element $d$ could be computed as $d=d_1\cdots d_k$ in $k-1$ operations. But $d$ is the determinant of matrix $A$, so
\eqnum{d=\alpha^{\ov a,\sminus\ov d}(k,k)\beta^{\ov a,\sminus\ov d}(k,k)-\alpha^{\ov a,\sminus\ov d}(k-1,k)\beta^{\ov a,\sminus\ov d}(k+1,k);\label{(D)}}
therefore we can compute $d$ with just $6$ operations, $3$ of them to compute $\beta^{\ov a,\sminus\ov d}(k+1,k)$ from $\beta^{\ov a,\sminus\ov d}(k,k)$ and $\beta^{\ov a,\sminus\ov d}(k-1,k)$ with Lemma \ref{recurrences}(3). This option is more efficient when $k>7$, the differences when $k\leq7$ being negligible, so we favour Formula \eqref{(D)} to compute $d$.\\
\n{Remark:} The previous two items are essentially equivalent to computing $A$ and then finding $\pi^{\ov a,\sminus\ov d}(k,k)$ and $d$ respectively as its trace and determinant, getting also $\alpha^{\ov a,\sminus\ov d}(r,k)$, etc. as subproducts, since they appear as entries of the matrices $A_r\cdots A_1$, etc.
\item We find $\alpha^{\ov a,\sminus\ov d}(k+r,k)$ as
\[\alpha^{\ov a,\sminus\ov d}(k+r,k)=\alpha^{\ov a,\sminus\ov d}(k,k)\alpha^{\ov a,\sminus\ov d}(r,k)+\alpha^{\ov a,\sminus\ov d}(k-1,k)\beta^{\ov a,\sminus\ov d}(r+1,k)\]
by Corollary \ref{secondperiod}. Since the evaluations of the continuant polynomials of the first period are already computed, this adds 3 operations.
\item We compute $U(m)$ and $U(m-1)$ from $\pi^{\ov a,\sminus\ov d}(k,k)$ and $d$ by the divide-and-conquer algorithm (\ref{divideandconquer}) in $(18\lfloor\log_2 m\rfloor+7)_K + (3\lfloor\log_2 m\rfloor)_\Z$ operations.
\item Formula \ref{det1} for $D^{\ov a,\ov d}(n,k)$ adds $4$ more operations.
\eite
To the algorithm making calculations as described above we call D3. The grand total number of operations for D3 is
\[\C_{\text{D3}}(D^{\ov a,\ov d}(mk+r,k))=(18\lfloor\log_2m\rfloor+7k+12)_K+(3\lfloor\log_2m\rfloor)_\Z.\]

\item\n{Algorithm D4.}\label{algorithmD4} Formula \eqref{det2} for the determinant is
\[D^{\ov a,\ov d}(n,k)=U(m+1)\alpha^{\ov a,\sminus\ov d}(r,k)+U(m)d_kd_1\cdots d_r\alpha_{r+1}^{\ov a,\sminus\ov d}(k-r-2,k).\]
Algorithm D4 is based on this formula, with subalgorithms similar to those of D3. It requires the following computations:
\ite
\item The $d_i$ parameters require $k$ operations.
\item $\pi^{\ov a,\sminus\ov d}(k,k)$ is computed as before through the recurrence relations of polynomials of types $\alpha$ and $\beta$ (Lemma \ref{recurrences}(2,3)) in $6k-8$ operations, with $\alpha^{\ov a,\sminus\ov d}(r+1,k)$ as a subproduct.
\item $\alpha_{r+1}^{\ov a,\sminus\ov d}(k-r-2,k)$ is computed in $3(k-r-3)$ operations also through Lemma \ref{recurrences}(2).
\item We could compute $d=d_1\cdots d_k$ in $k-1$ operations, carried out so as to get the value of $d':=d_kd_1\ldots d_r$ as a subproduct. But in general it is better to get $d$ through Formula \ref{(D)} in $6$ operations, and $d'$ in $3$ operations as
    \[d'=\det(A_{r+1}\cdots A_1)=\alpha^{\ov a,\sminus\ov d}(r+1,k)\beta^{\ov a,\sminus\ov d}(r+1,k)-\alpha^{\ov a,\sminus\ov d}(r,k)\beta^{\ov a,\sminus\ov d}(r+2,k).\]
\item $U(m+1)$ and $U(m)$ are computed as before (through Algorithm \ref{divideandconquer}), requiring (in worst case) $(18\lfloor\log_2 m\rfloor+7)_K + (3\lfloor\log_2 m\rfloor)_\Z$ operations.
\item Formula \eqref{det2} for $D(n,k)$ adds $4$ more operations.
\eite
Hence we get
\[\C_{\text{D4}}(D^{\ov a,\ov d}(mk+r,k))=(18\lfloor\log_2m\rfloor+7k+12+3(k-r-3))_K+(\lfloor\log_2m\rfloor)_\Z.\]
\eenum

\smallskip

Let us now compare the algorithms by taking computation costs in $K$ and $\Z$ as being equivalent (see Section \ref{introcomplexity}). D1 is less efficient in general, as it is linear in $m$ while D2-D4 are logarithmic. Nevertheless, the simplest D1 is the most efficient and thus potentially interesting when $m=1$, $k$ is big enough to have an impact on computing time, and $r$ is small with respect to $k$, with the savings, with respect to algorithm D3, being approximately of $3(k-r)$ computations.
As seen from the comparison of complexities, algorithm D2 is slightly less efficient than D3, which is also more efficient than D4: despite Formula \eqref{det2} being apparently simpler than Formula \eqref{det1} (once we apply Corollary \ref{secondperiod} to compute $\alpha^{\ov a,\sminus\ov d}(k+r,k)$), since $r<k$, algorithm D4 in fact requires more computations than algorithm D3 in general, more so the greater the difference $k-r$ is (D4 is only marginally more efficient when $r\in\{k-2,k-1\}$, up to $6$ operations). This loss of efficiency of algorithm D4 happens because the favourable dependence of Formula \eqref{det2} on $r$ is overridden by the computation of $\pi^{\ov a,\sminus\ov d}(k,k)$, whose best implementation as far as we know necessitates all polynomials of types $\alpha$ and $\beta$, and also because we need to compute $\alpha_{r+1}^{\ov a,\sminus\ov d}(k-r-2,k)$ apart from the rest of calculations. In conclusion, algorithm D3 is the most efficient one in general.

\begin{rem}[\n{Comparison with da Fonseca and Petronilho's formula}]\label{FonsecaPetronilhocomplexity}
\mbox{}

By Remark \ref{FonsecaPetronilho}, Da Fonseca and Petronilho's formula for the determinant of an irreducible matrix over the complex numbers found in \cite[Section 4]{FonsecaPetronilho2005} is similar to Formula \eqref{det2}, on which algorithm D4 (\ref{algorithmD4}) is based, with the evaluations of polynomials of type $\alpha$ and $\pi$ written as determinants, and the generalized Fibonacci polynomials written as Chebyshev polynomials. By the paragraph above this remark and by Remark \ref{Chebyshevcomplexity}, an algorithm based on this formula can only be more efficient than algorithm D3 (\ref{algorithmD3}) if the determinants involved can be computed more efficiently than we compute the continuant polynomials with the recurrence formulas of Lemma \ref{recurrences}(2,3); if that was the case, then we could alter algorithm D3 to compute the continuant polynomials through said determinants, and find in this way a more efficient algorithm again, by using the generalized Fibonacci polynomials instead of Chebyshev polynomials and Formula \eqref{det1} instead of Formula \eqref{det2}.
\end{rem}

\begin{rem}[\n{Complexity with eigenvalues}]\label{eigenvaluescomplexity}
\mbox{}

If $K$ is a field, the computation of $U(m), U(m-1)$ needed in Formula \eqref{det1} can be produced via the eigenvalues of $A$ through the formulas given in Remarks \ref{eigenvalues}. We study the complexity when we modify algorithm D3 (\ref{algorithmD3}) accordingly:
\enum
\setcounter{enumi}{4}
\item \n{Algorithm D3-eigenvalues.}\label{algorithmD4eigenvalues} Let $K$ be a field, $\ov K$ be an algebraic closure, and $\lambda_1,\lambda_2\in\ov K$ be the eigenvalues of $A$. In worst case we have $\lambda_1\neq\lambda_2$, and so by Formula \eqref{(A1)}
\[U(m)=\frac{\lambda_2^m-\lambda_1^m}{\lambda_2-\lambda_1}.\]
We need to compute $U(m)$ and $U(m-1)$. As in the case of matrices, exponentiation is more efficient through iterative squaring: writing $m=\sum_{i=0}^{L} m_i2^i$ in base $2$ with $L:=\lfloor\log_2 m\rfloor$ (which needs $L$ divisions in $\Z$), for any $\lambda\in\ov K$ we get $\lambda^m$ as $a_L$ in the recurrence equation $a_{i+1}:=(a_i)^2\lambda^{m_{L-i+1}}$ with $a_0:=\lambda$. To compute $U(m)$ and $U(m-1)$, we first compute $\lambda_1^{m-1}, \lambda_2^{m-1}$, then get $\lambda_1^m,\lambda_2^m$ via multiplication by $\lambda_i$ in 2 more operations. Therefore the worst case occurs when $m$ is a power of $2$, and then the computation of $\lambda^m$ requires $(2L)_{\ov K}+(L)_\Z$ operations. Thus we can compute $U(m)$ and $U(m-1)$ in $(4L+7)_{\ov K}+(L)_\Z$ operations. To these calculations we need to add the cost of computing the eigenvalues of $A$: If $\ch(K)\neq2$ we use the quadratic formula applied to the characteristic polynomial and some algorithm $F$ to compute square roots, the cost being $7+\C_F(\sqrt{e})$ operations in $\ov K$, where $e:=\pi^{\ov a,\ov d}(k,k)^2-4d$; if $\ch(K)=2$ we use the $R$ operation (see Remarks \ref{eigenvalues}(1)) and some algorithm $G$ to compute it, the cost being, in the worst case in which $\pi^{\ov a,\ov d}(k,k)\neq0$, of $4+\C_G(R(f))$ operations in $\ov K$ where $f:=d/\pi^{\ov a,\ov d}(k,k)^2$. Therefore $\C_{\text{D3-eigenvalues}}(D^{\ov a,\ov d}(mk+r,k))$ is
\[\left\{\begin{array}{lc}
(4\lfloor\log_2m\rfloor+7k+19+\C_F(\sqrt{e}))_{\ov K}+(\lfloor\log_2m\rfloor)_\Z, &\text{ if }\ch(K)\neq2\\
(4\lfloor\log_2m\rfloor+7k+16+\C_G(R(f))_{\ov K}+(\lfloor\log_2m\rfloor)_\Z, &\text{ if }\ch(K)=2
\end{array}\right..\]
\eenum
\end{rem}

\begin{rem}[\n{Computing several determinants}]\label{severaldets}
\mbox{}

\enum
\setcounter{enumi}{5}
\item \n{Algorithm D3-twice.}\label{algorithmD3twice} In algorithm D3 (\ref{algorithmD3}), to get $D^{\ov a,\ov d}(mk+r,k)$ we need to compute
\[U(m-1),U(m), \alpha^{\ov a,\ov d}(0,k),\ldots,\alpha^{\ov a,\ov d}(k,k), \beta^{\ov a,\ov d}(1,k),\ldots,\beta^{\ov a,\ov d}(k+1,k).\]
Observe that the only values depending on $n$ are the first two; so, after one call to algorithm D3, to compute $D^{\ov a,\ov d}(m'k+r',k)$ (evaluated in the same vectors $\ov a,\ov d$) we only need to compute $U(m'),U(m'-1)$ through algorithm \ref{divideandconquer} (see Remark \ref{eigenvaluescomplexity} if $K$ is a field) and apply Formula \eqref{det1} to get the result. To this procedure we call algorithm D3-twice. Hence
\eq{
&\C_{\text{D3-twice}}(D^{\ov a,\ov d}(mk+r,k),D^{\ov a,\ov d}(m'k+r',k))=\\
&=(18\lfloor\log_2m\rfloor+18\lfloor\log_2m'\rfloor+7k+26)_K+(3\lfloor\log_2m\rfloor+3\lfloor\log_2m'\rfloor)_\Z.
}
\eenum
It is straightforward to generalize the idea to compute the $p$ determinants \\
$D^{\ov a,\ov d}(n_1,k),\ldots,D^{\ov a,\ov d}(n_p,k)$. Moreover, if $n_i$ and $n_j$ have the same quotient $m$ when divided by $k$, then $U(m),U(m-1)$ need to be computed only once; in particular, once we compute $D^{\ov a,\ov d}(mk+r,k)$, any of \[D^{\ov a,\ov d}(mk,k),\ldots, D^{\ov a,\ov d}(mk+k-1,k)\]
can be found with just $6$ more elementary operations.
\enum
\setcounter{enumi}{6}
\item \n{Computing all determinants.}\label{algorithmD1all} On the other hand, the best way to compute all determinants from $D^{\ov a,\ov d}(0,k)$ to $D^{\ov a,\ov d}(n,k)$ (as we need later) is through algorithm D1 (\ref{algorithmD1}), with complexity
\[\C_{\text{D1}}(D^{\ov a,\ov d}(0,k),\ldots,D^{\ov a,\ov d}(n,k))=3n+k-3.\]
\eenum
\end{rem}

\begin{rem}[\n{Computing shifted determinants}]\label{shifteddets}
\mbox{}

Since $D_s^{\ov a,\ov d}(n,k)=D^{\sigma_s(\ov a),\sigma_s(\ov d)}(n,k)$, with algorithm D3 (\ref{algorithmD3}) we can compute a shifted determinant in $(18\lfloor\log_2m\rfloor+7k+12)_K+(3\lfloor\log_2m\rfloor)_\Z$ operations (the circular permutation has no cost, as it just implies a different evaluation).

To compute all shifted determinants $D_s^{\ov a,\ov d}(n-s,k)$ for $0\leq s\leq n$ efficiently (which we need later) we can use the following variant of algorithm D1 (\ref{algorithmD1}):
\enum
\setcounter{enumi}{7}
\item \n{Algorithm D1-shifted.}\label{algorithmD1shifted} Using the recurrence formula for shifted universal determinants (\ref{detrecurrences}(2)) we get
\[D^{\ov a,\ov d}_s(n-s,k)=a_{s+1}D^{\ov a,\ov d}_{s+1}(n-s-1,k)-d_{s+1}D^{\ov a,\ov d}_{s+2}(n-s-2,k)\]
with $D_n^{\ov a,\ov d}(0,k)=1$, $D_{n-1}^{\ov a,\ov d}(1,k)=a_r$ as initial conditions and $s$ ranging from $n-2$ to $0$; so with $k$ operations to get $d_i=b_ic_i$ for $1\leq i\leq k$, we have
\[\C_{\text{D1-shifted}}(D_n^{\ov a,\ov d}(0,k),\ldots,D_0^{\ov a,\ov d}(n,k))=3n+k-3.\]
\eenum
\end{rem}

\begin{rem}[\n{\small Computing the determinant of a general tridiagonal matrix}]\label{complexitygeneraldeterminant}
\mbox{}

Algorithms D1-D4 assume $n>k$. If $n\leq k$ then there is no periodicity in the matrix to be exploited by the algorithms.  In this case we only have to compute $d_1,\ldots,d_n$ and then $D^{\ov a,\ov d}(n,k)$ can be computed by recurrence (Lemma \ref{detrecurrences}(1)) with a total complexity of $4n-3$ operations.
\end{rem}

\subsection{Computation of spectral properties}

\enum
\item\n{Characteristic polynomial.} By Corollary \ref{charpolcor} we can find $p_{T_n^k(\ov a,\ov b,\ov c)}(x)$ with Formula \eqref{det1}, so using algorithm D3 (\ref{algorithmD3}) we get
\[\C_{\text{D3}}(p_{T_{mk+r}^k(\ov a,\ov b,\ov c)}(x))=(18\lfloor\log_2m\rfloor+7k+12)_{K[x]}+(3\lfloor\log_2m\rfloor)_\Z.\]
\n{Remark:} The expression of the characteristic polynomial given by algorithm D3 is not a full expansion in the canonical basis of $K[x]$, due to the shape of Formula \eqref{det1} and to the compact presentation of $U_m(x,y)$ given by the divide-and-conquer algorithm (\ref{divideandconquer}).
\item\n{Eigenvectors.}\label{algorithmEIG} By Theorem \ref{eigenvectors}a), given an eigenvalue $\lambda$ of $T_n^k(\ov a,\ov b,\ov c)$ and an element $z\in\Ann(p_{T_n^k(\ov a,\ov b,\ov c)}(\lambda))$, the vector $v^z(\lambda):=(v^z_1(\lambda),\ldots,v^z_n(\lambda))$, if nonzero, is an eigenvector associated to $\lambda$, with
\eqnum{v^z_i(\lambda):=z\prod_{j=i}^{n-1}b_j \cdot D^{\ov\lambda-\ov a,\ov d}(i-1,k). \label{eigv}}
To compute $v^z(\lambda)$ we need to compute all determinants $D^{\ov\lambda-\ov a,\ov d}(0,k),\ldots,$ $D^{\ov\lambda-\ov a,\ov d}(n-1,k)$, which can be done with algorithm D1 (\ref{algorithmD1all}) in $3n+k-6$ operations. In addition, since we need all the products $\prod_{j=i}^{n-1}b_j$ for $1\leq i\leq n$, the best way to compute them is iteratively (as $1,b_{n-1},b_{n-1}b_{n-2},\ldots$) in $n-2$ operations. Calling this procedure algorithm EIG, including all multiplications coming from \ref{eigv}, we get
\[\C_{\text{EIG}}(v^z(\lambda))=6n+k-10.\]
Note that the only advantage that this procedure takes on the matrix being $k$-Toeplitz is the use of $d_1,\ldots,d_k$ instead of $d_1,\ldots,d_n$.

Analogously, the (potential) eigenvector $w^z(\lambda):=(w^z_1(\lambda),\ldots,w^z_n(\lambda))$ with
\[w^z_i(\lambda):=z\prod_{j=1}^{i-1}c_j \cdot D_i^{\ov\lambda-\ov a,\ov d}(n-i,k)\]
(Theorem \ref{eigenvectors}b)) can be found with the same cost by using algorithm D1-shifted (\ref{algorithmD1shifted}) for computing the shifted determinants $D_1^{\ov\lambda-\ov a,\ov d}(n-1,k),\ldots,$ $D_n^{\ov\lambda-\ov a,\ov d}(0,k)$.
\eenum

\begin{rem}[\n{\small Computing the spectral properties of a tridiagonal matrix}]\label{complexitygeneralspectralproperties}
\mbox{}

If $n\leq k$ then there is no periodicity in the matrix that can be exploited by the algorithms. We show the cost of computing the spectral properties of a general tridiagonal matrix through our previous results.
\enum
\setcounter{enumi}{2}
\item\n{Characteristic polynomial.} In this case, the characteristic polynomial is computed through Corollary \ref{charpolcor} as the determinant of a general tridiagonal matrix, so it can be found by recurrence in $4n-3$ operations (see Remark \ref{complexitygeneraldeterminant}).
\item\n{Eigenvectors.} The only part of algorithm EIG in which the periodicity is exploited is the computation of all determinants $D^{\ov\lambda-\ov a,\ov d}(0,k),\ldots,$ $D^{\ov\lambda-\ov a,\ov d}(n-1,k)$, which in this case can be done in $4n-6$ operations by recurrence (Lemma \ref{detrecurrences}(1)). Therefore the modified algorithm, GENEIG, has complexity
    \[\C_{\text{GENEIG}}(v^z(\lambda))=7n-10.\]
\eenum
\end{rem}

\subsection{Computation of the inverse}
\mbox{}

If $D^{\ov a,\ov d}(n,k)$ is a unit of $K$ then $T^k_n(\ov a,\ov b,\ov c)$ is invertible and by Theorem \ref{inverse} the $(i,j)$ entry $a_{ij}=a_{ij}^{(n,k)}$ of its inverse is given by
\eqnum{a_{ij}^{(n,k)}=(-1)^{i+j}\prod_{p=i}^{j-1}b_p\prod_{p=j}^{i-1}c_p\frac{\displaystyle D^{\ov a,\ov d}(\min(i,j)-1,k)D_{\max(i,j)}^{\ov a,\ov d}(n-\max(i,j),k)}{\displaystyle D^{\ov a,\ov d}(n,k)}.\label{invformula}}
We study the complexity of computing, respectively, one entry and all entries of the inverse through Formula \eqref{invformula}. Write $n=mk+r$ by Euclidean division.
\enum
\item \n{An entry of the inverse.}\label{algorithmENTRY} We find the cost of computing one entry of the inverse of $T^k_n(\ov a,\ov b,\ov c)$ with Formula \eqref{invformula}. Suppose without loss of generality that $i\leq j$ (so $\min(i,j)=i,\max(i,j)=j$ and $\prod_{p=j}^{i-1}c_p=1$), and call $m_x,r_x\in\N$, respectively, to the quotient and remainder resulting when dividing $x\in\N$ by $k$ (note that $m_n=m$).
\ite
\item We compute $D^{\ov a,\ov d}(n,k)$ and $D^{\ov a,\ov d}(i-1,k)$ with algorithm D3-twice (\ref{algorithmD3twice}).
\item We compute $D_j^{\ov a,\ov d}(n-j,k)$ with another call to algorithm D3 (see Remark \ref{shifteddets}), taking into account that $d_1,\ldots,d_k$ and $d$ have already been computed.
\item We write $\prod_{p=i}^{j-1}b_p=(b_1\cdots b_k)^{m_{j-i}}b$, where $b$ includes the last $r_{j-i}$ factors of the product. We compute $b_1\cdots b_k$ in $k-1$ operations, in the right order ($b_i,b_ib_{i+1},\ldots$) so as to get $b$ as a subproduct; then we compute $B:=(b_1\cdots b_k)^{m_{j-i}}$ by iterative squaring (see \ref{algorithmD4eigenvalues}) in $(2\lfloor\log_2 m_{j-i}\rfloor)_K+(\lfloor\log_2 m_{j-i}\rfloor)_\Z$ operations; lastly we compute $B\cdot b$ in one operation.
\item Formula \eqref{invformula} adds other $4$ operations (when $i+j$ is odd).
\eite
This procedure we call algorithm ENTRY. Calling $p:=\min(i,j), q:=\max(i,j)$, the number of elementary of elementary operations needed to compute $a_{ij}^{(n,k)}$ through ENTRY is $(18(\lfloor\log_2m_n\rfloor+\lfloor\log_2m_{p-1}\rfloor+\lfloor\log_2m_{n-q}\rfloor)+2\lfloor\log_2 m_{q-p}\rfloor+14k+32)_K+
3(\lfloor\log_2m_n\rfloor+\lfloor\log_2m_{p-1}\rfloor+\lfloor\log_2m_{n-q}\rfloor+\lfloor\log_2 m_{q-p}\rfloor)_\Z.$
Therefore, in worst case ($p=n, q=1$ and $m_{n-1}=m_n=m$),
\[\C_{\text{ENTRY}}(a_{ij}^{(mk+r,k)}) = (56\lfloor\log_2m\rfloor+14k+32)_K+(12\lfloor\log_2m\rfloor)_\Z.\]
\item \n{All entries of the inverse.}\label{algorithmINV} We find now the cost of computing all entries of the inverse of $T^k_n(\ov a,\ov b,\ov c)$ through Formula \eqref{invformula}.
\ite
\item We compute $D^{\ov a,\ov d}(0,k),\ldots,D^{\ov a,\ov d}(n,k)$ with algorithm D1 (\ref{algorithmD1all}).
\item We compute $D_n^{\ov a,\ov d}(0,k),\ldots,D_1^{\ov a,\ov d}(n-1,k)$ with algorithm D1-shifted (\ref{algorithmD1shifted}), taking into account that $d_1,\ldots,d_k$ have already been computed.
\item To compute all products $\prod_{p=i}^{j-1}b_p$ for $1\leq i,j\leq n$, we proceed as follows: In one batch we compute $b_1, b_1b_2, \ldots, b_1\cdots b_k=:b$ and then cyclically $b_2, b_2b_3,\ldots, b_2\cdots b_k$, \, $b_3,b_3b_4, \ldots, b_3\cdots b_kb_1$, etc. up to \\
    $b_{k-1}, b_{k-1}b_k, \ldots, b_{k-1}b_kb_1b_2\cdots b_{k-3}$, in $(k-1)^2$ operations. In a second batch we compute $b, b^2, \ldots, b^m$ in $m-1$ operations. Then we multiply all elements of the first batch with all elements of the second batch in $(m-1)(k-1)^2$ operations\footnote{For example, if $k=5$ then the product for the element $a_{4,22}$ is $b_4b_5b_1\cdots b_5b_1 = b_4b_5b^3b_1 = (b_4b_5b_1)b^3$.}. Analogously we compute all products $\prod_{p=i}^{j-1}c_p$ for $1\leq i,j\leq n$.
\item Since the two determinants appearing in the numerator of Formula \eqref{invformula} depend only on $\min(i,j)$ and $\max(i,j)$, their product is the same for entries $a_{ij}$ and $a_{ji}$. Moreover, if $i\leq j$ then $\prod_{p=j}^{i-1}c_p=1$, while if $j\leq i$ then $\prod_{p=i}^{j-1}b_p=1$. Therefore we can compute $a_{11},\ldots,a_{nn}$ from Formula \eqref{invformula} with another $5(n+1)n/2$ products.
\eite
To this procedure we call algorithm INV. Then
\[\C_{\text{INV}}(a_{11}^{(mk+r,k)},\ldots,a_{nn}^{(mk+r,k)})=5n^2/2+2k^2m+17n/2-4mk+4m+k-8.\]
Note that $mk\approx n$, so
\[\C_{\text{INV}}(a_{11}^{(mk+r,k)},\ldots,a_{nn}^{(mk+r,k)})\approx 5n^2/2+(2k+9/2)n+4m+k-8.\]
\eenum

\begin{rem}[\n{Complexity of Lewis' formula}]\label{Lewiscomplexity}
\mbox{}

In \cite[Theorem 1]{Lewis1982}, Lewis provides the following formula for an entry of the inverse of an irreducible tridiagonal matrix $T=(t_{ij})$ over a field (we write it only for the case $i<j$):
\eqnum{a_{ij}= \left(\prod_{p=i}^{j-1}e_p\right)\nu_i\mu_jD,\label{Lewisformula}}
where $e_p:=t_{p,p+1}/t_{p+1,p}$, $\nu_i, \mu_j$ are respectively computed through the recurrence equations
\[\nu_i:=-h_{i-1}\nu_{i-1}-\frac1{g_{i-1}}\nu_{i-2}, \,\, \mu_i:=-f_{i+1}\mu_{i+1}-g_{i+1}\mu_{i+2},\]
with $v_1:=1, v_2:=-h_1$, $\mu_n:=1, \mu_{n-1}:=-f_n$, $f_i:=t_{ii}/t_{i,i+1}, g_i:=t_{i,i+1}/t_{i,i-1}$, $h_i:=t_{ii}/t_{i,i+1}$, and $D:=(t_{11}\mu_1+t_{12}\mu_2)^{-1}$.

If we want to compute the element $a_{ij}$ with an algorithm inspired by Lewis' formula, then we need to compute the recurrence of the $\nu$'s up to the $i$th term ($3(i-2)$ operations) and the recurrence of the $\mu$'s down to $\mu_1$ ($3n$ operations), since $D$ depends on $\mu_1,\mu_2$ (we need other $4$ operations to compute $D$). If the tridiagonal matrix is $k$-Toeplitz, for the recurrences we need to compute the corresponding divisions for $f_1,\ldots,f_k$, $g_1,\ldots,g_k$ and $h_1,\ldots,h_{i-1}$ ($2k+\min(i-1,k)$ operations). If $i\neq j-1$ then we also need to compute the product
$\prod_{t=i}^{j-1}e_t$, which needs $j-i$ divisions to get the $e$'s and can be produced in $(k+2\lfloor\log_2 m_{j-i}\rfloor)_K+(\lfloor\log_2 m_{j-i}\rfloor)_{\Z}$ operations, as in algorithm ENTRY (\ref{algorithmENTRY}), if the tridiagonal matrix is $k$-Toeplitz. So, the number of elementary operations needed to compute $a_{ij}^{(n,k)}$ through this procedure, which we call Lewis, is $(3n+3k+\min(k,i-1)+2i+j+2\lfloor\log_2 m_{j-i}\rfloor-2)_K+(\lfloor\log_2 m_{j-i}\rfloor)_{\Z}$ when $i\neq j-1$; in worst case, which is $i=n-1, j=n$, we get
\[\C_{\text{Lewis}}(a_{ij}^{(n,k)})=6n+3k-3,\]
which is less efficient than our algorithm ENTRY, which is linear only in $k$ (not in $n$) and logarithmic with $m$.
\end{rem}

\begin{rem}[\n{Computing the inverse of a general tridiagonal matrix}]\label{complexitygeneralinverse}
\mbox{}

If $n\leq k$ then there is no periodicity in the tridiagonal matrix that can be exploited by the algorithms.  We show the cost of algorithms based on Formula \eqref{invformula} for computing one element and all elements of a general tridiagonal matrix:
\enum
\setcounter{enumi}{2}
\item\n{Algorithm GENENTRY.} When computing the element $a_{ij}^{(n)}$ (suppose $i\leq j$), we can compute the determinants in $7n-3j-6$ operations by recurrence (as in Remark \ref{complexitygeneraldeterminant}), by computing $d_1,\ldots,d_n$ once and getting the determinant of size $i-1$ as a subproduct of the computation of the determinant of size $n$. In addition, the product $\prod_{p=i}^{j-1} b_p$ now requires $j-i$ operations. Hence the number of elementary operations needed to compute $a_{ij}$ through GENENTRY is $7n-2j-i-2$. Therefore, in worst case $(i=1=j)$ we get
    \[\C_{\text{GENENTRY}}(a_{ij}^{(n)})=7n-5.\]
    We briefly compare this result with an algorithm based on Lewis' formula: Looking at Remark \ref{Lewiscomplexity} we see that, if the tridiagonal matrix is not $k$-Toeplitz, the only change required in Lewis algorithm in worst case is the need of computing $f_1,\ldots,f_n$, $g_1,\ldots,g_n$ and $h_1,\ldots,h_{n-2}$, which requires $3n-2$ operations, giving in the end
    \[\C_{\text{Lewis}}(a_{ij}^{(n)})=9n-5.\]
    This shows that algorithm GENENTRY is more efficient than Lewis' for general tridiagonal matrices (moreover, it works for any matrix over any commutative ring).
\item\n{Algorithm GENINV.} When computing all elements of the inverse, we compute $d_1,\ldots,d_n$ once ($n$ operations) and then by recurrence (Lemma \ref{detrecurrences}) we can compute the sequences of all determinants $D^{\ov a,\ov d}(n,k),\ldots,D^{\ov a,\ov d}(n,k)$ and all shifted determinants $D_n^{\ov a,\ov d}(0,k),\ldots,D_1^{\ov a,\ov d}(n-1,k)$ in $3n-3$ operations each. To get all products $\prod_{p=i}^{j-1}b_p$ for all $1\leq i\leq j\leq n$, we compute the sequences $b_1,b_1b_2,\ldots,b_1\cdots b_{n-1}$, then $b_2,b_2b_3,\ldots,b_2\cdots b_{n-1}$, etc., up to $b_{n-2}b_{n-1}$ for a total of $\frac12(n-1)(n-2)$ operations, and analogously we get all products $\prod_{p=j}^{i-1}c_p$ for all $1\leq j\leq i\leq n$. As in the last step of algorithm INV (\ref{algorithmINV}), we can compute $a_{11},\ldots,a_{nn}$ from Formula \ref{invformula} with another $\frac52(n+1)n$ products. Therefore
    \[\C_{\text{GENINV}}(a_{11}^{(n)},\ldots,a_{nn}^{(n)})=\frac72n^2+\frac{13}2n-3.\]
\eenum
\end{rem}

\section{Examples}\label{sectionexamples}
\numberwithin{equation}{section}

Informally speaking, the theorems and algorithms we have developed in the previous sections produce the universal tridiagonal $k$-Toeplitz example, which is free in $K,n,k,\ov a,\ov b,\ov c$, and can then be evaluated to any specific example over any commutative unital ring. Now we construct two examples which are more specific: one free in $K,\ov a,\ov b,\ov c$ but with $n,k$ fixed, and a completely specific one over $\Z/60\Z$. In these examples, in the computations that we make of the determinant, spectral properties and an entry of the inverse, we follow the algorithms of Section \ref{complexitysection} as close as the writing allows without undermining the exposition. For ease of presentation we choose $k=3$ and $n=19$, although the gain in efficiency of some subalgorithms is apparent only for greater values of $k,n$.

\begin{exam}[\n{Universal example}]\label{universalexample}
Let $R$ be a commutative unital ring and consider $K:=R[a_1,a_2,a_3,b_1,b_2,b_3,c_1,c_2,c_3]$, $k:=3, n:=19$ and $T^3_{19}(\ov a,\ov b,\ov c)$ over $K$ with $\ov a=(a_1,a_2,a_3), \ov b=(b_1,b_2,b_3), \ov c=(c_1,c_2,c_3)$. Put $d_i:=b_ic_i$ for $1\leq i\leq 3$ and $\ov d:=(d_1,d_2,d_3)$. We have $n=mk+r$ with $m=6$, $r=1$.
To compute the determinant we carry out the following computations from algorithm D3 (\ref{algorithmD3}):
First we compute the continuant polynomials through their recurrence relations:
\eq{
&\alpha^{\ov a,\sminus\ov d}(0,3)=1, \, \alpha^{\ov a,\sminus\ov d}(1,3)=a_1, \, \alpha^{\ov a,\sminus\ov d}(2,3)=a_2a_1-d_1,\\
&\alpha^{\ov a,\sminus\ov d}(3,3)=a_3(a_2a_1-d_1)-d_2a_1,\\
&\beta^{\ov a,\sminus\ov d}(1,3)=0, \, \beta^{\ov a,\sminus\ov d}(2,3)=-d_3, \, \beta^{\ov a,\sminus\ov d}(3,3)=a_2(-d_3),\\
&\beta^{\ov a,\sminus\ov d}(4,3)=a_3a_2(-d_3)+d_2d_3,\\
&\pi^{\ov a,\sminus\ov d}(3,3)=\alpha^{\ov a,\sminus\ov d}(3,3)+\beta^{\ov a,\sminus\ov d}(3,3)=a_3(a_2a_1-d_1)-d_2a_1+a_2(-d_3),\\
&d=\alpha^{\ov a,\sminus\ov d}(3,3)\beta^{\ov a,\sminus\ov d}(3,3)-\alpha^{\ov a,\sminus\ov d}(2,3)\beta^{\ov a,\sminus\ov d}(4,3) = d_1d_2d_3.
}
Then, through the divide-and-conquer algorithm (\ref{divideandconquer}) we find
\[U_6(x,y)=((x^2-2y)x-xy)(x^2-y), U_5(x,y)=(x^2-y)(x^2-2y)-y^2\]
evaluated in $x=\pi^{\ov a,\sminus\ov d}(3,3)$, $y=d$ to get $U(6), U(5)$ respectively.
Lastly we compute $\alpha^{\ov a,\sminus\ov d}(4,3)=\alpha^{\ov a,\sminus\ov d}(3,3)\alpha^{\ov a,\sminus\ov d}(1,3)+\alpha^{\ov a,\sminus\ov d}(2,3)\beta^{\ov a,\sminus\ov d}(2,3)$.
Then
\[\det(T^3_{19}((\ov a,\ov b,\ov c)) = U(6)\alpha^{\ov a,\sminus\ov d}(4,3) -dU(5)\alpha^{\ov a,\sminus\ov d}(1,3).\]
Similarly, to find the characteristic polynomial we compute
\eq{
&\alpha^{\ov x-\ov a,\sminus\ov d}(0,3)=1, \, \alpha^{\ov x-\ov a,\sminus\ov d}(1,3)=x-a_1, \, \alpha^{\ov x-\ov a,\sminus\ov d}(2,3)=(x-a_2)(x-a_1)-d_1,\\
&\alpha^{\ov x-\ov a,\sminus\ov d}(3,3)=(x-a_3)((x-a_2)(x-a_1)-d_1)-d_2(x-a_1),\\
&\beta^{\ov x-\ov a,\sminus\ov d}(1,3)=0, \, \beta^{\ov x-\ov a,\sminus\ov d}(2,3)=-d_3, \, \beta^{\ov x-\ov a,\sminus\ov d}(3,3)=-(x-a_2)d_3,\\
&\beta^{\ov x-\ov a,\sminus\ov d}(4,3)=-(x-a_3)(x-a_2)d_3+d_2d_3,\\
&\pi^{\ov x-\ov a,\sminus\ov d}(3,3)=(x-a_3)((x-a_2)(x-a_1)-d_1)-d_2(x-a_1)-(x-a_2)d_3,
}
we evaluate $U_6(x,y), U_5(x,y)$ in $x=\pi^{\ov x-\ov a,\sminus\ov d}(3,3), y=d$ to respectively get $U^x(6),U^x(5)$, and we compute\\
$\alpha^{\ov x-\ov a,\sminus\ov d}(4,3)=\alpha^{\ov x-\ov a,\sminus\ov d}(3,3)\alpha^{\ov x-\ov a,\sminus\ov d}(1,3)+\alpha^{\ov x-\ov a,\sminus\ov d}(2,3)\beta^{\ov x-\ov a,\sminus\ov d}(2,3)$ to get
\[p_{T^3_{19}(\ov a,\ov b,\ov c)}(x)=U^x(6)\alpha^{\ov x-\ov a,\sminus\ov d}(4,3) -dU^x(5)\alpha^{\ov x-\ov a,\sminus\ov d}(1,3).\]
If $\lambda$ is an eigenvalue of $T^3_{19}(\ov a,\ov b,\ov c)$ then we find a potential eigenvector associated to $\lambda$ with algorithm EIG (\ref{algorithmEIG}): We compute the following determinants by their recurrence relation:
\eq{
&D^{\ov\lambda -\ov a,\ov d}(0,k)=1, \, D^{\ov\lambda -\ov a,\ov d}(0,k)=\lambda-a_1, \, D^{\ov\lambda -\ov a,\ov d}(2,k)=(\lambda-a_2)(\lambda-a_1)-d_1\\
&D^{\ov\lambda -\ov a,\ov d}(3,k)=(\lambda-a_3)((\lambda-a_2)(\lambda-a_1)-d_1)-d_2(\lambda-a_1)\\
&\vdots\\
&D^{\ov\lambda -\ov a,\ov d}(19,k)=(\lambda-a_1)D^{\ov\lambda -\ov a}(18,k)-d_3D^{\ov\lambda -\ov a}(17,k).
}
We consider all products $b_{18}=b_3, b_{18}b_{17}=b_3b_2,\ldots, b_{18}\cdots b_1=(b_3b_2b_1)^6$. We fix an element $0\neq z\in\Ann(p_{T^3_{19}(\ov a,\ov b,\ov c)}(\lambda))$.  Then, if nonzero, an eigenvector associated to $\lambda$ is
\[v^z(\lambda):=z\cdot((b_3b_2b_1)^6D^{\ov\lambda -\ov a,\ov d}(0,k),(b_3b_2)^6b_1^5D^{\ov\lambda -\ov a,\ov d}(1,k),\ldots,D^{\ov\lambda -\ov a,\ov d}(18,k)).\]
Now we compute entry $a_{5,11}$ of the inverse by applying algorithm ENTRY (\ref{algorithmENTRY}): We find $D^{\ov a,\ov d}(4,3)=U(1)\alpha^{\ov a,\sminus\ov d}(4,3)-dU(0)\alpha^{\ov a,\sminus\ov d}(1,3)=\alpha^{\ov a,\sminus\ov d}(4,3)$. With a new iteration of algorithm D1 we find $D_{11}^{\ov a,\ov d}(8,3)=D^{\sigma_2(\ov a),\sigma_2(\ov d)}(8,3)=D_2^{\ov a,\ov d}(8,3)$: we need to compute the new shifted continuant polynomials
\[\alpha_2^{\ov a,\sminus\ov d}(1,3),\ldots,\alpha_2^{\ov a,\sminus\ov d}(3,3),\beta_2^{\ov a,\sminus\ov d}(1,3),\ldots,\beta_2^{\ov a,\sminus\ov d}(3,3)\]
to get $\alpha_2^{\ov a,\sminus\ov d}(5,3)$, $\pi_2^{\ov a,\sminus\ov d}(3,3)$ and $U(2)_2=U_2(\pi_2^{\ov a,\sminus\ov d}(3,3),d)$; then
\[D_{11}^{\ov a,\ov d}(8,3)=U(2)_2\alpha_2^{\ov a,\sminus\ov d}(5,3)-d\alpha_2^{\ov a,\sminus\ov d}(2,3).\]
We compute $b_5\cdots b_{11}=(b_1b_2b_3)^2b_2$. Then
\[a_{5,11}=(b_1b_2b_3)^2b_2\frac{D^{\ov a,\ov d}(4,3)D_{11}^{\ov a,\ov d}(8,3)}{\det(T^3_{19}((\ov a,\ov b,\ov c)))}.\]
\end{exam}

\newpage
\begin{exam}[\n{Example over $\Z/60\Z$}] We particularize Example \ref{universalexample} with $K=R:=\Z/60\Z$, $\ov a:=(1,2,3)$, $\ov b:=(1,-1,1)$, $\ov c:=(12,7,1)\in K^3$. We just need to evaluate the corresponding values of $a_1,\ldots,c_3$ in the previous computations.\\
\n{Determinant:}
\eq{
&d_1=b_1c_1=12, d_2=-7, d_3=1,\\
&\alpha^{\ov a,\sminus\ov d}(0,3)=1, \, \alpha^{\ov a,\sminus\ov d}(1,3)=1, \, \alpha^{\ov a,\sminus\ov d}(2,3)=2\cdot1-12=-10,\\
&\alpha^{\ov a,\sminus\ov d}(3,3)=3(-10)+7\cdot 1 = -23,\\
&\beta^{\ov a,\sminus\ov d}(1,3)=0, \, \beta^{\ov a,\sminus\ov d}(2,3)=-1, \, \beta^{\ov a,\sminus\ov d}(3,3)=2(-1)=-2,\\
&\beta^{\ov a,\sminus\ov d}(4,3)=3(-2)+(-7)1 = -13,\\
&\alpha^{\ov a,\sminus\ov d}(4,3)=(-23)1+(-10)(-1) = -13,\\
&\pi^{\ov a,\sminus\ov d}(3,3)=-23-2=-25, \, d=(-23)(-2)-(-10)(-13) = -24,\\
&U(6)=(((-25)^2-2(-24))(-25)-(-25)(-24))((-25)^2-(-24)) = -25,\\
&U(5)=((-25)^2-(-24))((-25)^2-2(-24))-(-24)^2 = 1.
}
\[\det(T^3_{19}((\ov a,\ov b,\ov c)) = (-25)(-13) -(-24)1\cdot1 = -11.\]
\n{Characteristic polynomial:}
\eq{
&\alpha^{\ov x-\ov a,\sminus\ov d}(0,3)=1, \, \alpha^{\ov x-\ov a,\sminus\ov d}(1,3)=x-1, \, \alpha^{\ov x-\ov a,\sminus\ov d}(2,3)=(x-2)\cdot(x-1)-12,\\
&\alpha^{\ov x-\ov a,\sminus\ov d}(3,3)=(x-3)((x-2)(x-1)-12)+7(x-1),\\
&\beta^{\ov x-\ov a,\sminus\ov d}(1,3)=0, \, \beta^{\ov x-\ov a,\sminus\ov d}(2,3)=-1, \, \beta^{\ov x-\ov a,\sminus\ov d}(3,3)=(x-2)(-1)=-(x-2),\\
&\beta^{\ov x-\ov a,\sminus\ov d}(4,3)=-(x-3)(x-2)+(-7)1 =-(x-3)(x-2)-7,\\
&\alpha^{\ov x-\ov a,\sminus\ov d}(4,3)=((x-3)((x-2)(x-1)-12)+7(x-1))(x-1)-((x-2)(x-1)-12),\\
&\pi^{\ov x-\ov a,\sminus\ov d}(3,3)=(x-3)((x-2)(x-1)-12)+7(x-1)-(x-2),\\
&U^x(6)=(((\pi^{\ov x-\ov a,\sminus\ov d}(3,3))^2-2(-24))(\pi^{\ov x-\ov a,\sminus\ov d}(3,3))-\\
&\,\,\,\,\,\,\,\,\,\,\,\,\,\,\,\, -(\pi^{\ov x-\ov a,\sminus\ov d}(3,3))(-24))((\pi^{\ov x-\ov a,\sminus\ov d}(3,3))^2-(-24)),\\
&U^x(5)=((\pi^{\ov x-\ov a,\sminus\ov d}(3,3))^2-(-24))((\pi^{\ov x-\ov a,\sminus\ov d}(3,3))^2-2(-24))-(-24)^2.
}
\eq{
&p_{T^3_{19}(\ov a,\ov b,\ov c)}(x)=U^x(6)\alpha^{\ov x-\ov a,\sminus\ov d}(4,3)+24U^x(5)(x-1) =\\
=&x^{19} + 23x^{18} + 6x^{17} + 57x^{15} + 39x^{14} + 37x^{13} + 29x^{12} + 15x^{11} + 53x^{10} +  \\
+&52x^9 +54x^8 + 22x^7 + 50x^6 + 3x^5 + 49x^4 + 41x^3 + 39x^2 + 19x + 11.
}
We have written the characteristic polynomial in expanded form for ease of reading, but the compact form actually computed by the formulas is more efficient for evaluation.
Observe that $\lambda:=1$ is a zero of $p_{T^3_{19}(\ov a,\ov b,\ov c)}$, so it is an eigenvalue of $T^3_{19}(\ov a,\ov b,\ov c)$. We compute an associated eigenvector:
\eq{
&D^{\ov\lambda -\ov a,\ov d}(0,k)=1, \, D^{\ov\lambda -\ov a,\ov d}(0,k)=0, \, D^{\ov\lambda -\ov a,\ov d}(2,k)=(-1)0-12=-12,\\
&D^{\ov\lambda -\ov a,\ov d}(3,k)=(-2)(-12)-(-7)0 = 24, \, D^{\ov\lambda -\ov a,\ov d}(4,k)=12, \, D^{\ov\lambda -\ov a,\ov d}(5,k)=0, \\
&D^{\ov\lambda -\ov a,\ov d}(6,k)=24, \, D^{\ov\lambda -\ov a,\ov d}(7,k)=0, \, D^{\ov\lambda -\ov a,\ov d}(8,k)=12, \, D^{\ov\lambda -\ov a,\ov d}(9,k)=-24,\\ &D^{\ov\lambda -\ov a,\ov d}(10,k)=-12, \, D^{\ov\lambda -\ov a,\ov d}(11,k)=0, \, D^{\ov\lambda -\ov a,\ov d}(12,k)=-24, \, D^{\ov\lambda -\ov a,\ov d}(13,k)=0,\\
&D^{\ov\lambda -\ov a,\ov d}(14,k)=-12, \, D^{\ov\lambda -\ov a,\ov d}(15,k)=24, \, D^{\ov\lambda -\ov a,\ov d}(16,k)=12, \, D^{\ov\lambda -\ov a,\ov d}(17,k)=0,\\
&D^{\ov\lambda -\ov a,\ov d}(18,k)=24.
}
Since $b_1=1, b_2=-1, b_3=1$, the products $\prod_{j=i}^{n-1}b_j$ for $1\leq i\leq n$ are easy to compute in this case, being $1,1,\sminus1,\sminus1,\sminus1,1,1,1,\sminus1,\sminus1,\sminus1,1,1,1,\sminus1,\sminus1,\sminus1,1,1$. Since $\Ann(0)=K$ we can pick $z:=1$. Thus
\[v^1(1)=(1, 0, 12, -24, -12, 0, 24, 0, -12, 24, 12, 0, -24, 0, 12, -24, -12, 0, 24).\]
Interestingly, $p_{T^3_{19}(\ov a,\ov b,\ov c)}(\lambda)$ is a zero divisor of $K$ for many other values of $\lambda$,\footnote{The only elements of $K$ which are not eigenvalues of $T^3_{19}(\ov a,\ov b,\ov c)$ are $0,2,8,12,14,18,20,24,30,32,38,42,44,48,50,54$.} but all the associated eigenvalues found by Theorem \ref{eigenvectors} are multiples of $v^1(1)$.\\
\smallskip
\pagebreak

\n{Entry $a_{5,11}$ of the inverse:} We use that $f_2^{\ov a,\ov d}=f(\sigma_2(\ov a),\sigma_2(\ov d))$ for any $f\in K[x_1,\ldots,y_k]$:
\eq{
&(b_1b_2b_3)^2b_2 = -1, \, D^{\ov a,\ov d}(4,3)=\alpha^{\ov a,\ov d}(4,3) = -13, \, \det(T^3_{19}((\ov a,\ov b,\ov c)) = -11,\\
&\ov a':=\sigma_2(\ov a)=(3,1,2), \, \ov d':=\sigma_2(\ov d)=(1,12,-7), \\
&\alpha^{\ov a',\sminus\ov d'}(1,3)=3, \, \alpha^{\ov a',\sminus\ov d'}(2,3)=1\cdot3-1=2, \, \alpha^{\ov a',\sminus\ov d'}(3,3)=2\cdot2-12\cdot 3 = 28,\\
&\beta^{\ov a',\sminus\ov d'}(1,3)=0, \, \beta^{\ov a',\sminus\ov d'}(2,3)=7, \, \beta^{\ov a',\sminus\ov d'}(3,3)=1\cdot7=7,\\
&\alpha^{\ov a',\sminus\ov d'}(4,3)=28\cdot3+2\cdot7 = -22, \, \pi^{\ov a',\sminus\ov d'}(3,3)=28+7=-25, \, U(2)_2 = -25,\\
&D_{11}^{\ov a,\ov d}(8,3)=(-25)(-22)-(-24)3 = 22.
}
\[a_{5,11}=(-1)\frac{(-13)22}{-11} = -26.\]
\end{exam}

\bigskip
\bigskip
\bigskip

\bibliographystyle{abbrv}
\bibliography{mybibfile}

\begin{thebibliography}{10}

\bibitem{AlbertoBrox}
J.~Alberto and J.~Brox.
\newblock Inverses of $k$-{T}oeplitz matrices with applications to resonator
  arrays with multiple receivers.
\newblock {\em Appl. Math. Comput.}, 377(125185):14 pp., 2020.

\bibitem{albertofast}
J.~Alberto, U.~Reggiani, L.~Sandrolini, and H.~Albuquerque.
\newblock Fast calculation and analysis of the equivalent impedance of a
  wireless power transfer system using an array of magnetically coupled
  resonators.
\newblock {\em PIER B}, 80:101--112, 2018.

\bibitem{MATRIX}
J.~Alberto, U.~Reggiani, L.~Sandrolini, and H.~Albuquerque.
\newblock Accurate calculation of the power transfer and efficiency in
  resonator arrays for inductive power transfer.
\newblock {\em PIER B}, 83:61--76, 2019.

\bibitem{AlvarezNodarsePetronilhoQuintero}
R.~\'Alvarez-Nodarse, J.~Petronilho, and N.~R. Quintero.
\newblock On some tridiagonal $k$-{T}oeplitz matrices: algebraic and analytical
  aspects. {A}pplications.
\newblock {\em J. Computat. Appl. Math.}, 184(2):518--537, 2005.

\bibitem{AndelicFonseca2020}
M.~An{\dj}eli\'{c}, Z.~Du, C.~da~Fonseca, and E.~Kılı\c{c}.
\newblock A matrix approach to some second-order difference equations with
  sign-alternating coefficients.
\newblock {\em J. Difference Equ. Appl.}, 26(2):149--162, 2020.

\bibitem{BurgisserClausenShokrollahi1997}
P.~Bürgisser, M.~Clausen, and M.~A. Shokrollahi.
\newblock {\em Algebraic {C}omplexity {T}heory. With the collaboration of
  Thomas Lickteig}, volume 315 of {\em Grundlehren der mathematischen
  Wissenschaften (Fundamental Principles of Mathematical Sciences)}.
\newblock Springer-Verlag, Berlin, 1997.

\bibitem{Brown1993}
W.~C. Brown.
\newblock {\em Matrices over Commutative Rings}, volume 169 of {\em Monographs
  and Textbooks in Pure and Applied Mathematics}.
\newblock Marcel Dekker, Inc., New York, 1993.

\bibitem{Chandler-WildeGover1989}
S.~N. Chandler-Wilde and M.~J.~C. Gover.
\newblock On the application of a generalization of {T}oeplitz matrices to the
  numerical solution of integral equations with weakly singular convolution
  kernels.
\newblock {\em IMA J. Numer. Anal.}, 9(4):525--544, 1989.

\bibitem{Chandler-WildeGoverHothersah1986}
S.~N. Chandler-Wilde, M.~J.~C. Gover, and D.~C. Hothersah.
\newblock Sound propagation over inhomogeneous boundaries.
\newblock {\em In: Inter-Noise, Cambridge, Mass.}, 86:377--382, 1986.

\bibitem{Cox}
D.~A. Cox.
\newblock {\em Galois theory (second edition)}, volume 196 of {\em Pure and
  Applied Mathematics (Hoboken)}.
\newblock John Wiley \& Sons, Inc., Hoboken, NJ, 2012.

\bibitem{FonsecaKowalenkoLosonczi2020}
C.~M. da~Fonseca, V.~Kowalenko, and L.~Losonczi.
\newblock Ninety years of $k$-tridiagonal matrices.
\newblock {\em Studia Sci. Math. Hungar.}, 57(3):298--311, 2020.

\bibitem{FonsecaPetronilho2001}
C.~M. da~Fonseca and J.~Petronilho.
\newblock Explicit inverses of some tridiagonal matrices.
\newblock {\em Linear Algebra Appl.}, 325(1-3):7--21, 2001.

\bibitem{FonsecaPetronilho2005}
C.~M. da~Fonseca and J.~Petronilho.
\newblock Explicit inverse of a tridiagonal $k$-{T}oeplitz matrix.
\newblock {\em Numer. Math.}, 100(3):457--482, 2005.

\bibitem{EgervarySzasz1928}
E.~Egerváry and O.~Szász.
\newblock Einige {E}xtremalprobleme im {B}ereiche der trigonometrischen
  {P}olynome.
\newblock {\em Math. Z.}, 27(1):641--652, 1928.

\bibitem{ElsnerRedheffer1967}
L.~Elsner and R.~Redheffer.
\newblock Remarks on band matrices.
\newblock {\em Numer. Math.}, 10:153–161, 1967.

\bibitem{FischerUsmani1969}
C.~F. Fischer and R.~A. Usmani.
\newblock Properties of some tridiagonal matrices and their application to
  boundary value problems.
\newblock {\em SIAM J. Numer. Anal.}, 6:127--142, 1969.

\bibitem{Gover1994}
M.~J.~C. Gover.
\newblock The eigenproblem of a tridiagonal 2-{T}oeplitz matrix.
\newblock {\em Linear Algebra Appl.}, 197/198:63--78, 1994.

\bibitem{JoyeQuisquater}
M.~Joye and J.-J. Quisquater.
\newblock Efficient computation of full {L}ucas sequences.
\newblock {\em Electronics Letters}, 32(6):537--538, 1996.

\bibitem{Koepf}
W.~Koepf.
\newblock Efficient computation of {C}hebyshev polynomials.
\newblock {\em In: Computer Algebra Systems: A Practical Guide,}, Wiley, New
  York:79--99, 1999.

\bibitem{Lanczos1950}
C.~Lanczos.
\newblock An iteration method for the solution of the eigenvalue problem of
  linear differential and integral operators.
\newblock {\em J. Research Nat. Bur. Standards}, 45(4):255–282, 1950.

\bibitem{Lewis1982}
J.~Lewis.
\newblock Inversion of tridiagonal matrices.
\newblock {\em Numer. Math.}, 38:333–345, 1982.

\bibitem{LuSun}
L.~Lu and W.~Sun.
\newblock The minimal eigenvalues of a class of block-tridiagonal matrices.
\newblock {\em IEEE Trans. Inform. Theory}, 43(2):787–791, 1997.

\bibitem{Lucas}
E.~Lucas.
\newblock Théorie des fonctions numériques simplement périodiques.
\newblock {\em Amer. J. Math.}, 1(4):289–321, 1878.

\bibitem{Mallik2001}
R.~K. Mallik.
\newblock The inverse of a tridiagonal matrix.
\newblock {\em Linear Algebra Appl.}, 325(1-3):109--139, 2001.

\bibitem{MarcellanPetronilho1997}
F.~Marcell\'an and J.~Petronilho.
\newblock Eigenproblems for tridiagonal $2$-{T}oeplitz matrices and quadratic
  polynomial mappings.
\newblock {\em Linear Algebra Appl.}, 260:169--208, 1997.

\bibitem{MarcellanPetronilho2000}
F.~Marcell\'an and J.~Petronilho.
\newblock Orthogonal polynomials and cubic polynomial mappings ({I}).
\newblock {\em Commun. Anal. Theory Contin. Fract.}, 8:88--116, 2000.

\bibitem{MattheijSmooke1986}
R.~M.~M. Mattheij and M.~D. Smooke.
\newblock Estimates for the inverse of tridiagonal matrices arising in
  boundary-value problems.
\newblock {\em Linear Algebra Appl.}, 73:33--57, 1986.

\bibitem{Meurant1992}
G.~Meurant.
\newblock A review on the inverse of symmetric tridiagonal and block
  tridiagonal matrices.
\newblock {\em SIAM J. Matrix Anal. Appl.}, 13(3):707--728, 1992.

\bibitem{Muir1882}
T.~Muir.
\newblock {\em A Treatise on the Theory of Determinants}.
\newblock McMillan and Co., London, 1882.

\bibitem{MuirMetzler1933}
T.~Muir and W.~H. Metzler.
\newblock {\em A Treatise on the Theory of Determinants. Second edition revised
  and enlarged by W.H. Metzler}.
\newblock Longmans, Green and Company. Republished and corrected by Dover
  Publications, Inc., New York 1960, 1933.

\bibitem{Ostrowski1954}
A.~Ostrowski.
\newblock On two problems in abstract algebra connected with {H}orner's rule.
\newblock {\em In: Studies presented to R. von Mises,}, Academic Press, New
  York:40--48, 1954.

\bibitem{Ribenboim}
P.~Ribenboim.
\newblock {\em My numbers, my friends}.
\newblock Popular Lectures on Number Theory. Springer-Verlag, New York, 2000.

\bibitem{Rozsa1969}
P.~Rózsa.
\newblock On periodic continuants.
\newblock {\em Linear Algebra Appl.}, 2:267--274, 1969.

\bibitem{Winograd1971}
S.~Winograd.
\newblock On the algebraic complexity of functions.
\newblock {\em In: Actes du Congrès International des Mathématiciens (Nice,
  1970), Gauthier-Villars, Paris,}, Tome 3:283--288, 1971.

\bibitem{Wittenburg1998}
J.~Wittenburg.
\newblock Inverses of tridiagonal {T}oeplitz and periodic matrices with
  applications to mechanics.
\newblock {\em J. Appl. Math. Mech.}, 62(4):575--587, 1998.

\bibitem{Yamamoto2001}
T.~Yamamoto.
\newblock Inversion formulas for tridiagonal matrices with applications to
  boundary value problems.
\newblock {\em Numer. Funct. Anal. Optim.}, 22(3-4):357–385, 2001.

\end{thebibliography}

\end{document}